\documentclass{amsart}

\usepackage{amsmath}
\usepackage{amssymb,amsthm}
\usepackage[all]{xy}
\usepackage{graphicx}
\usepackage{mathrsfs}
\usepackage{stmaryrd}
\usepackage[active]{srcltx}
\usepackage{euscript}

\usepackage{mathabx,epsfig}

\RequirePackage{ifpdf}
\ifpdf
  \IfFileExists{pdfsync.sty}{\RequirePackage{pdfsync}}{}
  \RequirePackage[pdftex, 
   colorlinks = true,
   pagebackref,
   bookmarksopen=true]{hyperref}
\else
  \RequirePackage[hypertex]{hyperref}
\fi

\usepackage{tikz}
\tikzset{
	MyPersp/.style={scale=2,x={(0.8cm,0cm)},y={(0cm,0.25cm)},
    z={(0cm,1cm)}},
	MyPoints/.style={fill=white,draw=black,thick}
		}

\usetikzlibrary{decorations.pathmorphing,arrows,calc}

 \definecolor{darkgreen}{HTML}{336633}
 \definecolor{darkred}{HTML}{993333}

\RequirePackage{color}
\definecolor{myred}{rgb}{0.75,0,0}
\definecolor{mygreen}{rgb}{0,0.5,0}
\definecolor{myblue}{rgb}{0,0,0.65}

\newcommand{\C}{\mathbb{C}}
\newcommand{\R}{\mathbb{R}}

\newcommand{\F}{\mathbb{F}}

\newcommand{\bk}{\Bbbk}

\newcommand{\Z}{\mathbb{Z}}

\DeclareMathOperator{\Spec}{Spec}
\newcommand{\pt}{\mathrm{pt}}

\mathchardef\mhyphen="2D

\newcommand{\scA}{\mathscr{A}}

\newcommand{\Mod}{\mhyphen\mathsf{Mod}}
\newcommand{\Mof}{\mhyphen\mathsf{Mof}}
\newcommand{\perf}{\mhyphen\mathsf{perf}}

\newcommand{\Tilt}{\mathsf{Tilt}}
\newcommand{\Db}{D^{\mathrm{b}}}
\newcommand{\Kb}{K^{\mathrm{b}}}
\newcommand{\Parity}{\mathsf{Parity}}

\newcommand{\bX}{\mathbf{X}}

\newcommand{\Gr}{{\EuScript Gr}}
\newcommand{\Fl}{{\EuScript Fl}}

\newcommand{\Perv}{\mathsf{Perv}}

\newcommand{\IC}{\mathscr{I}\hspace{-1pt}\mathscr{C}}

\newcommand{\fR}{\mathfrak{R}}

\newcommand{\simto}{\xrightarrow{\sim}}

\newcommand{\End}{\mathrm{End}}

\makeatletter
\def\lotimes{\@ifnextchar_{\@lotimessub}{\@lotimesnosub}}
\def\@lotimessub_#1{\mathchoice{\mathbin{\mathop{\otimes}^L}_{#1}}%
  {\otimes^L_{#1}}{\otimes^L_{#1}}{\otimes^L_{#1}}}
\def\@lotimesnosub{\mathbin{\mathop{\otimes}^L}}
\makeatother

\newcommand{\For}{\mathsf{For}}

\newcommand{\id}{\mathrm{id}}

\newcommand{\scE}{\mathscr{E}}
\newcommand{\scF}{\mathscr{F}}
\newcommand{\scH}{\mathscr{H}}
\newcommand{\scL}{\mathscr{L}}
\newcommand{\scG}{\mathscr{G}}

\newcommand{\scT}{\mathscr{T}}

\DeclareMathOperator{\Hom}{Hom}
\DeclareMathOperator{\Ext}{Ext}

\DeclareMathOperator{\Res}{Res}

\newcommand{\aff}{\mathrm{aff}}

\newcommand{\excise}[1]{}

\newcommand{\Av}{\mathsf{Av}}

\newtheorem*{conj*}{Conjecture}
\newtheorem*{thm*}{Theorem}
\newtheorem*{cor*}{Corollary}
\numberwithin{equation}{section}
\newtheorem{thm}{Theorem}[section]
\newtheorem{lem}[thm]{Lemma}
\newtheorem{prop}[thm]{Proposition}
\newtheorem{cor}[thm]{Corollary}

\theoremstyle{definition}

\theoremstyle{remark}
\newtheorem{rmk}[thm]{Remark}

\newtheorem*{rmk*}{Remark}

\DeclareMathOperator{\Rep}{Rep}

\renewcommand{\a}{\alpha}
\newcommand{\Wf}{W_{\mathrm{f}}}

\newcommand{\Waff}{W_{\mathrm{aff}}}

\newcommand{\fW}{{}^{\mathrm{f}} W_{\mathrm{aff}}}

\newcommand{\Sf}{S_{\mathrm{f}}}
\newcommand{\Saff}{S_{\mathrm{aff}}}
\newcommand{\asph}{\mathrm{asph}}

\newcommand{\Haff}{\mathcal{H}_{\mathrm{aff}}}

\newcommand{\Masph}{\mathcal{M}^{\asph}}

\newcommand{\puN}{{}^{\ell} \hspace{-1pt} \underline{N}}

\newcommand{\ppn}{{}^{\ell} \hspace{-1pt} n}

\newcommand{\scK}{\mathscr{K}}
\newcommand{\scO}{\mathscr{O}}
\newcommand{\IW}{\mathcal{IW}}

\newcommand{\sph}{\mathrm{sph}}

\newcommand{\Iw}{\mathrm{Iw}}
\newcommand{\Iwu}{\mathrm{Iw}_{\mathrm{u}}}
\newcommand{\Iwuell}{\mathrm{Iw}_{\mathrm{u},\ell}}
\newcommand{\Iwun}{\mathrm{Iw}_{\mathrm{u},n}}
\newcommand{\Ga}{\mathbb{G}_{\mathrm{a}}}
\newcommand{\Gm}{\mathbb{G}_{\mathrm{m}}}
\newcommand{\st}{\mathsf{t}}

\newcommand{\LAS}{\mathscr{L}_{\mathrm{AS}}}
\newcommand{\Smith}{\mathsf{Sm}}

\newcommand{\coH}{\mathsf{H}}

\newcommand{\Sh}{\mathsf{Sh}}
\newcommand{\ubk}{\underline{\bk}}
\newcommand{\Satake}{\mathsf{S}}
\newcommand{\block}{\mathsf{b}}

\newcommand{\roots}{{\boldsymbol{\mu}}}
\newcommand{\bG}{\mathbf{G}}
\newcommand{\bB}{\mathbf{B}}
\newcommand{\bT}{\mathbf{T}}

\newcommand{\SmithParity}{\mathsf{SmParity}}

\title{Smith--Treumann theory and the linkage principle}

\author{Simon Riche}
\address{Universit\'e Clermont Auvergne, CNRS, LMBP, F-63000 Clermont-Ferrand, France.}
\email{simon.riche@uca.fr}

\author{Geordie Williamson}
\address{School of Mathematics and Statistics F07, University of
  Sydney NSW 2006, Australia. }
\email{g.williamson@sydney.edu.au}

\dedicatory{Dedicated to Roman Bezrukavnikov,\\in admiration.}

\begin{document}

\begin{abstract}
%In this paper we prove even more than we thought we could.
We apply Treumann's ``Smith theory for sheaves" in the context of the Iwahori--Whittaker model of the Satake category.
% obtained in recent work of Bezrukavnikov, Gaitsgory, Mirkovi\'c, Rider and the first author.
We deduce two results in the representation theory of reductive algebraic groups over fields of positive characteristic: (a) a geometric proof of the linkage principle; (b)
a character formula for tilting modules in terms of the $\ell$-canonical basis, valid in all blocks and in all characteristics.
\end{abstract}

\maketitle

%%%%%%%%%%%%%%%%%%%%%%%%%
\section{Introduction}
\label{sec:intro}
%%%%%%%%%%%%%%%%%%%%%%%%%

%---------------------------------------------------------------
\subsection{Geometric representation theory of reductive algebraic groups}
%---------------------------------------------------------------

Let $\bG$ be a connected reductive algebraic group over an algebraically
closed field $\bk$ of characteristic $\ell>0$, and consider
the category $\Rep(\bG)$ of finite-dimensional algebraic representations of $\bG$. The
study of this category has 
%led to 
seen significant progress 
%in modular representation theory 
over the last fifty years; however several
fundamental questions (e.g.~dimensions and characters of simple and indecomposable
tilting modules) remain only partially understood.

One tempting avenue of pursuit is to find relationships to $\mathscr{D}$-modules
or constructible sheaves, and hence bring sheaf theory into play. The
archetypal example of the success of such an approach is the
Be{\u\i}linson--Bernstein localization theorem, which establishes such an
link for modules over complex semi-simple Lie algebras. 
Be{\u\i}linson--Bernstein localization is an indispensable tool in modern
Geometric Representation Theory, leading to proofs of the
Kazhdan--Lusztig conjectures, character formulas for real reductive
groups, etc.

Back in the setting of $\Rep(\bG)$, the \emph{geometric Satake
  equivalence} provides such a connection to constructible sheaves. To
$\bG$ we can associate the affine Grassmannian $\Gr_H$ of its complex Langlands dual
group $H$ (an ``infinite-dimensional algebraic variety''), and one has an equivalence
of tensor categories between $\Rep(\bG)$ and a certain category of perverse sheaves with
$\bk$-coefficients on $\Gr_H$.
The geometric Satake equivalence is central
to modern approaches to the Langlands program, and has become a
cornerstone of Geometric Representation Theory.

However, in contrast to Be{\u\i}linson--Bernstein localization, the geometric
Satake equivalence has been surprisingly ineffective at answering questions about $\Rep(\bG)$.
For example, several basic statements and constructions involving $\Rep(\bG)$ (e.g.~the
linkage principle, or Frobenius twist) have no geometric
explanation. This is the more surprising, as several known or
conjectured formulas (e.g.~Lusztig's character formula) involve
Kazhdan--Lusztig polynomials or their $\ell$-counterparts, which
encode dimensions of stalks of complexes on the affine Grassmannian and flag
variety of $H$. Nowadays we have several proofs of Lusztig's character
formula for large $\ell$, however none of them pass through the
geometric Satake equivalence!

%---------------------------------------------------------------
\subsection{Overview}
%---------------------------------------------------------------

The main result of the present paper is a proof of the linkage principle via the geometric
Satake equivalence. Our proof also explains that each ``block'' in the linkage
principle is controlled by a partial affine flag variety for the Langlands dual
group, which gives us new proofs of Lusztig's conjecture on simple
characters (for large $\ell$) and of a conjecture of the authors on
tilting characters (for all $\ell$). 
%The techniques of this paper provide a powerful new tool in the study of representations of $\bG$.

Our argument is a simple application of two new tools in Geometric
Representation Theory. The first one is Smith--Treumann theory, which is a variant of equivariant
localization for tori. In this theory the circle action is
replaced by the action of a cyclic group of order $\ell$, and the
coefficients are of characteristic $\ell$. We apply this
theory to the loop rotation action on the affine Grassmannian. Whilst
the fixed points under the full loop rotation action (infinitely many partial flag varieties for $H$) are rather boring, the fixed points under the subgroup
of $\ell$-th roots of unity (\emph{affine} flag varieties for $H$,
which are finite in number if $H$ is semi-simple) are rich.

The second ingredient is the Iwahori--Whittaker realisation of the
Satake category. This replaces the category of perverse
sheaves in the Satake equivalence with an equivalent category of perverse sheaves satisfying a certain
equivariance condition with respect to the pro-unipotent radical of
the Iwahori subgroup. (This condition, introduced---to our knowledge---by Bezrukavnikov, is inspired by the ``Whittaker
conditions'' in the representation theory of $p$-adic  groups, hence the name.) It
turns out that in the Iwahori--Whittaker realisation, the components of
the fixed points discussed above match precisely the decomposition
of $\Rep(\bG)$ given by the linkage principle. Our main theorem asserts
that the Smith restriction functor gives an equivalence between
tilting sheaves in the Iwahori--Whittaker realisation and a certain
category of parity complexes on the fixed points. It is then
straightforward to deduce the linkage principle. The character
formulas for simple and tilting modules alluded to above are also an 
immediate consequence.

In the rest of the introduction, we give a more detailed overview of
the techniques and results of this paper.

%------------------------------------------------------------
\subsection{The linkage principle}
%------------------------------------------------------------

As above, let $\bG$ be a connected reductive algebraic group over an algebraically
closed field $\bk$ of characteristic $\ell>0$, and let $\Rep(\bG)$ be its
category of finite-dimensional algebraic representations. Fix a maximal torus and Borel
subgroup $\bT \subset \bB \subset \bG$, and let $\fR^+ \subset \fR \subset \bX$
denote the (positive) roots inside the lattice of characters of
$\bT$.\footnote{We warn the reader that in the body of the paper we
  switch to a Langlands dual notation.} The simple objects in $\Rep(\bG)$ are classified by dominant weights $\bX^+ \subset \bX$; given
$\lambda \in \bX^+$ we denote by $\nabla(\lambda)$ the induced $\bG$-module of highest weight $\lambda$, and by $\mathsf{L}(\lambda)$ its simple socle.
%\footnote{\color{red}Again, difference to notation in the text $L(\lambda)$ vs. $L_\lambda$. I prefer the later, as it takes up less space! Won't worry about this now \dots}

Let $\Wf$ denote the Weyl group of $(\bG,\bT)$, and 
%(with $\ell$ as above)
consider the affine Weyl group
\[
\Waff := \Wf \ltimes \Z \fR
\] 
which acts naturally on $\bX$. 
%(The group $W_\ell$ is a subgroup of the \emph{affine Weyl group} $W := W_f \rtimes \ell \Z \fR$.) 
The \emph{linkage principle} \cite{verma, humphreys, jantzen-linkage, andersen}
states that we have a decomposition
\begin{equation} \label{eq:linkage}
\Rep(\bG) = \bigoplus_{\gamma \in \bX / (\Waff, \bullet_\ell)} \Rep_\gamma(\bG),
\end{equation}
where each summand is the Serre subcategory
\begin{equation*}
\Rep_\gamma \bG = \langle \mathsf{L}(\lambda) : \lambda \in \gamma \cap \bX^+ \rangle.
\end{equation*}
Notice that we do not consider the standard
action of $\Waff$ on $\bX$, but rather the ``dot'' action (denoted
$\bullet_\ell$); that is, if $\rho := \frac{1}{2}
\sum_{\a \in \fR^+} \a$, then
\[
(w\st_\mu) \bullet_\ell \lambda := w(\lambda + \ell\mu + \rho) - \rho
\] 
for $w \in \Wf$, $\mu \in \Z\fR$ and $\lambda \in \bX$.

\begin{rmk}
The subcategory $\Rep_\gamma(\bG)$ will be called the \emph{block} of
$\gamma$. This is an abuse since this subcategory might be
decomposable, hence is not a ``block'' in the strict sense. The
terminology is convenient however. The blocks of $\Rep(\bG)$ in the strict sense have
been described by Donkin~\cite{donkin-blocks}. In a work in preparation,
E.~Zabeth uses the methods of the present paper to recover Donkin's result geometrically.
%We believe our methods should shed some light on this description, and hope to come back to this question in a future publication.
\end{rmk}

% The summands $\Rep_\gamma \bG$ are often indecomposable; a 
% result of Donkin \cite{donkin-blocks} refines \eqref{eq:linkage} to give the the block decomposition.
% The linkage principle is one of the most fundamental ingredients in our
% understanding of the representation theory of $\bG$. For example, it
% is via the linkage principle that Lusztig
% formulates his celebrated conjecture \cite{lusztig} on the characters of
% the simple modules $L_\lambda$ for large $\ell$, in terms of
% Kazhdan-Lusztig polynomials attached to the affine Weyl group.

% It is
% through the linkage principle that the affine Weyl group $W_\ell$
% comes to prominence. If $\ell$ is large enough, then Lusztig's
% character formula and the Steinberg tensor product formula give us a
% formula for the characters of the simple modules in 
% each summand in terms of Kazhdan-Lusztig polynomials.

\subsection{The geometric Satake equivalence} Let
$H$ be the complex\footnote{In a few paragraphs we will instead
  assume that $H$ is defined over a field of characteristic $p$ where
  $p \ne 0, \ell$.} reductive group which is Langlands dual to
$\bG$, and denote its maximal torus by $T$ (so that $X_*(T)=\bX=X^*(\bT)$).
Let $LH$ and $L^+H$ denote the ``loop'' group (ind-)schemes
whose $R$-points are $H(R( \hspace{-1pt} (z) \hspace{-1pt}))$ and $H(R[\hspace{-1pt}[z]\hspace{-1pt}])$ respectively, for any
$\C$-algebra $R$; and
let $\Gr_H := LH / L^+H$ denote the affine Grassmannian. The
affine Grassmannian is an ind-projective ind-scheme whose $T$-fixed
points (resp.~$L^+H$-orbits) are in a natural bijection with $\bX$ (resp. $\bX^+$), through a map denoted $\lambda \mapsto L_\lambda$ (resp.~$\mu \mapsto \Gr_H^\mu$). Here, $\Gr_H^\mu$ contains $L_\lambda$ iff $\lambda \in \Wf (\mu)$.

The \emph{geometric Satake equivalence} \cite{lusztig-weight, ginzburg, bd,mv}
gives an equivalence of Tannakian categories 
% \[
% \Satake: (\Perv_{\HO}(\Gr_H,\bk), \star^{\HO}) \cong (\Rep(G), \otimes).
% \]
\begin{equation} \label{eq:satake}
( \Perv_{L^+H}(\Gr_H,\bk), \star) \cong (\Rep(\bG), \otimes)
\end{equation}
where $\Perv_{L^+H}(\Gr_H,\bk)$ denotes the
category of $L^+H$-equivariant perverse sheaves on $\Gr_H$ with coefficients in $\bk$, with its natural convolution product $\star$.

% Since the discovery of the geometric Satake equivalence it has been a
% long-standing 
% open problem in geometric representation theory to explain the linkage principle via the geometric Satake
% equivalence.
% This is attractive, because one would like to use the geometric Satake
% equivalence as a bridge between representation theory and
% constructible sheaves, analogous to the Beilinson-Bernstein
% localization theorem for Lie algebra representations over
% $\C$. However, without a proof of the linkage principle on the geometric side, it is not
% clear how to use geometric Satake to say interesting things about $\Rep \bG$.

% In this paper we solve this problem, providing a geometric proof of the linkage principle. Two
% immediate consequences include:
% \begin{enumerate}
% \item a proof (for all $\ell$) of a conjecture
% of the authors on tilting characters for $\bG$.
% \item a direct geometric proof of Lusztig's character formula for large
% $\ell$.
% \end{enumerate}

%-------------------------------------------------------------------------
\subsection{Smith--Treumann theory}
%-------------------------------------------------------------------------

A fundamental role in our proof is played by Treumann's ``Smith theory
for sheaves''~\cite{treumann}. The basic idea is that, when
dealing with coefficients of characteristic $\ell$, one should be able
to localize to fixed points of actions of cyclic groups of order
$\ell$. (This theory should be compared with Borel's ``localization
theorem'' for manifolds equipped with an action of the circle group;
in this analogy, finite cyclic groups become ``discrete circles'';
see~\cite{williamson-cdm} for more comments on this analogy.)

More precisely, let $X$ be a variety endowed with an action of the group $\roots_\ell$ of
$\ell$-th roots of unity, and denote by $X^{\roots_\ell}$ the subvariety of $\roots_\ell$-fixed points. One has two (Verdier dual)
restriction functors 
\[
\begin{tikzpicture}[scale=0.7]
\node (l) at (0,0) {$\Db_{\roots_\ell}(X,\bk)$};
\node (r) at (4,0) {$\Db_{\roots_\ell}(X^{\roots_\ell},\bk)$};
\draw[->] (l) to[out=20,in=160] node[above] {$i^!$} (r);
\draw[->] (l) to[out=-20,in=-160] node[below] {$i^*$} (r);
\end{tikzpicture}
\]
between the $\roots_\ell$-equivariant derived categories of constructible
$\bk$-sheaves on $X$ and on $X^{\roots_\ell}$. 

A fundamental observation of Treumann is that the compositions of these
functors with the quotient functor to the \emph{Smith category}
\[
\Smith_{\mathrm{Treu}} (X^{\roots_\ell}, \bk) := \Db_{\roots_\ell}(X^{\roots_\ell}, \bk) /
\left \langle 
\text{$\roots_\ell$-perfect complexes}
\right \rangle
\]
become canonically isomorphic. Here an object in $\Db_{\roots_\ell}(X^{\roots_\ell}, \bk)$ is called \emph{$\roots_\ell$-perfect} if
its stalks (naturally complexes of $\roots_\ell$-modules) may be
represented by a bounded complex of free $\bk[\roots_\ell]$-modules.
The resulting \emph{Smith
  restriction} functor
\[
\begin{tikzpicture}[scale=0.7]
\node (l) at (0,0) {$\Db_{\roots_\ell}(X,\bk)$};
\node (r) at (4,0) {$\Smith_{\mathrm{Treu}}(X^{\roots_\ell},\bk)$};
\draw[->] (l) to node[above] {$i^{!*}$} (r);
\end{tikzpicture}
\]
has remarkable properties. For example, it commutes with essentially
all sheaf theoretic functors~\cite{treumann}. It can be thought
of as an analogue of hyperbolic localization~\cite{braden} for $\roots_\ell$-actions.

For technical reasons (most notably, to ensure that the Smith category of a point satisfies appropriate parity vanishing properties), we use a variant of Treumann's construction,
proposed by the second author in \cite{williamson-cdm}. Namely, we assume that the
action of $\roots_\ell$ can be extended to an action of the
multiplicative group $\Gm$ on $X$ and consider the \emph{equivariant
Smith category}
\[
\Smith(X^{\roots_\ell}, \bk) := \Db_{\Gm}(X^{\roots_\ell}, \bk) / \left \langle 
\begin{array}{c} \text{complexes whose restriction} \\ \text{to
  $\roots_\ell \subset \Gm$ are $\roots_\ell$-perfect} \end{array} \right \rangle.
\]
%For us, the main advantage of the equivariant Smith quotient over its non-equivariant quotient is that 
With this definition,
the theory of parity complexes from~\cite{jmw} applies in the Smith quotient, which will be crucial for our arguments.\footnote{The fact that Smith--Treumann theory can be made to accommodate the theory of parity sheaves was first pointed out by Leslie--Lonergan \cite{leslie-lonergan}. The version they use is however different, and---from our point of view---technically more involved.}

%------------------------------------------------------
\subsection{Fixed points} 
%------------------------------------------------------

To apply this idea in our setting, note that $\Gr_H$ has a
natural action of $\Gm$ via ``loop rotation,'' induced by the action
of $\Gm$ on $\C( \hspace{-1pt} (z) \hspace{-1pt})$ which ``rescales'' $z$. A beautiful
fact (that we first learned from R.~Bezrukavnikov) is that, if
$\roots_\ell \subset \Gm$ denotes (as above) the subgroup of $\ell$-th roots of unity, we have a decomposition
\begin{equation}
\label{eq:fixed points}
(\Gr_H)^{\roots_\ell} = \bigsqcup_{\gamma \in \bX / (\Waff, \square_\ell)} \Gr_{H,\gamma},
\end{equation}
where the action $\square_\ell$ of $\Waff$ on $\bX$ is defined by $(w\st_\mu) \square_\ell \lambda = w(\lambda+\ell\mu)$ for $w \in \Wf$, $\mu \in \Z\fR$ and $\lambda \in \bX$. Moreover,
each component on the right-hand side of~\eqref{eq:fixed points} is a partial affine flag
variety for the loop group
$L_\ell H$ representing $R \mapsto H(R( \hspace{-1pt}(z^\ell)\hspace{-1pt}))$, whose ``partiality'' is governed by the stabilizer of an element in $\gamma$. For example, for $\gamma = \Waff \square_\ell 0$ we obtain the ``thin affine 
Grassmannian'' (defined as above for $\Gr_H$, but now with $z$ replaced by $z^\ell$); and if $\gamma$ has trivial stabiliser under $\Waff$ then $\Gr_{H,\gamma}$ is the full
affine flag variety for $H$.

The similarity between \eqref{eq:fixed points} and \eqref{eq:linkage}
is rather striking; for example there are as many
components in the right-hand side of~\eqref{eq:fixed points} as summands in the decomposition~\eqref{eq:linkage}.  However there is a fundamental difference:~\eqref{eq:linkage} involves the dot action (with $\Wf$ fixing $-\rho$), whereas~\eqref{eq:fixed points} involves the
unshifted action (with $\Wf$ fixing $0$). Thus we
do not expect the Smith restriction functor to realise the linkage
principle in this setting.\footnote{The effect of Smith restriction in this setting is investigated in~\cite{leslie-lonergan}. The authors show that it realises the ``Frobenius contraction'' functor of Gros--Kaneda~\cite{gk}.}
%, i.e.~the adjoint of the Frobenius twist functor.

%---------------------------------------------------------------------
\subsection{The Iwahori-Whittaker model}
\label{ss:intro-IW}
%---------------------------------------------------------------------

To get around this issue, we
replace the ``traditional'' Satake category $\Perv_{L^+H}(\Gr_H,\bk)$
with the ``Iwahori--Whittaker mo\-del'' considered
in~\cite{bgmrr}. There it is proved that (under a mild assumption, satisfied e.g.~if $H$ is semisimple of adjoint type, which we assume from now on for simplicity) one has an equivalence of
abelian categories
\begin{equation}
  \label{eq:IW}
\Perv_{L^+H}(\Gr_H,\bk) \simto \Perv_{\IW}(\Gr_H,\bk)  
\end{equation}
where the right-hand side denotes a category of perverse sheaves of $\Gr_H$ which satisfy
a certain equivariance condition with respect to the pro-unipotent
radical $\Iwu^+$ of an Iwahori subgroup; such
perverse sheaves are called
``Iwahori--Whittaker.''\footnote{One can make sense of this condition in
  various ways. In this work (following Bezrukavnikov) we use \'etale
  sheaves and the Artin--Schreier covering, which necessitates passing
  to $\Gr_H$ defined over a field of characteristic $p > 0$ (with $p \neq \ell$). The
  geometric Satake equivalence is also available in this setting.
}

A crucial point for us is that on simple objects the equivalence
\eqref{eq:IW} sends the intersection cohomology complex associated with the $L^+ H$-orbit parametrized by $\lambda \in \bX^+$ (which corresponds to $\mathsf{L}(\lambda)$ under~\eqref{eq:satake}) to the Iwahori--Whittaker intersection cohomology complex associated with the $\Iwu^+$-orbit parametrized by $\lambda+\rho$.
%\footnote{\color{red} Is this our preferred notation for
%  $\IC$?}
%\[
%\IC_\lambda \mapsto \IC^{\IW}_{\lambda + \rho}.
%\]
Thus, after passage to the Iwahori--Whittaker model, our issue with the
two distinct actions of $\Waff$ goes away, and the linkage principle is reflected
perfectly in the geometry of the $\roots_{\ell}$-fixed points. In particular, two simple Iwahori--Whittaker perverse sheaves, parametrized by some weights $\lambda$ and $\mu$,
%$\IC^{\IW}_{\lambda}$ and $\IC^{\IW}_{\mu}$ 
lie in
the same summand in the linkage principle if and only if the
corresponding points $L_\lambda$ and $L_\mu$ lie in the same
component of the fixed points!

Another favorable property of the $\Iwu^+$-action on $\Gr_H$ is that
each orbit is isomorphic to an affine space. This setting is known to
imply nice properties for categories of perverse sheaves (see
e.g.~\cite{bgs}), and in particular that this category admits a
transparent structure of a highest weight category. The situation is even more favorable here in that the ``relevant'' orbits (i.e.~those which support a nonzero Iwahori--Whittaker local system) have dimensions of constant parity in each connected component of $\Gr_H$. This implies that the tilting objects in $\Perv_{\IW}(\Gr_H,\bk)$ are parity in the sense of~\cite{jmw}; in particular the indecomposable tilting perverse sheaves coincide with the self-dual indecomposable parity objects.

\subsection{Main theorems}
\label{ss:intro-main}
%-----------------------------------------------------------------

Recall that $\Gm$ acts on $\Gr_H$ via loop rotation. The Iwahori--Whittaker
condition and the loop rotation equivariance are compatible; we thus
obtain a Smith restriction functor
\[
i^{!*} : \Db_{\IW,\Gm}(\Gr_H, \bk) \to
\Smith_{\IW}((\Gr_H)^{\roots_\ell}, \bk).
\]
We will write $\Parity_{\IW, \Gm}(\Gr_H, \bk)$, resp.~$\SmithParity_{\IW}((\Gr_H)^{\roots_\ell}, \bk)$, for the
additive category of parity sheaves in the source, resp.~target, of this
functor, and $\mathsf{PerPar}_{\IW, \Gm}(\Gr_H, \bk)$ for the full subcategory of $\Parity_{\IW, \Gm}(\Gr_H, \bk)$ whose objects are the \emph{perverse} parity complexes. Our first main result (stated more precisely in Theorem~\ref{thm:i!*-parities}) is the following.

\begin{thm}
\label{thm:main1}
  Smith restriction yields a fully-faithful functor
\[
i^{!*} : \mathsf{PerPar}_{\IW, \Gm}(\Gr_H, \bk) \to
\SmithParity_{\IW}((\Gr_H)^{\roots_\ell}, \bk).
\]
\end{thm}

Remarkably, the proof of this theorem is a few lines once one has the
appropriate technology in place. It is an easy consequence of
Be{\u\i}linson's lemma, once one knows that $i^{!*}$ preserves standard and
costandard objects; this in turn follows because $i^{!*}$ commutes
with $*$- and $!$-extensions.

Recall from~\S\ref{ss:intro-IW} that the self-dual indecomposable Iwahori--Whittaker parity complexes on $\Gr_H$ coincide with the indecomposable tilting perverse sheaves in $\Perv_{\IW}(\Gr_H,\bk)$ (which in turn
correspond to tilting $\bG$-modules under the geometric Satake
equivalence). Given $\lambda \in \rho+\bX^+$, we denote
by $\scE_\lambda^\IW$ the corresponding indecomposable parity
complex.

By Theorem~\ref{thm:main1}, the image of any $\scE^\IW_\lambda$ under $i^{!*}$ has to be
supported on a single component. This has the following consequence,
from which one easily obtains the promised proof of the linkage principle (see~\S\ref{ss:linkage}).

\begin{cor}
If $\Hom(\scE^{\IW}_\lambda, \scE^{\IW}_\mu) \ne 0$ then $\Waff \square_\ell \lambda = \Waff \square_\ell \mu$.
\end{cor}

Theorem~\ref{thm:main1} implies more generally that many questions about tilting Iwahori--Whittaker perverse sheaves on $\Gr_H$ (and hence about tilting $\bG$-modules) may be answered after applying Smith restriction. This will be crucial to our second application in representation theory, to the tilting character formula (see~\S\ref{ss:intro-tilting-char} below). But, in order to apply this idea,
one needs a more ``concrete'' description of indecomposable parity objects in $\Smith_{\IW}((\Gr_H)^{\roots_\ell}, \bk)$.
%the Smith parity complexes on $(\Gr_H)^{\roots_\ell}$, which 
This is the subject of our second main result on the geometric side.

Consider the following diagram of quotient and forgetful functors:
\[
\begin{tikzpicture}[scale=0.5]
\node (bl) at (-8.5,0) {$\Smith_{\IW}((\Gr_H)^{\roots_\ell}, \bk)$};
\node (u) at (0,2) {$\Db_{\IW_\ell,\Gm}((\Gr_H)^{\roots_\ell}, \bk)$};
\node (br) at (8.5,0) {$\Db_{\IW_\ell}((\Gr_H)^{\roots_\ell}, \bk)$.};
\draw[<-] (bl) to node[above] {$\mathsf{Q}$} (u);
\draw[<-] (br) to node[above] {$\For_{\Gm}$} (u);
\end{tikzpicture}
\]
(Here the subscript in $\IW_\ell$ indicates that the Iwahori--Whittaker condition is imposed with respect to the action of an Iwahori subgroup in $L_\ell H$ now.)
Our second main theorem (which combines Proposition~\ref{prop:Hom-parity-Q} and Corollary~\ref{cor:Q-indec}) is the following.

\begin{thm}
\label{thm:main2}
The functors $\mathsf{Q}$ and $\For_{\Gm}$ preserve indecomposable parity complexes. Moreover,
if $\scE, \scF \in \Db_{\IW_\ell,\Gm}((\Gr_H)^{\roots_\ell}, \bk)$ are parity objects of the same parity, then there exists a canonical isomorphism
\begin{equation}
 \label{eq:Qhom}
 \Hom( \mathsf{Q}(\scE), \mathsf{Q}(\scF)) = \Hom^\bullet( \For_{\Gm}(\scE),\For_{\Gm}(\scF)).
\end{equation}
\end{thm}

The proof of Theorem~\ref{thm:main2} is given in~\S\ref{ss:comparison}.
The first step
is the observation that
given parity complexes $\scE, \scF \in \Db_{\IW_\ell,\Gm}((\Gr_H)^{\roots_\ell}, \bk)$ of the same parity we have
canonical isomorphisms:
\begin{align}
\label{eq:for}
  \Hom^\bullet( \For_{\Gm}(\scE), \For_{\Gm}(\scF)) &=  \Hom^\bullet(
  \scE, \scF) \otimes_{x \mapsto 0}  \bk.
\\
\label{eq:Q}
  \Hom( \mathsf{Q}(\scE), \mathsf{Q}(\scF)) &=  \Hom^\bullet(
  \scE, \scF) \otimes_{x \mapsto 1} \bk.
\end{align}
(Here, $x \in \mathsf{H}^2_{\Gm}(\pt,\bk)$ denotes the equivariant parameter, and $\Hom^\bullet(
  \scE, \scF) $ is viewed as an $\mathsf{H}^\bullet_{\Gm}(\pt, \bk)$-module in the
  natural way. The tensor products are taken over $\mathsf{H}^\bullet_{\Gm}(\pt, \bk)$, with the indicated module structure on $\bk$.)  The isomorphism~\eqref{eq:for} is simply the equivariant formality of
homomorphisms between parity complexes, which follows from a standard parity
argument. The isomorphism~\eqref{eq:Q} essentially follows from the analysis of
the Smith category of a point; see the proof of Proposition~\ref{prop:Hom-parity-Q} for details.

The preservation of indecomposable parity objects by the functor $\For_{\Gm}$ is immediate from~\eqref{eq:for}, and is a
basic ingredient in the theory of parity complexes. The statement for
$\mathsf{Q}$ is potentially more surprising, as inverting the equivariant
parameter rarely preserves indecomposability. The key point is that 
the $\Gm$-action on $(\Gr_H)^{\roots_\ell}$ factors through an action of
$\Gm/\roots_{\ell} = \Gm$. As a consequence, this action
is indistinguishable from the trivial action when we take
$\bk$-coefficients. Hence we obtain a canonical isomorphism
%\begin{align}
\[
\Hom^{\bullet}(\scE, \scF) = \mathsf{H}^{\bullet}_{\Gm}(\pt, \bk) \otimes_{\bk} \Hom^{\bullet}(\For_{\Gm}(\scE),
\For_{\Gm}(\scF)),
\]
% \label{eq:Qhom}
% \Hom( Q(\scE), Q(\scF)) &= \Hom^\bullet( \For(\scE),
% \For(\scF)).
% \end{align}
see Lemma~\ref{lem:Hom-parity-For-fixed-points}. Combining this isomorphism with~\eqref{eq:for}--\eqref{eq:Q} we deduce~\eqref{eq:Qhom}.
Now preservation of indecomposable parity objects by $\mathsf{Q}$ follows from the similar property for $\For_{\Gm}$, as a
finite-dimensional $\Z$-graded algebra is local if and only if its
degree-$0$ part is.

\begin{rmk}
  The isomorphism \eqref{eq:Qhom} shows that ``$\mathsf{Q} \circ (\For_{\Gm})^{-1}$''
  behaves like a degrading functor. Degrading functors are ubiquitous
  in modern Geometric Representation Theory. In algebra, they are
  often realised by forgetting the grading on a graded module; in geometry
  they are often associated with forgetting a mixed structure. The
  above shows that Smith--Treumann theory provides another possible
  geometric realisation of degrading functors.
\end{rmk}

%-----------------------------------------------------------
\subsection{Tilting characters} 
\label{ss:intro-tilting-char}
%-----------------------------------------------------------

A fundamental question in the
representation theory of $\bG$ is to determine the characters of the
indecomposable tilting modules. In~\cite{rw} we started advocating the
idea that character formulas for $\bG$-modules should be expressed in
terms of the $\ell$-canonical basis of $\Waff$, and illustrated this idea by a conjectural formula for
characters of indecomposable tilting modules in the principal block,
under the assumption that $\ell$ is bigger than the Coxeter number $h$ of $\bG$. This formula was proved in case $\bG=\mathrm{GL}_n(\bk)$ in~\cite{rw}, and then for a general reductive
group in a joint work with P.~Achar and S.~Makisumi, see~\cite{amrw}.
A simple consequence of the results of~\S\ref{ss:intro-main} is a new and much simpler proof of this character
formula, along with an extension to a formula valid in all blocks of
$\Rep(\bG)$, without any restriction of $\ell$. 

To state this formula, recall that the summands on the right-hand side
of~\eqref{eq:linkage} can be parametrized by the weights in
the intersection $\mathcal{A}$ of the weight lattice with the closure
of the fundamental alcove for the dot action of $\Waff$. For $\lambda \in
\mathcal{A}$ we denote by $W_\lambda \subset \Waff$ the stabilizer of $\lambda$ for
$\bullet_\ell$ (a standard parabolic subgroup), and by $\Waff^{(\lambda)}$ the subset of $\Waff$ consisting
of elements $w$ which are both maximal in $wW_\lambda$ and minimal in
$\Wf w$. Then the
indecomposable tilting $\bG$-modules in the block of $\Waff \bullet_\ell \lambda$ are
in a natural bijection with $\Waff^{(\lambda)}$, and we denote by
$\mathsf{T}(w \bullet_\ell \lambda)$ the module of highest weight $w
\bullet_\ell \lambda$.

The tilting character formula alluded to above can be stated as follows.

\begin{thm}
\label{thm:tilting}
For any $\lambda \in \mathcal{A}$ and $w \in \Waff^{(\lambda)}$
we have
\[
\mathsf{ch}(\mathsf{T}(w \bullet_\ell \lambda)) = \sum_{y \in \Waff^{(\lambda)}} \ppn_{y,w}(1) \cdot \chi_{y \bullet_\ell \lambda},
\]
where $\chi_\mu$ is the Weyl character formula attached to the
dominant weight $\mu$, and $\ppn_{y,w}$ is the antispherical
$\ell$-Kazhdan--Lusztig polynomial attached to $(y,w)$.  
\end{thm}

See~\cite[\S 1.4]{rw} for a comparison of this formula with an earlier formula conjectured by Andersen~\cite{andersen-filtrations}, which was one of our sources of inspiration. The proof of Theorem~\ref{thm:tilting} proceeds as follows. By standard considerations, to compute the characters of indecomposable tilting modules it suffices to compute the dimensions of morphism spaces between such objects. In the present setting, thanks to Theorem~\ref{thm:main1} this boils down to computing the dimensions of  morphism spaces between indecomposable parity objects in $\Smith_{\IW}((\Gr_H)^{\roots_\ell}, \bk)$. Formula~\eqref{eq:Qhom} allows to translate this question into that of computing dimensions of morphism spaces between indecomposable parity objects in $\Db_{\IW_\ell}((\Gr_H)^{\roots_\ell}, \bk)$, which can be done using~\eqref{eq:fixed points} and the fact that the dimensions of stalks of indecomposable parity complexes on (partial) affine flag varieties are computed by $\ell$-Kazhdan--Lusztig polynomials, as follows from the results of~\cite[Part~III]{rw}.

%-----------------------------------------------------------------------------
\subsection{Simple characters} 
%-----------------------------------------------------------------------------

Using ideas of
Andersen~\cite{andersen-filtrations} recently refined by
Soba\-je~\cite{sobaje}, from the formula in Theorem~\ref{thm:tilting} one can in theory deduce a character formula for simple $\bG$-modules, in all blocks and all characteristics. This can be done in at least two ways. The first possibility is to use a ``reciprocity formula'' due to Andersen~\cite{andersen2} (based on earlier work of Jantzen) which expresses multiplicities of simple modules in Weyl modules in terms of multiplicities of induced modules in indecomposable tilting modules. This method has the advantage of allowing to deduce Lusztig's conjectural formula~\cite{lusztig} in case the relevant $\ell$-Kazhdan--Lusztig polynomials coincide with the corresponding standard Kazhdan--Lusztig polynomials, but it requires the assumption that $\ell \geq 2h-2$, and does not produce a very natural formula in general, since it involves a certain ``twist'' (denoted $y \mapsto \hat{y}$ below) on indices.

To explain this in more detail, let us assume that $\ell \geq 2h-2$ and that $\bG$ is quasi-simple, and denote by $\alpha_0^\vee$ the highest coroot. We then set
\[
 Y := \{w \in \Waff \mid \text{$w$ is minimal in $\Wf w$ and $\langle w \square_\ell \rho, \alpha_0^\vee \rangle < \ell(h-1)$}\}.
\]
This subset does not depend on $\ell$, and is an ideal in the Bruhat order on the set of elements $w$ minimal in $\Wf w$; in fact, in terms of the notation of~\cite{rw-simple}, it consists of the elements $w$ sending the fundamental alcove $A_{\mathrm{fund}}$ inside the portion of the dominant region delimited by the hyperplane orthogonal to $\alpha_0^\vee$ and passing through $\rho$. Consider also the operation $y \mapsto \hat{y}$ on $\Waff$ corresponding to the operation $A \mapsto \hat{A}$ on alcoves considered in~\cite{soergel} or~\cite{rw-simple} (through the canonical bijection between $\Waff$ and the set of alcoves). Then, in view of~\cite[Proposition~1.8.1]{rw}, from Theorem~\ref{thm:tilting} we obtain that for any $w \in Y$ we have
\[
 [\nabla(w \bullet_\ell 0)] = \sum_{y \in Y} \ppn_{w,\hat{y}}(1) \cdot [\mathsf{L}(y \bullet_\ell 0)]
\]
in the Grothendieck group of $\Rep(\bG)$. In order to compare this formula with that in Lusztig's conjecture, one needs to invert these equations. In general we do not know how to do that explicitly. However, if we assume that each polynomial $\ppn_{w,\hat{y}}$ in these formulas coincides with the corresponding ``standard'' Kazhdan--Lusztig polynomial $n_{w,\hat{y}}$ (as considered e.g.~in~\cite{soergel}), then the inverse matrix is computed in~\cite[Theorem~5.1]{soergel}; from this result we obtain that
\[
 [\mathsf{L}(w \bullet_\ell 0)] = \sum_{y \in Y} (-1)^{\ell(w)+\ell(y)} h_{y,w}(1) \cdot [\nabla(y \bullet_\ell 0)]
\]
for any $w \in Y$, as predicted by Lusztig in~\cite{lusztig}. This property is well known to hold in large characteristic (without any known explicit bound), which explains why Theorem~\ref{thm:tilting} provides a new proof of Lusztig's conjecture in large characteristics.

\begin{rmk}
 The condition on $w$ considered above is not the same as in Lusztig's formulation of his conjecture. However, the two versions are known to be equivalent under the present assumptions, due to results of Kato; see~\cite[\S\S1.12--1.13]{williamson-takagi} for more details and references.
\end{rmk}

The other method to obtain a character formula for simple $\bG$-modules out of a character formula for indecomposable tilting $\bG$-modules, which works for all values of $\ell$ thanks to the results of~\cite{sobaje}, is to express multiplicities of the simple $\bG$-modules whose highest weight is restricted in the baby Verma $\bG_1 \bT$-modules. In this way one obtains a formula that may be compared with the ``periodic'' formulation of Lusztig's conjecture, see~\cite{lusztig-patterns}. This formula was made explicit in~\cite{rw-simple}, under the assumption that $\ell \geq 2h-1$. The extension of the tilting character formula in Theorem~\ref{thm:tilting} makes it desirable to extend the validity of these results to smaller values of $\ell$, and we plan to come back to this question in a future publication.

%-----------------------------------------------------------------
\subsection{Acknowledgements}
%-----------------------------------------------------------------

It is a pleasure and an honor to dedicate this paper to Roman Bez\-rukavnikov, who taught us so much about our subject. His inspiring influence on our works should be clear, and cannot be overestimated.

One of the main ideas of this paper crystalized whilst the second
author was preparing \cite{williamson-cdm}. He would like to thank
  the organizers for the invitation to deliver lectures as part of the ``Current
  Developments in Mathematics'' at Harvard, as well as their
  insistence that he write \cite{williamson-cdm}.
We thank Luc Illusie and Weizhe Zheng for their help with the proof of
Proposition~\ref{prop:illusie-zheng}, and Ofer Gabber for providing an
alternative proof. We thank Timo Richarz for indicating to us the
reference~\cite{dchl}, whose constructions play a crucial role in the
proof of Proposition~\ref{prop:fixed-points-Grass}, and spotting some
typos. We also thank Pramod Achar, Drago\c s Fr\u a\c til\u a, Jesper Grodal, Emilien Zabeth and
the anonymous referee for useful comments and pointing out
typos. Finally, we would like to thank the organizers and participants
of the Oberwolfach Arbeitsgemeinschaft ``Geometric Representation
Theory'', where the results of this paper were discussed in detail.

This project has received
funding from the European Research Council (ERC) under the European Union's Horizon 2020
research and innovation programme (grant agreements No~677147 and~01002592).

%-----------------------------------------------------------------
\subsection{Notation}
%-----------------------------------------------------------------

For an abelian category $\scA$, $D \scA$, $D^+\scA$ and $\Db \scA$ denote its
unbounded, bounded below and bounded derived category
respectively, and $\scH^n$ denotes the associated cohomology functors. If $A$ is an algebra, we will denote by $A\Mod$ the category of (left) $A$-modules. If $A$ is (left) Noetherian, we will denote by $A\Mof$ the subcategory of $A$-modules of finite type.

%%%%%%%%%%%%%%%%%%%%%%%%%
\section{Preliminaries on equivariant sheaves}
\label{sec:equiv-sheaves}
%%%%%%%%%%%%%%%%%%%%%%%%%

%-------------------------------------
\subsection{Equivariant sheaves}
%-------------------------------------

We start by recalling some generalities on \'etale shea\-ves on schemes endowed with an action of a finite group. We fix a (commutative) coefficient ring $\bk$.
% (not necessarily assumed to be commutative).

Let $X$ be a scheme, and let $A$ be a finite group acting on $X$ (by scheme automorphisms). For any $g \in A$, we denote by $\alpha_g : X \simto X$ the action on $X$. Recall that an $A$-equivariant (\'etale) sheaf of $\bk$-modules is the datum of an \'etale sheaf $\scF$ of $\bk$-modules on $X$ together with a collection $(\varphi_g)_{g \in A}$ where, for any $g \in A$,
\[
\varphi_g : \alpha_g^* \scF \simto \scF
\]
is an isomorphism of sheaves of $\bk$-modules, this collection satisfying the condition that for $g,h \in A$ we have
\begin{equation}
\label{eqn:cocycle}
 \varphi_h \circ \alpha_h^*(\varphi_g) = \varphi_{gh}
\end{equation}
as morphisms from $\alpha_{gh}^* \scF$ to $\scF$. (We will often abuse notation, and omit the isomorphisms $(\varphi_g)_{g \in A}$ from the notation.) Morphisms of $A$-equivariant sheaves are defined as morphisms of sheaves compatible (in the natural way) with the isomorphisms $\varphi_g$. The (abelian) category of $A$-equivariant sheaves of $\bk$-modules will be denoted $\Sh_A(X,\bk)$. We have a natural ``forgetful'' exact functor
\begin{equation}
\label{eqn:ForA}
 \For_A : \Sh_A(X,\bk) \to \Sh(X,\bk)
\end{equation}
(which simply forgets the collection of isomorphisms $(\varphi_g)_{g \in A}$), where $\Sh(X,\bk)$ denotes the category of sheaves of $\bk$-modules on $X$. If $A$ acts trivially on $X$, we have a canonical identification
\begin{equation}
\label{eqn:equ-sheaves-triv-action}
 \Sh_A(X,\bk) = \Sh(X,\bk[A]).
\end{equation}

If $X,Y$ are two schemes with actions of $A$ (with actions denoted $\alpha^X_{-}$ and $\alpha^Y_{-}$ respectively), and $f : X \to Y$ is an $A$-equivariant morphism, then for any $g \in A$ we have a canonical isomorphism
\[
 (\alpha^X_g)^* \circ f^* \cong f^* \circ (\alpha^Y_g)^*.
\]
As a consequence, the functor $f^*$ induces an exact functor
\[
 \Sh_A(Y,\bk) \to \Sh_A(X,\bk),
\]
which will also be denoted $f^*$. Similarly, we have a canonical isomorphism
%for any $n \in \Z$ we have
\[
 (\alpha^Y_g)^* \circ f_* \cong f_* \circ (\alpha^X_g)^*.
\]
(Here we use the fact that $(\alpha^Z_g)^* \cong (\alpha^Z_{g^{-1}})_*$ for $Z=X$ or $Y$.) As a consequence, $f_*$ induces a functor
\[
 f_* : \Sh_A(X,\bk) \to \Sh_A(Y,\bk).
\]

\subsection{Equivariant sheaves and injective resolutions}
\label{ss:equ-sheaves-inj}
%----------------------------------------------

We consider again a sche\-me $X$ endowed with an action of the finite group $A$.
For any $\scF$ in $\Sh(X,\bk)$ we set
\[
 \Av_A(\scF) := \bigoplus_{g \in A} \alpha_g^* \scF.
\]
We endow $\Av_A(\scF)$ with the structure of an $A$-equivariant sheaf by defining, for any $g \in A$, the isomorphism
\[
 \varphi_g : \alpha_g^* \Av_A(\scF) \simto \Av_A(\scF)
\]
as the canonical identification
\[
 \alpha_g^* \left( \bigoplus_{h \in A} \alpha_h^* \scF \right) = \bigoplus_{h \in A} \alpha_{hg}^* \scF = \bigoplus_{a \in A} \alpha_a^* \scF.
\]
This construction extends in a natural way to an exact functor
\[
 \Av_A : \Sh(X,\bk) \to \Sh_A(X,\bk).
\]

\begin{lem}
\label{lem:Av-adjoint}
 The functor $\Av_A$ is left and right adjoint to the forgetful functor~\eqref{eqn:ForA}.
\end{lem}

\begin{proof}
 To prove the lemma we have to define morphisms of functors
 \[
  \For_A \circ \Av_A \to \id, \quad \id \to \For_A \circ \Av_A, \quad \Av_A \circ \For_A \to \id, \quad \id \to \Av_A \circ \For_A,
 \]
and check the appropriate zigzag relations. Here we have
\[
 \For_A \circ \Av_A = \bigoplus_{g \in A} \alpha_g^*,
\]
and the first two morphisms are defined as the projection to and embedding of the factor $\alpha_e^*=\id$. On the other hand, for $(\scF, (\varphi_g)_{g \in A})$ in $\Sh_A(X)$, the morphisms
\[
 \scF \to \Av_A \circ \For_A(\scF) \to \scF
\]
are defined as $\bigoplus_{g \in A} (\varphi_g)^{-1}$ and $\bigoplus_{g \in A} \varphi_g$ respectively. (These morphisms of sheaves are morphisms of $A$-equivariant sheaves thanks to the cocyle condition~\eqref{eqn:cocycle}.) The zigzag relations are all trivial to check.
\end{proof}

Lemma~\ref{lem:Av-adjoint} implies that the functor $\Av_A$ sends injective objects of $\Sh(X,\bk)$ to injective objects of $\Sh_A(X,\bk)$. Since the category $\Sh(X,\bk)$ has enough injectives (see~\cite[\href{https://stacks.math.columbia.edu/tag/01DU}{Tag 01DU}]{stacks-project}), it follows that the same property holds in $\Sh_A(X,\bk)$. In fact, if $\scF$ belongs to $\Sh_A(X,\bk)$, and if $\mathscr{I}$ is an injective object of $\Sh(X,\bk)$ such that we have an injection $\For_A(\scF) \hookrightarrow \mathscr{I}$, then the map $\scF \to \Av_A(\mathscr{I})$ deduced by adjunction is injective, and $\Av_A(\mathscr{I})$ is injective in $\Sh_A(X,\bk)$.

In particular, if $f : X \to Y$ is an $A$-equivariant morphism between schemes with $A$-actions, recall that we have the (non derived) pushforward functor
\[
 f_* : \Sh_A(X,\bk) \to \Sh_A(Y,\bk).
\]
From the considerations on injective objects above we deduce that this functor admits a derived functor
\[
 R f_* : D^+ \Sh_A(X,\bk) \to D^+ \Sh_A(Y,\bk),
\]
which can be computed by means of injective resolutions.

% \begin{lem}
%  For any injective object $\mathscr{I}$ in $\Sh_A(X,\bk)$, the object $\For_A(\mathscr{I})$ is injective in $\Sh(X,\bk)$.
% \end{lem}
% 
% \begin{proof}
%  If $\mathscr{J}$ is an injective object in $\Sh(X,\bk)$ and $\For_A(\mathscr{I}) \hookrightarrow \mathscr{J}$ is an embedding, the injection $\mathscr{I} \hookrightarrow \Av_A(\mathscr{J})$ must be split since both objects are injective. Hence $\mathscr{I}$ is a direct summand in $\Av_A(\mathscr{J})$, so that $\For_A(\mathscr{I})$ is a direct summand in the injective sheaf $\For_A \circ \Av_A(\mathscr{J}) = \bigoplus_{g \in A} \alpha_g^* \mathscr{J}$, hence is injective.
% \end{proof}

%From this lemma one obtains 
Since the functor $\For_A$ admits an exact left adjoint, it sends injective objects to injective objects. It follows
that for an equivariant morphism $f : X \to Y$ as above we have a natural commutative diagram
\[
 \xymatrix@C=1.5cm{
 D^+ \Sh_A(X,\bk) \ar[r]^-{Rf_*} \ar[d]_-{\For_A} & D^+ \Sh_A(Y,\bk) \ar[d]^-{\For_A} \\
 D^+ \Sh(X,\bk) \ar[r]^-{Rf_*} & D^+ \Sh(Y,\bk),
 }
\]
where the lower horizontal arrow is the standard pushforward functor. In particular, in case $X$ and $Y$ are of finite type over some field $\F$ of finite cohomological dimension (e.g.~algebraically closed), and $\bk$ is torsion (e.g.~a field of positive characteristic), it is known that the ``standard'' functor $Rf_*$ sends $\Db \Sh(X,\bk)$ into $\Db \Sh(Y,\bk)$, see~\cite[\href{https://stacks.math.columbia.edu/tag/0F10}{Tag 0F10}]{stacks-project}. It follows that the ``equivariant'' functor $Rf_*$ considered above restricts to a functor
\[
 Rf_* : \Db \Sh_A(X,\bk) \to \Db \Sh_A(Y,\bk).
\]

Of course, since the functor
\[
 f^* : \Sh_A(Y,\bk) \to \Sh_A(X,\bk)
\]
is exact, we have an induced functor
\[
 f^* : D \Sh_A(Y,\bk) \to D \Sh_A(X,\bk)
\]
which maps $D^+ \Sh_A(Y,\bk)$ into $D^+ \Sh_A(X,\bk)$ and $\Db \Sh_A(Y,\bk)$ into $\Db \Sh_A(X,\bk)$ (and is compatible with the usual pullback functor $f^*$ in the obvious way). It is easily checked that the functor
\[
 f^* : D^+ \Sh_A(Y,\bk) \to D^+ \Sh_A(X,\bk)
\]
is left adjoint to
\[
 Rf_* : D^+ \Sh_A(X,\bk) \to D^+ \Sh_A(Y,\bk).
\]

%-------------------------------------
\subsection{Equivariant sheaves and quotient map}
\label{ss:equiv-quotient}
%-------------------------------------

From now on we assume that $X$ is of finite type over some base scheme
$S$, that $A$ is abelian,\footnote{This assumption is probably not necessary. It is made in order to use~\cite{stacks-project} (where rings are assumed to be commutative) as a reference for some facts about (sheaves of) $\bk[A]$-modules.} and that each $\alpha_g$ is an automorphism of $S$-schemes. We assume furthermore that the action is admissible in the sense of~\cite[Expos\'e~5, D\'efinition~1.7]{sga1}. Then we have a quotient scheme $X/A$, and a finite quotient morphism $q : X \to X/A$, see~\cite[Expos\'e~V, Corollaire~1.5]{sga1}. (Here, by definition $X/A$ is the scheme which represents the functor $Z \mapsto \Hom(X,Z)^A$, where $A$ acts on $\Hom(X,Z)$ via its action on $X$. It can be constructed by gluing affine schemes of the form $\Spec(R^A)$ where $\Spec(R) \subset X$ is an $A$-stable affine open subscheme of $X$.)

By finiteness the functor $q_*$ is then exact (see~\cite[\href{https://stacks.math.columbia.edu/tag/03QP}{Tag 03QP}]{stacks-project}), and defines an exact functor
\[
 q_* : \Sh_A(X,\bk) \to \Sh_A(X/A,\bk) \overset{\eqref{eqn:equ-sheaves-triv-action}}{\cong} \Sh(X/A,\bk[A]).
\]

\begin{rmk}
\label{rmk:admissible}
 As explained in~\cite[Expos\'e~V, Proposition~1.8]{sga1}, the action of $A$ on $X$ is admissible iff each orbit of $A$ is included in an affine open subset of $X$. This condition is automatic if $S=\Spec(\F)$ for some field $\F$ and $X$ is quasi-projective over $S$, see e.g.~\cite[p.~59, Exemple~1]{serre}. (This setting is the only one we will consider in practice.)
\end{rmk}

Recall that a complex $\scG \in D\Sh(X/A,\bk[A])$ is said to have
\emph{tor amplitude in $[a,b]$} (for $a, b \in \Z$ with $a \le b$) if for any $\scG' \in \Sh(X/A,\bk[A])$ we have
\[
 \scH^n(\scG' \lotimes_{\bk[A]} \scG)=0 \quad \text{unless $n \in [a,b]$,}
\]
see~\cite[\href{https://stacks.math.columbia.edu/tag/08FZ}{Tag
  08FZ}]{stacks-project}. Recall also that $\scG$ has tor amplitude in
$[a,b]$ iff for any geometric point $\overline{x}$ of $X$ the complex
of $\bk[A]$-modules $\scG_{\overline{x}}$ has tor amplitude in $[a,b]$,
see~\cite[\href{https://stacks.math.columbia.edu/tag/0DJJ}{Tag
  0DJJ}]{stacks-project}. Finally, $\scG \in D\Sh(X/A,\bk[A])$ is said to be
\emph{of finite tor dimension} if it has tor amplitude in $[a,b]$ for
some $a, b \in \Z$.

% Recall that a complex $\scG \in D\Sh(X/A,\bk[A])$ is said to be \emph{of finite tor dimension} if there exist $a,b \in \Z$ such that for any $\scG' \in \Sh(X/A,\bk[A]^{\mathrm{op}})$ we have
% \[
%  \scH^n(\scG' \lotimes_{\bk[A]} \scG)=0 \quad \text{unless $n \in [a,b]$,}
% \]
% see~\cite[\href{https://stacks.math.columbia.edu/tag/08FZ}{Tag 08FZ}]{stacks-project}. Recall also that $\scG$ has finite tor dimension iff for any geometric point $\overline{x}$ of $X$ the complex of $\bk[A]$-modules $\scG_{\overline{x}}$ has finite tor dimension,
% see~\cite[\href{https://stacks.math.columbia.edu/tag/0DJJ}{Tag 0DJJ}]{stacks-project}.

The action of $A$ on $X$ induces an action on geometric points. Namely, if
\[
 \overline{x} : \Spec(K) \to X
\]
is a geometric point and $g \in A$, then the geometric point $g \cdot \overline{x}$ is the composition
\[
 \Spec(K) \xrightarrow{\overline{x}} X \xrightarrow{\alpha_{g}} X.
\]
We will say that the $A$-action on $X$ is \emph{free} if each geometric point of $X$ has trivial stabilizer for this action.

\begin{lem}
\label{lem:tor-dim}
 Assume that $\bk$ is a field, and that the $A$-action on $X$ is free. Then for any $\scG$ in $\Sh_A(X,\bk)$, the sheaf $q_* \scG \in \Sh(X/A,\bk[A])$ has finite tor dimension.
\end{lem}

\begin{proof}
By the comments above,
to prove the lemma it suffices to prove that there exist $a,b \in \Z$ such that for any geometric point
$\overline{y} : \Spec(K) \to X/A$ the $\bk[A]$-module
$\scG_{\overline{y}}$ has tor amplitude in $[a,b]$. In fact one can take $a=b=0$, as we now explain.
 By~\cite[\href{https://stacks.math.columbia.edu/tag/03QP}{Tag 03QP}]{stacks-project} we have
 \[
  (q_* \scG)_{\overline{y}} = \bigoplus_{\overline{x}} \scG_{\overline{x}},
 \]
where $\overline{x}$ runs over the set $X_{\overline{y}}$ of maps $\overline{x} : \Spec(K) \to X$ such that $q \circ \overline{x} = \overline{y}$. The morphism $q$ is surjective, and its fibers are the $A$-orbits (see~\cite[Expos\'e~V, \S 1]{sga1}). Hence $X_{\overline{y}}$ is nonempty, and $A$ acts transitively on this set. Our assumption ensures on the other hand that this action has trivial stabilizers. Therefore, if we fix some $\overline{x} \in X_{\overline{y}}$, we deduce a bijection $A \simto X_{\overline{y}}$ determined by $g \mapsto g \cdot \overline{x}$. The $A$-equivariant structure on $\scG$ provides a canonical isomorphism
\[
 \scG_{\overline{x}} \simto \scG_{g \cdot \overline{x}}
\]
for each $g \in A$. Using these data we obtain an isomorphism
\[
 (q_* \scG)_{\overline{y}} = \bk[A] \otimes_\bk \scG_{\overline{x}},
\]
which is easily seen to be $A$-equivariant. Hence $(q_*
\scG)_{\overline{y}}$ is free as a $\bk[A]$-module, in particular of tor
amplitude in $[0,0]$.
\end{proof}

%-------------------------------------------
\subsection{Stalks at fixed points}
\label{ss:stalks-fixed-points}
%-------------------------------------------

We continue with the assumptions of~\S\ref{ss:equiv-quotient}. The closed subscheme $X^A \subset X$ of $A$-fixed points is the scheme which represents the functor $Z \mapsto \Hom(Z,X)^A$, where $A$ acts on $\Hom(Z,X)$ via its action on $X$. (The representability of this functor is easy in our setting: since our action is admissible it suffices to treat the case $X=\Spec(R)$ is affine; then $X^A$ is the spectrum of the maximal quotient of $R$ on which $A$ acts trivially, i.e.~the quotient of $R$ by the ideal generated by the elements $x - g\cdot x$ for $x \in R$ and $g \in A$.) As a set, the closed subscheme $X^A \subset X$ consists of the points $x \in X$ which are fixed by $A$ and such that the induced action on the residue field $k(x)$ is trivial.

By definition, any geometric point of $X$ which is stable under the $A$-action considered in~\S\ref{ss:equiv-quotient} factors through a geometric point of $X^A$. In particular, if $A$ is a simple group then the $A$-action on the open subset $U := X \smallsetminus X^A \subset X$ is free.

\begin{lem}
\label{lem:U-admissible}
 The $A$-action on $U$ is admissible.
\end{lem}

\begin{proof}
 Consider the closed subset $X^A \subset X$. Since the quotient morphism $q$ is finite, hence closed, the subset $q(X^A) \subset X/A$ is closed. Now from the definition of the subset $X^A \subset X$ given above and the fact that the fibers of $q$ are the $A$-orbits in $X$, one sees that
 \[
 U = q^{-1}(X/A \smallsetminus q(X^A)).
 \]
 Hence the claim follows from~\cite[Expos\'e V, Corollaire~1.4]{sga1}.
\end{proof}

From this lemma we obtain in particular that the quotient $U/A$ exists as a scheme. In fact, in the proof of this lemma we have seen that the open embedding $j : U \hookrightarrow X$ induces an open embedding $\overline{\jmath} : U/A \to X/A$, with complement $q(X^A)$, and that the quotient morphism $q_U : U \to U/A$ is the restriction of $q$ to $U$.

We now denote by $i : X^A \to X$ the closed embedding. Note that $i$ is $A$-equivariant for the trivial $A$-action on $X^A$; we therefore have a functor
\[
 i^* : D \Sh_A(X,\bk) \to D \Sh_A(X^A, \bk) \overset{\eqref{eqn:equ-sheaves-triv-action}}{\cong} D \Sh(X^A,\bk[A]). 
\]
In other words, if $\scF$ belongs to $D \Sh_A(X,\bk)$, then $i^*(\For_A(\scF))$ admits a canonical ``lift'' as a complex of sheaves of $\bk[A]$-modules. In particular, for any geometric point $\overline{x}$ of $X^A$ 
%(which we consider also as a geometric point of $X$) 
the complex
\[
 \scF_{i(\overline{x})} = (i^* \scF)_{\overline{x}}
\]
is in a natural way an object of $D(\bk[A]\Mod)$.

From now on we assume that $S=\Spec(\F)$ for some field $\F$ of finite cohomological dimension, and that $X$ is of finite type over $\F$. (This assumption implies that $U$ is also of finite type over $\F$, see~\cite[Example~3.45]{gw}. By~\cite[Corollaire~1.5]{sga1} we deduce that $X/A$ and $U/A$ also are of finite type.) We also assume that $\bk$ is a field of characteristic $\ell>0$ which is nonzero in $\F$.
The proof of the following proposition was explained to us by L. Illusie and W. Zheng. (A different, longer and slightly less easy proof can also be deduced from~\cite[Proposition~3.7]{dl}.) Recall that a bounded complex of $\bk[A]$-modules is called perfect if it is isomorphic in $\Db (\bk[A]\Mod)$ to a bounded complex of finitely generated projective modules.

\begin{prop}
\label{prop:illusie-zheng}
Assume that $A$ is a simple group.
 If $\scF \in \Db \Sh_A(U,\bk)$ is such that $\For_A(\scF)$ has constructible cohomology sheaves (see~\cite[\href{https://stacks.math.columbia.edu/tag/03RW}{Tag 03RW}]{stacks-project}), then for any geometric point $\overline{x}$ of $X^A$ the complex of $\bk[A]$-modules
 \[
  (Rj_* \scF)_{\overline{x}}
 \]
is perfect.
\end{prop}

\begin{proof}
Recall that the morphism $j$ has finite cohomological dimension by~\cite[\href{https://stacks.math.columbia.edu/tag/0F10}{Tag 0F10}]{stacks-project}. As a consequence the complex $Rj_* \scF$ is bounded, so that $(Rj_* \scF)_{\overline{x}}$ is bounded. On the other hand, by~\cite[Th. finitude, Th\'eor\`eme~1.1]{sga4.5}, the complex $Rj_* \scF$ is constructible; hence its stalk $(Rj_* \scF)_{\overline{x}}$ has finite-dimensional cohomology. By~\cite[\href{https://stacks.math.columbia.edu/tag/0658}{Tag 0658}]{stacks-project}, we deduce that to prove that this complex is perfect it suffices to prove that it has finite tor dimension. But since the image $x$ of $\overline{x}$ is fixed by $A$ we have $q^{-1}(q(x))=\{x\}$, which implies that
\[
(Rj_* \scF)_{\overline{x}} = \bigl( q_* (Rj_* \scF) \bigr)_{q(\overline{x})}
\]
by~\cite[\href{https://stacks.math.columbia.edu/tag/03QP}{Tag 03QP}]{stacks-project}. Now we have $q \circ j = \overline{\jmath} \circ q_U$, so that
\[
q_* (Rj_* \scF) = R\overline{\jmath}_* ( (q_U)_* \scF).
\]
By Lemma~\ref{lem:tor-dim} the sheaf $(q_U)_* \scF$ has finite tor dimension, hence by~\cite[Expos\'e XVII, Th\'eor\`eme~5.2.11]{sga4} and~\cite[\href{https://stacks.math.columbia.edu/tag/0F10}{Tag 0F10}]{stacks-project} the complex $R\overline{\jmath}_* ((q_U)_* \scF)$ also does, so that its stalk
\[
\bigl( R\overline{\jmath}_* ( (q_U)_* \scF ) \bigr)_{q(\overline{x})}
\]
must have finite tor dimension, which finishes the proof.
\end{proof}

%%%%%%%%%%%%%%%%%%%%%%%%%
\section{Smith theory for \'etale sheaves}
\label{sec:Smith-etale}
%%%%%%%%%%%%%%%%%%%%%%%%%

%-------------------------------------
\subsection{\texorpdfstring{$\varpi$}{mul}-equivariant derived categories}
\label{ss:equiv-Db}
%-------------------------------------

The formalism of ``Smith theory'' that we will build will use the equivariant derived category of Bernstein--Lunts~\cite{bl}. This category is explicitly constructed only in a topological setting in~\cite{bl}, but it is well known that it applies also in the setting of \'etale sheaves under appropriate assumptions, see~\cite[\S 4.3]{bl}. 
%(See also~\cite[Chap.~10]{ar-book} for a detailed treatment in a more difficult setting.) 
In this subsection we briefly recall this construction in the particular case that we require.

So, from now on we fix an algebraically closed field $\F$ of characteristic $p$, and a finite field $\bk$ of characteristic $\ell \neq p$. We will consider (admissible) actions of the finite $\F$-group scheme $\varpi=\roots_\ell$ of $\ell$-th roots of unity on $\F$-schemes of finite type. Here since $\F$ is algebraically closed, $\roots_\ell$ is the constant group scheme associated with the finite group $\roots_\ell(\F)$, so that the constructions of Section~\ref{sec:equiv-sheaves} also apply in this setting. For simplicity, we will not explicitly distinguish the group scheme $\roots_\ell$ and the finite group $\roots_\ell(\F)$.

The construction of~\cite{bl} uses some ``acyclic resolutions.'' In this case these resolutions can be constructed explicitly as follows: 
%we choose an identification of $\varpi$ with the group of $\ell$-th roots of unity in $\F^\times$, and 
for any $n \geq 0$ we set
\[
 V_n := \F^n \smallsetminus \{0\},
\]
with the (admissible) action of $\varpi$ induced by the dilation action of the multiplicative group $\Gm$. We have
\[
 \coH^m(V_n; \bk) = \begin{cases}
                     \bk & \text{if $m=0$;} \\
                     0 & \text{if $1 \leq m \leq 2n-2$.}
                    \end{cases}
\]
From this we see that for any $\F$-scheme $X$ of finite type the projection
\[
 p^{X}_{n} : V_n \times X \to X
\]
is $(2n-2)$-acyclic, in the sense that for any $X$-scheme $Y$ of finite type the morphism $p^{X,Y}_{n} : Y \times_X (V_n \times X) \to Y$ is such that for any (\'etale) $\bk$-sheaf $\scF$ the morphism
\[
 \scF \to \tau_{\leq 2n-2}(R(p^{X,Y}_n)_* (p^{X,Y}_n)^* \scF)
\]
induced by adjunction is an isomorphism. (In fact, here by the K\"unneth formula~\cite[\href{https://stacks.math.columbia.edu/tag/0F1N}{Tag 0F1N}]{stacks-project} we have $R(p^{X,Y}_n)_* (p^{X,Y}_n)^* \scF \cong \coH^\bullet(V_n,\bk) \otimes_\bk \scF$.)

We now fix an $\F$-scheme $X$ of finite type endowed with an admissible action of $\varpi$. 
%(By Remark~\ref{rmk:admissible}, the admissibility will e.g.~be automatic if $X$ is quasi-projective over $\F$.) 
For any $n \geq 1$ we set
\[
 P_n^X := V_n \times X,
\]
and consider the projection $p^{X}_n : P_n^X \to X$ as above. Since the actions of $\varpi$ on $V_n$ and $X$ are admissible, this property also holds for the product $V_n \times X$ (with the diagonal action), so that we can consider the quotient
\[
\overline{P}_n^X := P_n^X / \varpi.
\]
With the notation of~\S\ref{ss:stalks-fixed-points}, we have $(P_n^X)^\varpi=\varnothing$; therefore, by~\cite[Expos\'e~V, Corollaire~2.3]{sga1} the quotient map $q_n^X : P_n^X \to \overline{P}_n^X$ is \'etale. In fact, in view of~\cite[Expos\'e~V, Proposition~2.6]{sga1} this map is an \'etale locally trivial principal homogeneous space for $\varpi$ in the sense of~\cite[\href{https://stacks.math.columbia.edu/tag/049A}{Tag 049A}]{stacks-project}. 
%(See also~\cite[\S 2.1]{ar-book} for a concise reminder on this notion.)

For any $n$, we will denote by
\[
 \Db (X,\varpi,n,\bk)
\]
the category whose
\begin{itemize}
 \item objects are triples $(\scF_n, \scF_X, \beta)$ where $\scF_n$ is an object of $\Db \Sh(\overline{P}_n^X,\bk)$, $\scF_X$ is an object of $\Db \Sh(X,\bk)$, and
 \[
  \beta : (q_n^X)^* \scF_n \simto (p_n^X)^* \scF_X
 \]
is an isomorphism;
\item morphisms from $(\scF_n, \scF_X, \beta)$ to $(\scG_n, \scG_X, \gamma)$ are pairs $(\varphi_n,\varphi_X)$ with
\[
 \varphi_n : \scF_n \to \scG_n, \qquad \varphi_X : \scF_X \to \scG_X
\]
morphisms such that $\gamma \circ ((q_n^X)^* \varphi_n) = ((p_n^X)^* \varphi_X) \circ \beta$.
\end{itemize}
% We then have a natural functor
% \[
%  \For_{X,n} : \Db (X,n,\bk) \to \Db \Sh(X,\bk).
% \]

For any bounded interval $I \subset \Z$ we denote by $D^I(X,\varpi,n,\bk)$ the full subcategory of $\Db (X,\varpi,n,\bk)$ whose objects are the triples $(\scF_n,\scF_X,\beta)$ where $\scH^m(\scF_X)$ vanishes unless $m \in I$. Then the category $D^I(X,\varpi,n,\bk)$ does not depend (up to canonical equivalence) on the choice of $n$, as long as $2n-2 \geq |I|$, where $|I|$ is the length of $I$; in fact, by the same arguments as in~\cite[\S 2.3.4]{bl}, 
%(see also~\cite[\S 10.3]{ar-book}), 
if $n,m$ satisfy this condition then the natural functors from $D^I(X,\varpi,n,\bk)$ and $D^I(X,\varpi,m,\bk)$ to the category defined similarly with $P_n^X$ and $P_m^X$ replaced by
\[
 V_n \times V_m \times X = P_n^X \times_X P_m^X
\]
(with the diagonal $\varpi$-action) are equivalences of categories.

We can therefore define the $\varpi$-equivariant derived category
\[
 \Db_\varpi(X,\bk)
\]
as the direct limit of the categories $D^I(X,\varpi,n,\bk)$ with $n \gg 0$, where $I$ runs over the bounded intervals of $\Z$. This category admits a canonical structure of triangulated category, see~\cite[\S\S2.5.1--2.5.2]{bl}. By construction, we have a canonical triangulated forgetful functor
\begin{equation}
\label{eqn:For-Dbequ}
 \For_\varpi : \Db_\varpi(X,\bk) \to \Db \Sh(X,\bk)
\end{equation}
which sends a triple $(\scF_n, \scF_X, \beta)$ to $\scF_X$.

\begin{rmk}
\label{rmk:Hom-Dbeq}
 As explained in~\cite[Lemma~2.3.2]{bl}, if $2n-2 \geq |I|$ the functor
 \[
  D^I(X,\varpi,n,\bk) \to \Db\Sh(\overline{P}_n^X,\bk)
 \]
sending $(\scF_n, \scF_X, \beta)$ to $\scF_n$ is fully faithful. In particular, morphisms between objects in $\Db_\varpi(X,\bk)$ can always be computed as morphisms in derived categories of quotients of ``sufficiently acyclic'' resolutions.
\end{rmk}

This construction is of course functorial in $X$. 
%Here, to make sure that the pusforward functors respect bounded derived categories we assume that our schemes are of finite type over $\F$ (see once again~\cite[\href{https://stacks.math.columbia.edu/tag/0F10}{Tag 0F10}]{stacks-project}). 
Namely, consider $\F$-schemes of finite type $X,Y$ endowed with admissible actions of $\varpi$, and a $\varpi$-equivariant morphism of $\F$-schemes $f : X \to Y$. (Note that $f$ is automatically quasi-compact since $X$ is Noetherian, see~\cite[\href{https://stacks.math.columbia.edu/tag/01P0}{Tag 01P0}]{stacks-project}. It is also locally of finite type by~\cite[\href{https://stacks.math.columbia.edu/tag/01T8}{Tag 01T8}]{stacks-project}, hence of finite type. Finally, $X$, $Y$ and $f$ are quasi-separated by~\cite[\href{https://stacks.math.columbia.edu/tag/01T7}{Tag 01T7}]{stacks-project}.)
\begin{enumerate}
\item
We have a ($*$-)pullback functor
\[
 f^* : \Db_\varpi(Y,\bk) \to \Db_\varpi(X,\bk),
\]
which can be explicitly described in terms of the pullback functors associated with $f$ and the induced morphism $\overline{P}_n^X \to \overline{P}_n^Y$.
\item
%If $f$ is 
%quasi-compact\footnote{Since $X$ and $Y$ are assumed to be of finite type, this condition is equivalent to requiring that $f$ is of finite type; see~\cite[\href{https://stacks.math.columbia.edu/tag/01T8}{Tag 01T8}]{stacks-project}.} and 
%quasi-separated we 
We have a ($*$-)pushforward functor
\[
 Rf_* : \Db_\varpi(X,\bk) \to \Db_\varpi(Y,\bk).
\]
(Here we use the fact that the usual pushforward functors respect the bounded derived categories, see~\cite[\href{https://stacks.math.columbia.edu/tag/0F10}{Tag 0F10}]{stacks-project}, and also that the morphisms $p_n^X$, $p_n^Y$, $q_n^X$, $q_n^Y$ are smooth and that the induced morphism $\overline{P}_n^X \to \overline{P}_n^Y$ is quasi-compact and quasi-separated, which allows to use the smooth base change theorem~\cite[\href{https://stacks.math.columbia.edu/tag/0EYU}{Tag 0EYU}]{stacks-project} to ``transport'' the isomorphism $\beta$.) 
\item
If we assume that $f$ is separated 
%of finite type 
%and that $Y$ is quasi-separated, 
then we also have a $!$-pushforward functor
\[
 Rf_! : \Db_\varpi(X,\bk) \to \Db_\varpi(Y,\bk),
\]
see~\cite[Expos\'e~XVII]{sga4}.
(Here, the fact that the $!$-pushforward functors respect bounded derived categories follows from~\cite[Expos\'e~XVII, Corollaire~5.2.8.1]{sga4}, and we use the base change theorem~\cite[Expos\'e~XVII, Th\'eor\`eme~5.2.6]{sga4} to ``transport'' $\beta$.) 
\item
Under the same assumption we also have a $!$-pullback functor
\[
 f^! : \Db_\varpi(Y,\bk) \to \Db_\varpi(X,\bk),
\]
see~\cite[Expos\'e~XVIII]{sga4}. (The fact that the $!$-pullback functors respect bounded derived categories is explained in~\cite[Th.~finitude, comments after Corollaire~1.5]{sga4.5}. And once again we use the smoothness of $q_n^X$, $q_n^Y$, $p_n^X$, $p_n^Y$, and the fact that for smooth maps the $*$- and $!$-pullback functors coincide up to shift, see~\cite[Expos\'e~XVIII, Th\'eor\`eme~3.2.5]{sga4}, to ``transport'' the isomorphisms $\beta$.)
\end{enumerate}

By construction, all of these functors are compatible with the forgetful functor~\eqref{eqn:For-Dbequ} in the obvious way, and satisfy the usual adjunction properties.

\begin{rmk}
\label{rmk:separated}
In practice all the schemes we will consider will be quasi-projective over $\F$, hence separated (see~\cite[\href{https://stacks.math.columbia.edu/tag/01VX}{Tag 01VX}]{stacks-project}), so that any morphism between them will automatically be separated (see~\cite[\href{https://stacks.math.columbia.edu/tag/01KV}{Tag 01KV}]{stacks-project}).
%And the morphisms for which we will consider the push/pull functors will be locally closed embeddings, hence they will also be separated by~\cite[\href{https://stacks.math.columbia.edu/tag/01L7}{Tag 01L7}]{stacks-project}.
\end{rmk}

%------------------------------------------
\subsection{Equivariant derived categories and equivariant sheaves}
%------------------------------------------

We continue with the setting of~\S\ref{ss:equiv-Db}.
In Section~\ref{sec:equiv-sheaves} we have studied equivariant sheaves on schemes, and in ~\S\ref{ss:equiv-Db} we have considered the equivariant derived category. It is now time to explain the relation between these two constructions. This relation is based on the observation that (for any $\F$-scheme $X$ of finite type with an admissible $\varpi$-action, and any $n \geq 1$) the natural pullback functor
\[
 \Sh(\overline{P}^X_n, \bk) \to \Sh_\varpi(P^X_n, \bk)
\]
is an equivalence of categories, by the sheaf condition applied to the \'etale covering $q_n^X : P_n^X \to \overline{P}_n^X$. Therefore, for any $\varpi$-equivariant sheaf $\scF$ on $X$ the pullback $(p_n^X)^* \scF$ admits a natural structure of $\varpi$-equivariant sheaf on $P_n^X$, hence descends to a sheaf $\scF_n$ on $\overline{P}_n^X$. Using this construction we define a canonical triangulated functor
\begin{equation}
 \label{eqn:eq-sheaves-Db}
 \Db \Sh_\varpi(X,\bk) \to \Db_\varpi(X,\bk).
\end{equation}

\begin{prop}
\label{prop:real-equiv}
 The functor~\eqref{eqn:eq-sheaves-Db} is an equivalence of categories.
\end{prop}

\begin{proof}
 Let us first show that our functor is fully faithful. For this we need to show that for any $\varpi$-equivariant sheaves $\scF,\scF'$ on $X$ and any $m \in \Z_{\geq 0}$, for $n \gg 0$ the natural map
 \[
  \Hom_{\Db \Sh_\varpi(X,\bk)}(\scF,\scF'[m]) \to \Hom_{\Db \Sh(\overline{P}^X_n, \bk)}(\scF_n, \scF'_n[m])
 \]
is an isomorphism, see Remark~\ref{rmk:Hom-Dbeq}. By construction, this amounts to proving that for $n \gg 0$ the pullback functor induces an isomorphism
\[
 \Hom_{\Db \Sh_\varpi(X,\bk)}(\scF,\scF'[m]) \simto \Hom_{\Db \Sh_\varpi(P_n^X,\bk)}((p_n^X)^* \scF, (p_n^X)^*\scF'[m]).
\]
However by adjunction we have
\begin{multline*}
 \Hom_{\Db \Sh_\varpi(P_n^X,\bk)}((p_n^X)^* \scF, (p_n^X)^*\scF'[m]) \\
 \cong \Hom_{\Db \Sh_\varpi(X,\bk)}(\scF, R(p_n^X)_*(p_n^X)^*\scF'[m]).
\end{multline*}
Now the right-hand side can be replaced by
 \begin{multline*}
 \Hom_{\Db \Sh_\varpi(X,\bk)}(\scF, \tau_{\leq 0}(R(p_n^X)_*(p_n^X)^*\scF'[m])) \\ 
 \cong \Hom_{\Db \Sh_\varpi(X,\bk)}(\scF, \tau_{\leq m}(R(p_n^X)_*(p_n^X)^*\scF') [m]).
\end{multline*}
If $2n-2 \geq m$ the morphism
\[
 \scF' \to \tau_{\leq m}(R(p_n^X)_*(p_n^X)^*\scF')
\]
induced by adjunction is an isomorphism, which concludes the proof of fully faithfulness.

Once fully faithfulness is established, to conclude the proof it suffices to prove that images of $\varpi$-equivariant sheaves on $X$ generate $\Db_\varpi(X,\bk)$ as a triangulated category. This is however clear from the construction of the ``standard'' t-structure in~\cite[\S\S2.5.1--2.5.2]{bl}.
\end{proof}

In particular, in the special case where the $\varpi$-action on $X$ is trivial, using Proposition~\ref{prop:real-equiv} combined with~\eqref{eqn:equ-sheaves-triv-action}, we obtain a canonical equivalence of triangulated categories
\begin{equation}
\label{eqn:equ-Db-triv-action}
\Db_\varpi(X,\bk) \cong \Db \Sh(X,\bk[\varpi]).
\end{equation}

We now consider two $\F$-schemes $X$ and $Y$ of finite type, with admissible actions of $\varpi$, and a $\varpi$-equivariant morphism $f : X \to Y$. We have considered functors
\[
 \xymatrix{
 \Db \Sh_{\varpi}(X,\bk) \ar@<1ex>[r]^-{Rf_*} & \ar@<1ex>[l]^-{f^*} \Db \Sh_{\varpi}(Y,\bk)
 }
\]
in~\S\ref{ss:equ-sheaves-inj}, and functors
\[
 \xymatrix{
 \Db_{\varpi}(X,\bk) \ar@<1ex>[r]^-{Rf_*} & \ar@<1ex>[l]^-{f^*} \Db_{\varpi}(Y,\bk)
 }
\]
in~\S\ref{ss:equiv-Db}. These functors are related in the natural way, as explained in the following lemma.

\begin{lem}
\label{lem:push-pull-equ}
The diagrams
\[
 \vcenter{
 \xymatrix{
 \Db \Sh_{\varpi}(Y,\bk) \ar[r]^-{f^*} \ar[d]_-{\wr} & \Db \Sh_{\varpi}(X,\bk) \ar[d]^-{\wr} \\
  \Db_{\varpi}(Y,\bk) \ar[r]^-{f^*} & \Db_{\varpi}(X,\bk)
 }} \quad \text{and} \quad
  \vcenter{
 \xymatrix{
 \Db \Sh_{\varpi}(X,\bk) \ar[r]^-{Rf_*} \ar[d]_-{\wr} & \Db \Sh_{\varpi}(Y,\bk) \ar[d]^-{\wr} \\
  \Db_{\varpi}(X,\bk) \ar[r]^-{Rf_*} & \Db_{\varpi}(Y,\bk)
 }}
\]
are commutative, where the vertical arrows are the equivalences of Proposition~\ref{prop:real-equiv}.
\end{lem}

\begin{proof}
The commutativity of the left diagram can be seen from the definitions; the commutativity of the right diagram follows by adjunction.
\end{proof}

%-----------------------------------------
\subsection{The crucial lemma}
\label{ss:crucial-lemma}
%-----------------------------------------

We can now prove the lemma that will allow us to develop the ``Smith theory for sheaves" from~\cite{treumann} in our setting of \'etale sheaves.

We consider again an $\F$-scheme $X$ of finite type, with an admissible action of $\varpi$. As in~\S\ref{ss:stalks-fixed-points} we consider the fixed points subscheme $X^\varpi$ and the closed, resp.~open, embedding
\[
i : X^\varpi \hookrightarrow X, \quad \text{resp.} \quad j : X \smallsetminus X^\varpi \hookrightarrow X.
\]
(Here $i$ and $j$ are automatically separated, see~\cite[\href{https://stacks.math.columbia.edu/tag/01L7}{Tag 01L7}]{stacks-project}.)
%~\cite[\href{https://stacks.math.columbia.edu/tag/01L7}{Tag 01L7}]{stacks-project}; therefore the $*$-pushforward functors between equivariant derived categories associated with these maps are well defined.
For any $\scF$ in $\Db_\varpi(X,\bk)$ we have a canonical morphism
\begin{equation}
\label{eqn:morph-Smith}
 i^! \scF \to i^* \scF
\end{equation}
in the category $\Db_\varpi(X^\varpi,\bk)$,
%\overset{\eqref{eqn:equ-Db-triv-action}}{\cong} \Db \Sh(X^\varpi, \bk[\varpi])$,
which can be obtained by applying the functor $i^*$ to the adjunction morphism $i_! i^! \scF \to \scF$. 

From now on we will not consider the entire $\varpi$-equivariant derived category $\Db_\varpi(X,\bk)$, but only the full triangulated subcategory 
$\Db_{\varpi,c}(X,\bk)$ whose objects are those $\scF \in \Db_{\varpi}(X,\bk)$ such that the complex $\For_\varpi(\scF)$ has constructible cohomology objects, where $\For_\varpi$ is as in~\eqref{eqn:For-Dbequ}. The pushforawrd and pullback functors considered in~\S\ref{ss:equiv-Db} preserve these subcategories (in the obvious sense) by~\cite[Th. finitude, Corollaire~1.5]{sga4.5}.

For any $\F$-scheme $Y$ of finite type with trivial action of $\varpi$, we will say that an object $\scF$ in $\Db_{\varpi,c}(Y,\bk)$ has perfect geometric stalks if, denoting by $\scF'$ the image of $\scF$ under the equivalence $\Db_\varpi(Y,\bk) \cong \Db \Sh(Y, \bk[\varpi])$ from~\eqref{eqn:equ-Db-triv-action}, for any geometric point $\overline{y}$ of $Y$ the complex $\scF'_{\overline{y}}$ is a perfect complex of $\bk[\varpi]$-modules. 

\begin{lem}
\label{lem:crucial-Smith}
For any $\scF$ in $\Db_{\varpi,c}(X,\bk)$, 
%and any geometric point $\overline{x}$ of $X^\varpi$, the stalk of 
the cone of~\eqref{eqn:morph-Smith} has perfect geometric stalks.
%at $\overline{x}$ is a perfect complex of $\bk[\varpi]$-modules.
\end{lem}

\begin{proof}
From the standard distinguished triangle in the ``recollement'' formalism we see that the cone of~\eqref{eqn:morph-Smith} is isomorphic to $i^* Rj_* j^*(\scF)$. The complex we want to consider is therefore
\[
\bigl( Rj_* j^*(\scF) \bigr)_{\overline{x}},
\]
where $\overline{x}$ is a geometric point of $X^\varpi$.
In these terms, the desired claim follows from Proposition~\ref{prop:illusie-zheng} and Lemma~\ref{lem:push-pull-equ}.
\end{proof}

Later we will also need the following lemma, whose proof is close to that of Lemma~\ref{lem:crucial-Smith}. Here we consider two $\F$-schemes of finite type $Z$ and $Y$ with trivial actions of $\varpi$, and a morphism of $\F$-schemes $f : Z \to Y$. Then we have a derived functor $Rf_* : \Db_{\varpi}(Z,\bk) \to \Db_{\varpi}(Y,\bk)$, see~\S\ref{ss:equiv-Db}, which sends the subcategory $\Db_{\varpi,c}(Z,\bk)$ into $\Db_{\varpi,c}(Y,\bk)$ by~\cite[Th. finitude, Th\'eor\`eme~1.1]{sga4.5}.
% so that by~\eqref{eqn:equ-Db-triv-action} we have
% \begin{equation}
% \label{eqn:equ-Db-triv-action-pushforward}
% \Db_\varpi(X,\bk) \cong \Db \Sh(X,\bk[\varpi]), \quad \Db_\varpi(Y,\bk) \cong \Db \Sh(Y,\bk[\varpi]).
% \end{equation}
% We will say that an object $\scF$ in $\Db_\varpi(X,\bk)$ has perfect geometric stalks if, denoting by $\scF'$ the image of $\scF$ under the first equivalence in~\eqref{eqn:equ-Db-triv-action-pushforward}, for any geometric point $\overline{x}$ of $X$ the complex $\scF_{\overline{x}}$ is a perfect complex of $\bk[\varpi]$-modules. We use the same terminology for objects in $\Db_\varpi(Y,\bk)$.
%We also consider a morphism of $\F$-schemes $f : X \to Y$.

\begin{lem}
\label{lem:pushforward-perfect-stalks}
 The functor
 \[
  Rf_* : \Db_{\varpi,c}(Z,\bk) \to \Db_{\varpi,c}(Y,\bk)
 \]
transforms objects with perfect geometric stalks into objects with perfect geometric stalks.
\end{lem}

\begin{proof}
 As in the proof of Proposition~\ref{prop:illusie-zheng} it suffices to check that $Rf_*$ transforms objects of finite tor dimension into objects of finite tor dimension, which follows from~\cite[Expos\'e~XVII, Th\'eor\`eme~5.2.11]{sga4} and~\cite[\href{https://stacks.math.columbia.edu/tag/0F10}{Tag 0F10}]{stacks-project}.
\end{proof}

%-----------------------------------------------------
\subsection{$\Gm$-equivariant derived categories}
\label{ss:Gm-equ-Db}
%-----------------------------------------------------

Later we will also need to consider equivariant derived categories for actions of the multiplicative group $\Gm$ over $\F$. The construction of this category is similar to, and in fact simpler than, the construction in~\S\ref{ss:equiv-Db}. Namely, for any $n \geq 1$ the $\varpi$-action on $V_n$ is obtained by restriction from a natural $\Gm$-action, and moreover we have a canonical map $V_n \to \mathbb{P}_\F^{n-1}$ which is a Zariski locally trivial principal $\Gm$-bundle. Therefore, given any $\F$-scheme $X$ endowed with an action of $\Gm$, we consider the diagonal action on $V_n \times X$, and we have a Zariski locally trivial principal $\Gm$-bundle
\[
 V_n \times X \to V_n \times^{\Gm} X
\]
for some scheme $V_n \times^{\Gm} X$ which can be constructed by (Zariski) gluing over the natural cover which trivializes the map $V_n \to \mathbb{P}_\F^{n-1}$. If we assume $X$ to be of finite type, then as in~\S\ref{ss:equiv-Db} the map $p_{n}^X : V_n \times X \to X$ is $(2n-2)$-acyclic, which allows to define a category $\Db(X,\Gm,n,\bk)$ in terms similar to those for $\Db(X,\varpi,n,\bk)$, and check that the subcategory $D^I(X,\Gm,n,\bk)$ does not depend on the choice of $n$ as long as $2n-2 \geq |I|$. We can finally define the equivariant derived category $\Db_{\Gm}(X,\bk)$ as the direct limit of the categories $D^I(X,\Gm,n,\bk)$ (with $n \gg 0$) over the finite intervals $I \subset \Z$. These categories have the same functoriality properties as the categories $\Db_\varpi(X,\bk)$; in particular we have a natural (triangulated) forgetful functor
\[
 \For_{\Gm} : \Db_{\Gm}(X,\bk) \to \Db \Sh(X,\bk).
\]
As in the $\varpi$-equivariant setting, we will denote by $\Db_{\Gm,c}(X,\bk)$ the full subcategory of $\Db_{\Gm}(X,\bk)$ whose objects are those $\scF$ such that $\For_{\Gm}(\scF)$ has constructible cohomology sheaves.

The category $\Db_{\Gm}(X,\bk)$ has a canonical object whose image under $\For_{\Gm}$ is the constant sheaf $\underline{\bk}_X$; it will also be denoted $\underline{\bk}_X$.
In these terms, the $\Gm$-equivariant cohomology of a complex $\scF$ in $\Db_{\Gm}(X,\bk)$ is defined as
\[
 \coH^\bullet_{\Gm}(X,\scF)=\bigoplus_{n \in \Z} \Hom_{\Db_{\Gm}(X,\bk)}(\underline{\bk}_X,\scF[n]).
\]
In the case $X=\Spec(\F)=:\pt$,
it is well known (and easy to see) that we have a graded algebra isomorphism\footnote{\label{fn:Tate-twist}To be more precise, to get the isomorphism~\eqref{eqn:eq-cohom-pt} one needs to fix a trivialization of the Tate sheaf on $\pt$, see e.g.~the proof of Lemma~\ref{lem:perf-fd-cohom} below. This is possible---though not canonical---since $\F$ is algebraically closed; we fix such a trivialization once and for all.}
\begin{equation}
\label{eqn:eq-cohom-pt}
 \coH^\bullet_{\Gm}(\pt,\ubk_{\pt}) = \bk[x],
\end{equation}
where $x$ has degree $2$. If $n$ is even, we will denote by
\[
 \mathrm{can}_{\pt}^n : \ubk_{\pt} \to \ubk_{\pt}[n]
\]
the morphism obtained as the inverse image of $x^{n/2}$.

%We now fix a non trivial $\ell$-th root of unity in $\F$, and consider the associated 
By definition we have an embedding $\varpi \subset \Gm$, which provides a $\varpi$-action on $X$ by restriction, and
%By construction 
we have a canonical ``restriction" triangulated functor
\[
 \Res^{\Gm}_\varpi : \Db_{\Gm}(X,\bk) \to \Db_{\varpi}(X,\bk).
\]
%where in the right-hand side we consider the $\varpi$-action obtained by restriction. 
This functor is compatible (in the obvious sense) with the pushforward and pullback functors when they are defined, and moreover satisfies
\[
 \For_{\varpi} \circ \Res^{\Gm}_\varpi \cong \For_{\Gm}.
\]
As a consequence, it must send $\Db_{\Gm,c}(X,\bk)$ into $\Db_{\varpi,c}(X,\bk)$.

%-----------------------------------------------------
\subsection{The ``Smith category'' of a point}
\label{ss:Smith-pt}
%-----------------------------------------------------

In this subsection we consider the special case of the constructions of~\S\ref{ss:Gm-equ-Db} where $X=\pt$. In this case, in view of~\eqref{eqn:equ-Db-triv-action} we have an equivalence of triangulated categories
\begin{equation}
\label{eqn:equv-Db-pt}
 \Db_{\varpi}(\pt,\bk) \cong \Db (\bk[\varpi]\Mod).
\end{equation}
Under this equivalence, the full subcategory $\Db_{\varpi,c}(\pt,\bk)$ corresponds to the full subcategory of $\Db (\bk[\varpi]\Mod)$ whose objects are the complexes whose cohomology is finite-dimensional (or equivalently finitely generated over $\bk[\varpi])$, which itself is canonically equivalent to the category $\Db (\bk[\varpi]\Mof)$.
%, where $\bk[\varpi]\Mof$ is the category of finite-dimensional $\bk[\varpi]$-modules. 
In particular, the sheaf $\Res^{\Gm}_\varpi(\ubk_{\pt})$ corresponds to the trivial module $\bk$.

We will denote by
\[
 \Db_{\Gm,c}(\pt,\bk)_{\varpi\perf} \subset \Db_{\Gm,c}(\pt,\bk)
\]
the full triangulated subcategory whose objects are the complexes $\scF$ such that $\Res^{\Gm}_\varpi(\scF)$, considered as a complex of $\bk[\varpi]$-modules through~\eqref{eqn:equv-Db-pt}, is perfect.
% (i.e.~isomorphic in $\Db (\bk[\varpi]\Mod)$ to a bounded complex of finitely generated projective modules).
We then set
\[
 \Smith(\pt,\bk) := \Db_{\Gm,c}(\pt,\bk) / \Db_{\Gm,c}(\pt,\bk)_{\varpi\perf},
\]
where we consider the Verdier quotient category.

The following lemma will be crucial for us below, in that it will allow to use some parity vanishing arguments in various variants of the category $\Smith(\pt,\bk)$.

\begin{lem}
\label{lem:Hom-Smith-pt}
 For any $n \in \Z$ we have
 \[
  \Hom_{\Smith(\pt,\bk)}(\underline{\bk}_{\pt}, \underline{\bk}_{\pt}[n]) \cong \begin{cases}
                                                               \bk & \text{if $n$ is even;}\\
                                                               0 & \text{otherwise,}
                                                              \end{cases}
 \]
where we omit the quotient functor $\Db_{\Gm,c}(\pt,\bk) \to \Smith(\pt,\bk)$.
\end{lem}

The proof of Lemma~\ref{lem:Hom-Smith-pt} will require some preparation. We start with the following claim.

\begin{lem}
\label{lem:eq-cohom-Gm-pi}
 For any $\scF$ in $\Db_{\Gm,c}(\pt,\bk)$, there exists a canonical isomorphism of graded $\bk$-vector spaces
 \[
  \bigoplus_{m \in \Z} \Hom_{\Db_{\varpi}(\pt,\bk)} \bigl( \ubk_\pt, \Res^{\Gm}_\varpi(\scF)[m] \bigr) \cong \coH^\bullet_{\Gm}(\pt, \scF) \oplus \coH^\bullet_{\Gm}(\pt, \scF)[-1].
 \]
\end{lem}

\begin{proof}
 We fix $\scF$ in $\Db_{\Gm,c}(\pt,\bk)$ and $m \in \Z$. Then for $n \gg 0$ the object $\scF$ is represented by a triple $(\scF_n,\scF_X,\beta)$ in $\Db(\pt,\Gm,n,\bk)$, and by an analogue of Remark~\ref{rmk:Hom-Dbeq} we have
 \[
  \coH^m_{\Gm}(\pt,\scF)=\Hom_{\Db(\mathbb{P}_\F^{n-1},\bk)}(\ubk_{\mathbb{P}_\F^{n-1}}, \scF_n[m]),
 \]
and similarly for $\coH^{m+1}_{\Gm}(\pt,\scF)$. If we denote by
\[
 \pi_n : V_n / \varpi \to \mathbb{P}_\F^{n-1}
\]
the natural map, then $\Res^{\Gm}_\varpi(\scF)$ is represented by the object $((\pi_n)^*\scF_n,\scF_X,\beta)$ in $\Db(\pt,\varpi,n,\bk)$, so that we have
\[
 \Hom_{\Db_{\varpi}(\pt,\bk)} \bigl( \ubk_\pt, \Res^{\Gm}_\varpi(\scF)[m] \bigr) = \Hom_{\Db(V_n / \varpi,\bk)}(\ubk_{V_n / \varpi}, (\pi_n)^*\scF_n[m]).
\]
To prove the lemma, it therefore suffices to prove that for any $\scG$ in $\Db(\mathbb{P}_\F^{n-1},\bk)$ we have a canonical isomorphism
\begin{multline}
\label{eqn:eq-cohom-Gm-pi}
 \Hom_{\Db(V_n / \varpi,\bk)}(\ubk_{V_n / \varpi}, (\pi_n)^*\scG) \cong \\
 \Hom_{\Db(\mathbb{P}_\F^{n-1},\bk)}(\ubk_{\mathbb{P}_\F^{n-1}}, \scG) \oplus \Hom_{\Db(\mathbb{P}_\F^{n-1},\bk)}(\ubk_{\mathbb{P}_\F^{n-1}}, \scG[-1]).
\end{multline}

We start by proving that
\begin{equation}
\label{eqn:isom-Euler-class}
 R(\pi_n)_* \ubk_{V_n / \varpi} = \ubk_{\mathbb{P}_\F^{n-1}} \oplus \ubk_{\mathbb{P}_\F^{n-1}}[-1].
\end{equation}
In fact, $V_n / \varpi$ is the complement of the zero section in the line bundle $\tilde{\pi}_n : \scO(\ell) \to \mathbb{P}_\F^{n-1}$. If $i_n : \mathbb{P}_\F^{n-1} \hookrightarrow \scO(\ell)$ is the embedding of the zero section, and $j_n : V_n / \varpi \hookrightarrow \scO(\ell)$ is the complementary open embedding, then we have a distinguished triangle
\[
 (i_n)_* \ubk_{\mathbb{P}_\F^{n-1}}[-2] \to \ubk_{\scO(\ell)} \to R(j_n)_* \ubk_{V_n / \varpi} \xrightarrow{[1]}.
\]
Applying the functor $R(\tilde{\pi}_n)_*$ we deduce a distinguished triangle
\[
 \ubk_{\mathbb{P}_\F^{n-1}}[-2] \to \ubk_{\mathbb{P}_\F^{n-1}} \to R(\pi_n)_* \ubk_{V_n / \varpi} \xrightarrow{[1]},
\]
in which the first map is (by definition) the (shift by $-2$ of the) Euler class of $\scO(\ell)$. Since $\bk$ has characteristic $\ell$ this Euler class vanishes, and we deduce the desired isomorphism~\eqref{eqn:isom-Euler-class}.

Next, we claim that for any $\scG$ in $\Db \Sh(\mathbb{P}_\F^{n-1},\bk)$ we have a canonical isomorphism
\begin{equation}
\label{eqn:isom-proj-formula}
 \scG \otimes_{\bk} R(\pi_n)_* \ubk_{V_n / \varpi} \simto R(\pi_n)_*(\pi_n)^* \scG.
\end{equation}
In fact adjunction provides a canonical morphism from the left-hand side to the right-hand side. To prove that this morphism is invertible it suffices to check this property after pullback under the surjective morphism $\pi_n$. However $\pi_n$ is a principal $\Gm/\varpi=\Gm$-bundle, so that we have a Cartesian diagram
\[
 \xymatrix{
 V_n / \varpi \times \Gm \ar[r] \ar[d] & V_n / \varpi \ar[d]^-{\pi_n} \\
 V_n / \varpi \ar[r]^-{\pi_n} & \mathbb{P}_\F^{n-1},
 }
\]
hence the claim follows from the smooth base change theorem~\cite[\href{https://stacks.math.columbia.edu/tag/0EYU}{Tag 0EYU}]{stacks-project} and the K\"unneth formula~\cite[\href{https://stacks.math.columbia.edu/tag/0F1N}{Tag 0F1N}]{stacks-project}.

Combining~\eqref{eqn:isom-Euler-class} and~\eqref{eqn:isom-proj-formula} we obtain, for any $\scG$ in $\Db \Sh(\mathbb{P}_\F^{n-1},\bk)$, an isomorphism
\[
 R(\pi_n)_*(\pi_n)^* \scG \cong \scG \oplus \scG[-1].
\]
In view of the isomorphism
\begin{multline*}
 \Hom_{\Db(V_n / \varpi,\bk)}(\ubk_{V_n / \varpi}, (\pi_n)^*\scG) = \Hom_{\Db(V_n / \varpi,\bk)}((\pi_n)^* \ubk_{\mathbb{P}_\F^{n-1}}, (\pi_n)^*\scG) \\
 \cong \Hom_{\Db(\mathbb{P}_\F^{n-1},\bk)}( \ubk_{\mathbb{P}_\F^{n-1}}, R(\pi_n)_*(\pi_n)^*\scG),
\end{multline*}
this implies~\eqref{eqn:eq-cohom-Gm-pi}, hence finishes the proof of the lemma.
\end{proof}

Using Lemma~\ref{lem:eq-cohom-Gm-pi} we will be able to give a more explicit description of the category $\Db_{\Gm,c}(\pt,\bk)_{\varpi\perf}$, as follows.

\begin{lem}
\label{lem:perf-fd-cohom}
 The full subcategory $\Db_{\Gm,c}(\pt,\bk)_{\varpi\perf} \subset \Db_{\Gm,c}(\pt,\bk)$ consists of the complexes $\scF$ such that
 \[
 \dim_\bk \bigl(\coH^\bullet_{\Gm}(\pt, \scF) \bigr) < \infty.
 \]
Moreover, in $\Smith(\pt,\bk)$ we have a canonical isomorphism of functors
 \[
 \id \cong [2].
 \]
\end{lem}

\begin{proof}
 Recall from~\eqref{eqn:eq-cohom-pt} that there exists a canonical morphism
 \[
  \mathrm{can}_{\pt}^2 : \ubk_\pt \to \ubk_\pt[2]
 \]
 in $\Db_{\Gm}(\pt,\bk)$.
%which corresponds to the canonical element $x \in \coH^2_{\Gm}(\pt,\bk)$. 
More explicitly, this morphism can be constructed as follows: consider the natural dilation action of $\Gm$ on $\mathbb{A}_\F^1$. If we denote by $i : \pt=\{0\} \hookrightarrow \mathbb{A}_\F^1$ the embedding, then we have $i^!(\ubk_{\mathbb{A}_\F^1}) = \ubk_{\pt}[-2]$. Using adjunction, we deduce a canonical map
\[
 i_*(\ubk_{\pt}[-2]) \to \ubk_{\mathbb{A}_\F^1}.
\]
(Note that we have ignored a Tate twist here; see Footnote~\ref{fn:Tate-twist}.)
Applying $i^*$ and shifting by $2$, we obtain the morphism $\mathrm{can}_{\pt}^2$. Since $(\mathbb{A}_\F^1)^\varpi=\{0\}$, from this description and Lemma~\ref{lem:crucial-Smith} we obtain that the cone $C$ of $\mathrm{can}_{\pt}^2$ belongs to $\Db_{\Gm,c}(\pt,\bk)_{\varpi\perf}$. In particular, since the tensor product of any bounded complex with a perfect complex is perfect, this implies that for any $\scF$ in $\Smith(\pt,\bk)$ we have a canonical isomorphism
\[
 \scF \simto \scF[2],
\]
providing the desired isomorphism of functors $\id \cong [2]$.

Now we claim that the triangulated subcategory $\langle C \rangle_{\Delta}$ of $\Db_{\Gm,c}(\pt,\bk)$ generated by $C$ is exactly the subcategory whose objects are the complexes $\scF$ such that
\[
\dim_\bk \bigl(\coH^\bullet_{\Gm}(\pt, \scF) \bigr) < \infty.
\]
%which will prove (in view of what we have already checked) that the latter subcategory is included in $\Db_{\Gm,c}(\pt,\bk)_{\varpi\perf}$. 
Indeed we have $\coH^\bullet_{\Gm}(\pt,C)=\bk[2]$, so that $C$ belongs
to this subcategory. To prove the opposite inclusion, we prove by
induction that for any $n \in \Z_{\geq 0}$, any complex $\scF$ such
that $\dim_\bk \bigl(\coH^\bullet_{\Gm}(\pt, \scF) \bigr)=n$ belongs
to $\langle C \rangle_{\Delta}$. In fact, if $n=0$ then using
Lemma~\ref{lem:eq-cohom-Gm-pi} and the fact that $\Db_{\varpi}(\pt,\bk)$ is generated (as a triangulated category) by $\underline{\bk}_\pt$ (see~\eqref{eqn:equv-Db-pt}) we see that any object $\scF$ such that
$\coH^\bullet_{\Gm}(\pt, \scF)=0$ satisfies
%\footnote{\color{red}GW: this wasn't immediately clear to me. It is if we use (3.6). Did you have something else in mind?}
$\Res^{\Gm}_\varpi(\scF)=0$, hence $\For_{\Gm}(\scF)=0$. From the definition, we see that this implies that $\scF=0$. Fix now $n \geq 0$, and assume the result is known for $n$. If $\dim_\bk \bigl(\coH^\bullet_{\Gm}(\pt, \scF) \bigr)=n+1$, and if $m$ is maximal such that
\[
 \coH^m_{\Gm}(\pt, \scF) \neq 0,
\]
then any choice of a nonzero vector in this space provides a morphism
\[
 \ubk_\pt[-m] \to \scF
\]
in $\Db_{\Gm}(\pt,\bk)$. By maximality the composition of this map with $\mathrm{can}_{\pt}^2[-m-2]$ vanishes, so that this map must factor through a morphism
\[
 C[-m-2] \to \scF.
\]
From the long exact sequence in equivariant cohomology we see that the cone $\scG$ of this map satisfies $\dim_\bk \bigl(\coH^\bullet_{\Gm}(\pt, \scG) \bigr)=n$, which allows to conclude by induction.

The two claims we have proved so far show that the subcategory with objects those complexes $\scF$ such that $\coH^\bullet_{\Gm}(\pt, \scF)$ is finite-dimensional is included in $\Db_{\Gm,c}(\pt,\bk)_{\varpi\perf}$. On the other hand, if $\scF$ is an object of $\Db_{\Gm,c}(\pt,\bk)$ such that $\Res^{\Gm}_{\varpi}(\scF)$ is perfect, then
\[
 \dim_\bk \left( \bigoplus_{n \in \Z} \Hom_{\Db_{\varpi}(\pt,\bk)}(\ubk_\pt, \Res^{\Gm}_\varpi(\scF)[n]) \right) < \infty.
\]
From Lemma~\ref{lem:eq-cohom-Gm-pi} we deduce that in this case $\coH^\bullet_{\Gm}(\pt, \scF)$ is finite-dimensional, which concludes the proof.
\end{proof}

\begin{rmk}
 Concretely, in terms of the equivalence~\eqref{eqn:equv-Db-pt}, the object $\Res^{\Gm}_\varpi(C)$ corresponds to the complex $(\cdots \to 0 \to \bk[\varpi] \xrightarrow{\varpi \cdot (-)} \bk[\varpi] \to 0 \to \cdots)$ where the nonzero terms are in degrees $-2$ and $-1$.
\end{rmk}

We can finally give the proof of Lemma~\ref{lem:Hom-Smith-pt}.

\begin{proof}[Proof of Lemma~\ref{lem:Hom-Smith-pt}]
Lemma~\ref{lem:perf-fd-cohom} shows in particular that $\ubk_{\pt}$ does not belong to $\Db_{\Gm,c}(\pt,\bk)_{\varpi\perf}$, hence has nonzero image in $\Smith(\pt,\bk)$. In view of the isomorphism $\id \cong [2]$, this shows that $\Hom_{\Smith(\pt,\bk)}(\underline{\bk}_{\pt}, \underline{\bk}_{\pt}[n]) \neq 0$ for any even $n$. Hence to conclude it only remains to prove that
\[
 \dim \Hom_{\Smith(\pt,\bk)}(\underline{\bk}_{\pt}, \underline{\bk}_{\pt}[n]) \leq \begin{cases}
                                                               1 & \text{if $n$ is even;}\\
                                                               0 & \text{otherwise.}
                                                              \end{cases}
\]

 A morphism $a$ from $\ubk_{\pt}$ to $\ubk_{\pt}[n]$ in $\Smith(\pt,\bk)$ is represented by a diagram
 \[
  \ubk_{\pt} \xleftarrow{f} \scF \xrightarrow{g} \ubk_{\pt}[n]
 \]
in which $\scF$ belongs to $\Db_{\Gm,c}(\pt,\bk)$, $f$ and $g$ are morphisms in $\Db_{\Gm,c}(\pt,\bk)$, and the cone of $f$ belongs to $\Db_{\Gm,c}(\pt,\bk)_{\varpi\perf}$, i.e.~has finite-dimensional equivariant cohomology (see Lemma~\ref{lem:perf-fd-cohom}). In particular, from the long exact sequence in equivariant cohomology and~\eqref{eqn:eq-cohom-pt} we obtain that there exists $N \in 2\Z$ (which, for later use, we will assume to be at least $-n$) such that for $m \geq N$ we have
\[
 \coH^m_{\Gm}(\pt,\scF)=\begin{cases}
                     \bk & \text{if $m$ is even;}\\
                     0 & \text{otherwise.}
                    \end{cases}
\]
If we choose a nonzero element in $\coH^N_{\Gm}(\pt,\scF)$, considered as a morphism $h: \ubk_{\pt}[-N] \to \scF$ in $\Db_{\Gm,c}(\pt,\bk)$, then the cone of $h$ has finite-dimensional equivariant cohomology, i.e.~belongs to $\Db_{\Gm,c}(\pt,\bk)_{\varpi\perf}$ by Lemma~\ref{lem:perf-fd-cohom}. As a consequence, $a$ can also be represented by the diagram
\[
  \ubk_{\pt} \xleftarrow{f \circ h} \ubk_{\pt}[-N] \xrightarrow{g \circ h} \ubk_{\pt}[n].
 \]
 In case $n$ is odd we have
 \[
  \Hom_{\Db_{\Gm,c}(\pt,\bk)}(\ubk_{\pt}[-N], \ubk_{\pt}[n])=0,
 \]
so that $g \circ h$ must be zero, which finishes the proof in this case.

On the other hand, if $n$ is even both spaces $\Hom_{\Db_{\Gm,c}(\pt,\bk)}(\ubk_{\pt}[-N], \ubk_{\pt})$ and $\Hom_{\Db_{\Gm,c}(\pt,\bk)}(\ubk_{\pt}[-N], \ubk_{\pt}[n])$ are $1$-dimensional, with a basis given by $\mathrm{can}_{\pt}^{N}[-N]$ and $\mathrm{can}_{\pt}^{n+N}[-N]$ respectively. Hence to conclude, it only remains to prove that for $M,M' \geq -n$ even, the diagrams
\[
 \ubk_{\pt} \xleftarrow{\mathrm{can}_{\pt}^{M}[-M]} \ubk_{\pt}[-M] \xrightarrow{\mathrm{can}_{\pt}^{n+M}[-M]} \ubk_{\pt}[n]
\]
and
\[
 \ubk_{\pt} \xleftarrow{\mathrm{can}_{\pt}^{M'}[-M']} \ubk_{\pt}[-M'] \xrightarrow{\mathrm{can}_{\pt}^{n+M'}[-M']} \ubk_{\pt}[n]
\]
represent the same morphism in $\Smith(\pt,\bk)$. However we can assume that $M' \geq M$; then the morphism $\mathrm{can}^{M'-M}_{\pt}[-M'] : \ubk_{\pt}[-M'] \to \ubk_{\pt}[-M]$ has a cone which belongs to $\Db_{\Gm,c}(\pt,\bk)_{\varpi\perf}$, and satisfies
\begin{gather*}
 \mathrm{can}_{\pt}^{M'}[-M'] = (\mathrm{can}_{\pt}^{M}[-M]) \circ (\mathrm{can}^{M'-M}_{\pt}[-M']), \\
 \mathrm{can}_{\pt}^{n+M'}[-M'] = (\mathrm{can}_{\pt}^{n+M}[-M]) \circ (\mathrm{can}^{M'-M}_{\pt}[-M']).
\end{gather*}
The desired claim follows.
\end{proof}

\section{Fixed points of roots of unity on the affine Grassmannian}
\label{sec:fixed-points-Gr}
%%%%%%%%%%%%%%%%%%%%%%%%%%%%%%%%%%%%%%%%%%%%%%%%%%%%%%

As in Section~\ref{sec:Smith-etale} we let $\F$ be an algebraically closed field of characteristic $p>0$.

%---------------------------------------------
\subsection{Affine Weyl group}
\label{ss:Waff}
%---------------------------------------------

Let $G$ be a connected reductive algebraic group over $\F$, and choose a Borel subgroup $B \subset G$ and a maximal torus $T \subset B$. The Weyl group of $(G,T)$ will be denoted $\Wf$. (Here, the subscript stands for ``finite,'' and is here to avoid any confusion with the affine Weyl group introduced below.) We will also denote by $U$ the unipotent radical of $B$, by $B^+$ the Borel subgroup opposite to $B$ with respect to $T$, and by $U^+$ the unipotent radical of $B^+$.

 We will denote by $\bX:=X^*(T)$ the character lattice of $T$, and by $\bX^\vee:=X_*(T)$ its cocharacter lattice. Let $\fR \subset \bX$ be the root system of $(G,T)$, and let $\fR^+ \subset \fR$ be the system of positive roots consisting of the $T$-weights in $\mathrm{Lie}(U^+)$. Let also $\fR^{\mathrm{s}}$ be the associated basis of $\fR$ (the ``simple roots''). These data define a set $\Sf$ of Coxeter generators for $\Wf$, consisting of the reflections $s_\alpha$ with $\alpha \in \fR^{\mathrm{s}}$.

%The \emph{extended affine Weyl group} is the semi-direct product
%\[
%\Wext := \Wf \ltimes \bX^\vee.
%\]
%For $\lambda \in \bX^\vee$ we will denote by $\st_\lambda$ the image of $\lambda$ in $\Wext$. Then, setting
%\[
%\ell(w\st_\lambda) = \sum_{\substack{\alpha \in \fR^+ \\ w(\alpha) \in \fR^+}} |\langle \lambda, \alpha \rangle| + \sum_{\substack{\alpha \in \fR^+ \\ w(\alpha) \in -\fR^+}} |1+\langle \lambda, \alpha \rangle|
%\]
%for $w \in \Wf$ and $\lambda \in \bX^\vee$, we obtain a function on $\Wext$ whose restriction to $\Waff:=\Wf \ltimes \Z\fR^\vee$ is the length function for a Coxeter group structure on $\Waff$. (Here, $\fR^\vee \subset \bX^\vee$ is the coroot system, and $\Z\fR^\vee$ is the coroot lattice.) We will denote by $\Saff \subset \Waff$ the subset of simple reflections, i.e.~of elements of length $1$. It consists of $\Sf$ together with the elements $\st_{\beta^\vee} s_\beta$ with $\beta$ a maximal root in $\mathfrak{R}$.
%
%Setting $\Omega := \{w \in \Wext \mid \ell(w)=0\}$, the conjugation action of $\Omega$ on $\Wext$ stabilizes $\Waff$, and induces Coxeter group automorphisms of this subgroup. Moreover, multiplication induces an isomorphism
%\[
%\Omega \ltimes \Waff \simto \Wext
%\]
%and for any $\omega \in \Omega$ and $w \in \Waff$ we have $\ell(w\omega)=\ell(\omega w)=\ell(w)$.

The \emph{affine Weyl group} is the semi-direct product
\[
\Waff := \Wf \ltimes \Z\fR^\vee,
\]
where $\fR^\vee \subset \bX^\vee$ is the coroot system, and $\Z\fR^\vee$ is the coroot lattice.
For $\lambda \in \Z\fR^\vee$ we will denote by $\st_\lambda$ the image of $\lambda$ in $\Waff$. The group $\Waff$ admits a natural structure of Coxeter group extending that of $\Wf$; the corresponding simple reflections $\Saff \subset \Waff$ consist of $\Sf$ together with the elements $\st_{\beta^\vee} s_\beta$ with $\beta$ a maximal root in $\mathfrak{R}$.

Given $n \in \Z \smallsetminus \{0\}$, we will consider two actions of $\Waff$ on $V:=\bX^\vee \otimes_\Z \R$ defined, for $w \in \Wf$ and $\lambda \in \Z\fR^\vee$, by
\[
(w\st_\lambda) \cdot_n \mu = w(\mu - n\lambda), \qquad (w\st_\lambda) \square_n \mu = w(\mu + n\lambda)
\]
for $\mu \in \bX^\vee$, where in the right-hand side we consider the natural action of $\Wf$ on $\bX^\vee$. 
(Here the action $\cdot_n$ appears due to the sign conventions in Bruhat--Tits theory; but the action $\square_n$ is closer to the action which will be relevant when considering Representation Theory.) Of course, these actions are related via
\[
 w \square_n \mu = -(w \cdot_n (-\mu))
\] 
for any $w \in \Waff$ and $\mu \in V$.

We set
\[
 \mathbf{a}_n := \{ \lambda \in V \mid \forall \alpha \in \fR^+, \, -n < \langle \lambda, \alpha \rangle < 0\}.
\]
Then the closure $\overline{\mathbf{a}}_n$ of $\mathbf{a}_n$ is a fundamental
domain for the action of $\Waff$ on $V$ via $\cdot_n$ and via $\square_n$. These actions stabilize $\bX^\vee$, and a fundamental domain for the action of $\Waff$ on $\bX^\vee$ (for each of these actions) is therefore $\overline{\mathbf{a}}_n \cap \bX^\vee$. (Compared to the conventions used in~\cite{dchl}, our ``fundamental alcove'' $\mathbf{a}_n$ is opposite. This is related to the fact that our fixed Borel subgroup $B$ is chosen to be \emph{negative}, while in~\cite{dchl} it is chosen to be positive.)

% The second action, denoted $\bullet_n$, is obtained by setting, for $w \in \Wext$ and $\mu \in \bX^\vee$,
% \[
% w \bullet_n \mu = w \cdot_n (\mu + \rho^\vee) - \rho^\vee,
% \]
% where $\rho^\vee$ is the halfsum of the positive coroots. (Here $\rho^\vee$ might not belong to $\bX^\vee$, but it is well known that $w \bullet_n \mu$ \emph{does} always belong to $\bX^\vee$.)
% It is easily seen 
% %that this action does not depend on the choice of $\varsigma$, and 
% that a fundamental domain for this action of $\Waff$ is given by
% \[
% A^\bullet_n := \{ \lambda \in \bX^\vee \mid \forall \alpha \in \fR^+, \, 0 \leq \langle \lambda+\rho^\vee, \alpha \rangle \leq n\}.
% \]
% (Note that $A_n^\bullet$ 
% %always contains the coweight $-\varsigma$, whose orbit is $-\varsigma + n \bX^\vee \subset \bX^\vee$. It 
% might contain $0$ or not, depending on the value of $n$.)
% 
% Often we will assume that there exists (and fix) an element $\varsigma \in \bX^\vee$ such that $\langle \varsigma, \alpha \rangle =1$ for all $\alpha \in \mathfrak{R}^{\mathrm{s}}$. Then we also have $w \bullet_n \mu = w \cdot_n (\mu + \varsigma) - \varsigma$ for any $w \in \Wext$ and $\mu \in \bX^\vee$, and $A^\bullet_n$ contains the coweight $-\varsigma$, whose orbit is $-\varsigma + n \bX^\vee \subset \bX^\vee$.

The \emph{affine roots} are the formal linear combinations $\alpha + m \hbar$ with $\alpha \in \mathfrak{R}$ and $m \in \Z$. To such a combination we attach an affine function $f^n_{\alpha+m\hbar}$ on $V$, determined by
\[
 f^n_{\alpha + m \hbar}(v) = \langle \alpha,v \rangle + nm,
\]
and an element $s_{\alpha + m \hbar} \in \Waff$ determined by
\[
 s_{\alpha + m \hbar} = \st_{m\alpha^\vee} s_\alpha.
\]
We then have
\[
 s_{\alpha + m \hbar} \cdot_n v = v - f^n_{\alpha + m \hbar}(v) \alpha^\vee
\]
for any $v \in V$.

\begin{rmk}
 In practice, when considering these constructions in later sections, the integer $n$ will be either $1$ or a prime number different from $p$. As this assumption does not simplify the discussion in any way, we will not impose any restriction on $n$ in this section.
\end{rmk}

%---------------------------------------------
\subsection{Some Bruhat--Tits theory}
%---------------------------------------------

For any positive integer $n$, we set $\scK_n:=\F ( \hspace{-1pt} ( z^n ) \hspace{-1pt} )$. We will consider $\scK:=\scK_1$ as a valued field with its natural valuation (such that $z$ has valuation $1$), and endow each $\scK_n$ with the valuation obtained by restriction. (In this way, all the fields $\scK_n$ are canonically isomorphic, but their valuations differ.) We will denote by $\scO_n$ the valuation ring of $\scK_n$, so that $\scO_n:=\F [ \hspace{-1pt} [ z^n ] \hspace{-1pt} ]$. For any $\lambda \in \bX^\vee$ we have a point $z^\lambda \in G(\scK)$, defined as the image of $z$ under the map $(\scK)^\times \to G(\scK)$ induced by $\lambda$. If $\lambda \in n \bX^\vee$, then $z^{\lambda}$ belongs to $G(\scK_n)$.

The group scheme $G \times_{\Spec(\F)} \Spec(\scK_n)$ is a (split)
connected reductive group scheme over $\scK_n$, so that one can
consider the associated ``enlarged'' Bruhat--Tits
 building $\mathfrak{B}_n$ (in the sense considered e.g.~in~\cite{prasad}) which carries an action of $G(\scK_n)$. Our choice of maximal torus in $G$ provides a split maximal torus $T \times_{\Spec(\F)} \Spec(\scK_n) \subset G \times_{\Spec(\F)} \Spec(\scK_n)$, which itself defines an apartment $A_n$ in $\mathfrak{B}_n$. This apartment is an affine space with underlying vector space $V$, and it is stable under the action of $N_G(T)(\scK_n)$. The vectorial part of the action of $N_G(T)(\scK_n)$ on $A_n$ factors through the natural action of $N_G(T)(\scK_n)/T(\scK_n) = \Wf$ on $V$, and for $\lambda \in \bX^\vee$ the element $z^{n\lambda} \in T(\scK_n)$ acts by translation by $-n\lambda$ (see~\cite[\S 1.8]{prasad}). Let us choose, for any $w \in \Wf$, a lift $\dot{w}$ of $w$ in $N_G(T)$. Then we will consider the map
\[
 \iota_n : \Waff \to N_G(T)(\scK_n)
\]
defined by $\iota_n(\st_\lambda w)=z^{n\lambda} \dot{w}$ for $w \in \Wf$ and $\lambda \in \Z\fR^\vee$. 

If we choose a $\Wf$-fixed point in $A_n$, then the action of $V$ on this point defines an identification 
\begin{equation}
\label{eqn:apartment-V}
V \simto A_n,
\end{equation}
under which the action of $N_G(T)(\scK_n)$ on $A_n$ identifies with the action of $\Waff$ on $V$ provided by $\cdot_n$. We will fix such an identification once and for all, and use it to identify all the data considered above related to $V$ as data related to $A_n$. (None of our considerations below will depend on the choice of identification~\eqref{eqn:apartment-V}.)
%to identify each $f^n_{\alpha + m \hbar}$ with an affine function on $A_n$. (These functions will not depend on the choice of identification~\eqref{eqn:apartment-V}.)

The collection of fixed points of the reflections $s_{\alpha + m \hbar}$ (or in other words of kernels of the functions $f^n_{\alpha + m \hbar}$) defines a hyperplane arrangement in $A_n$, hence a collection of \emph{facets} (see~\cite[Chap.~V, \S 1.2]{bourbaki}). In particular, 
%the image 
$\mathbf{a}_n$
% of $\mathbf{a}_n^0$ under any isomorphism as in~\eqref{eqn:apartment-V} 
is a facet of maximal dimension, i.e.~an alcove. 
%(This image does not depend on the choice of identification.) 
Another example of facet is the intersection $\mathbf{o}_n$ of the reflection hyperplanes associated with all the reflections $s_\beta$ with $\beta \in \mathfrak{R}$, i.e.~the set of $\Wf$-fixed points. The facets we will mainly be interested in are those contained in the closure $\overline{\mathbf{a}}_n$ of $\mathbf{a}_n$.

To any facet $\mathbf{f}$ in $A_n$, Bruhat--Tits theory
associates a ``parahoric group scheme'' $P_{\mathbf{f}}$ over $\Spec(\scO_n)$, a smooth affine group scheme over $\Spec(\scO_n)$ with connected geometric fibers such that
\begin{equation}
\label{eqn:parahoric-generic}
 P_{\mathbf{f}} \times_{\Spec(\scO_n)} \Spec(\scK_n) = G \times_{\Spec(\F)} \Spec(\scK_n),
\end{equation}
and whose $\scO_n$-points are of finite index in the pointwise stabilizer of $\mathbf{f}$ in $G(\scK_n)$. (This group scheme is denoted similarly in~\cite{dchl}; in~\cite{prasad} it is denoted $\mathscr{G}_{\mathbf{f}}^\circ$.) In particular we have
\[
 P_{\mathbf{o}_n} = G \times_{\Spec(\F)} \Spec(\scO_n),
\]
and $P_{\mathbf{a}_n}$ is an Iwahori group scheme, whose group of $\scO_n$-points is the inverse image 
%$\Iw_n$ 
of $B$ under the map $G(\scO_n) \to G$ of evaluation at $z^n=0$. This construction is compatible with inclusions of closures of facets in a natural way; in particular for any facet $\mathbf{f}$ contained in $\overline{\mathbf{a}}_n$ we have a closed embedding
\begin{equation}
\label{eqn:inclusion-parahoric}
 P_{\mathbf{a}_n} \subset P_{\mathbf{f}}.
\end{equation}

%------------------------------------------------
\subsection{Loop groups and partial affine flag varieties}
\label{ss:loop-flag}
%------------------------------------------------

As above we fix a positive integer $n$.
The $n$-th \emph{loop group} associated with $G$ is the ind-affine group ind-scheme $L_n G$ over $\F$ which represents the functor sending an $\F$-algebra $R$ to $G(R ( \hspace{-1pt} ( z^n ) \hspace{-1pt} ))$. The associated \emph{arc group} (or positive loop group) is the affine group scheme $L^+_n G$ over $\F$ which represents the functor sending $R$ to $G(R[ \hspace{-1pt} [ z^n ] \hspace{-1pt} ])$. (For basics about ind-schemes, see~\cite[\S 1]{richarz}.)

The case we are mostly interested in is when $n=1$. In this case (here and in later related notation), we will usually omit the subscript from the notation. The case of a more general $n$ however naturally appears when considering the action of $n$-th roots of unity by loop rotation. Namely, we have a natural action of the multiplicative group $\Gm$ over $\F$ on $L G$ by loop rotation. This action stabilizes the subgroup $L^+ G$. Denote now by $\roots_n \subset \Gm$ the subgroup scheme of $n$-th roots of unity; we can then consider the fixed-points ind-scheme $(L G)^{\roots_n}$ and the fixed-points scheme $(L^+ G)^{\roots_n}$ in the sense of~\S\ref{ss:stalks-fixed-points}.

\begin{lem}
\label{lem:fixed-points-loop-group}
 We have identifications
 \[
  (L G)^{\roots_n} = L_n G, \quad (L^+ G)^{\roots_n}=L_n^+ G.
 \]
\end{lem}

\begin{proof}
 For any $\F$-algebra $R$, the $R$-points $(L G)(R)$ consist of the $\F$-scheme morphisms $\Spec(R( \hspace{-1pt} ( z) \hspace{-1pt})) \to G$. Therefore, the $R$-points of $(L G)^{\roots_n}$ consist of the $\roots_n$-invariant morphisms $\Spec(R( \hspace{-1pt} ( z) \hspace{-1pt})) \to G$, i.e.~the morphisms which factor through the quotient $\Spec(R( \hspace{-1pt} ( z) \hspace{-1pt})) \to \Spec(R( \hspace{-1pt} ( z) \hspace{-1pt})) /\roots_n = \Spec(R( \hspace{-1pt} ( z^n) \hspace{-1pt}))$. This proves the first identification. The proof of the second one is similar.
\end{proof}

It is well known (see e.g.~\cite[\S 3]{richarz}) that the fppf sheafification of the functor
\[
 R \mapsto (L G)(R)/(L^+ G)(R)
\]
is represented by an ind-projective ind-scheme over $\F$, which is called the \emph{affine Grassmannian} of $G$, and will be denoted $\Gr_G$. The main goal of this section is to describe the ind-scheme $(\Gr_G)^{\roots_n}$, see Proposition~\ref{prop:fixed-points-Grass} below. This will require discussing more general ``partial affine flag varieties'' attached to $L_n G$, as follows.

If $\mathbf{f} \subset \overline{\mathbf{a}}_n$ is a facet, we can consider the affine group scheme $L^+_n P_{\mathbf{f}}$ over $\F$ which represents the functor sending $R$ to $P_{\mathbf{f}}(R[ \hspace{-1pt} [ z^n ] \hspace{-1pt} ])$. In view of~\eqref{eqn:parahoric-generic}, $L^+_n P_{\mathbf{f}}$ is a subgroup of $L_n G$. The partial affine flag variety 
\[
\Fl^n_{\mathbf{f}}
\]
associated with $\mathbf{f}$ is the ind-projective ind-scheme over $\F$ which represents the fppf sheafification of the presheaf
\[
 R \mapsto L_n G(R) / L_n^+ P_{\mathbf{f}}(R).
\]
These ind-schemes are the main object of study of~\cite{pappas-rapoport}. In particular, the connected components of $\Fl^n_{\mathbf{f}}$ are in a natural bijection with the algebraic fundamental group of $G$ (see~\cite[Theorem~0.1]{pappas-rapoport}); the component corresponding to the neutral element will be denoted $\Fl^{n,\circ}_{\mathbf{f}}$.

If $\alpha \in \mathfrak{R}$, we will denote by $U_\alpha$ the root subgroup of $G$ attached to $\alpha$. Then, for an affine root $\alpha + m \hbar$, we will denote by $U_{\alpha + m \hbar}$ the subgroup of $L G$ which, for any isomorphism $u_\alpha : \Ga \simto U_\alpha$, identifies with the image of the morphism $x \mapsto u_\alpha(xz^m)$. 
%If $m \in n\Z$, then $U_{\alpha + m\hbar}$ is a subgroup of $L_n G$.

The following statements are easily checked.

\begin{lem}
\label{lem:properties-root-subgroups}
Let $\alpha \in \mathfrak{R}$ and $m \in \Z$.
\begin{enumerate}
\item
\label{it:fixed-points-root-subgroups}
The subgroup $U_{\alpha+m\hbar}$ is stable under the action of $\roots_n$, and we have
\[
(U_{\alpha+m\hbar})^{\roots_n} = \begin{cases}
U_{\alpha+m\hbar} & \text{if $m \in n\Z$;} \\
\{1\} & \text{otherwise.}
\end{cases}
\]
\item
\label{it:conjugation-root-subgroups}
If $\lambda \in \bX^\vee$, we have
\begin{equation*}
%\label{eqn:conjugation-root-subgroups}
 z^{\lambda} \cdot U_{\alpha+m\hbar} \cdot z^{-\lambda} = U_{\alpha + (m+\langle \lambda,\alpha \rangle)\hbar}.
\end{equation*}
\item
If $\mathbf{f} \subset \overline{\mathbf{a}}_n$ is a facet, we have $U_{\alpha+mn\hbar} \subset L_n^+ P_{\mathbf{f}}$ iff $f^n_{\alpha+m\hbar}$ takes nonnegative values on $\mathbf{f}$.
\end{enumerate}
\end{lem}

%--------------------------------------------------------
\subsection{Fixed points on orbits}
\label{ss:orbits-Gr}
A crucial role in our discussion will be played by the following Iwahori subgroups of $L^+G$, for which we introduce special notation:
\[
\Iw := L_1^+ P_{\mathbf{a}_1}, \quad \Iw^+ := \dot{w}_0 \cdot \Iw \cdot (\dot{w}_0)^{-1}.
\]
(Here, $w_0$ is the longest element in $\Wf$.) More concretely, $\Iw$, resp.~$\Iw^+$,
is the inverse image of $B$, resp.~$B^+$, under the map $\mathrm{ev}_0 : L^+G \to G$ sending $z$ to $0$. We also denote by $\Iwu$ and $\Iwu^+$ the pro-unipotent radicals of $\Iw$ and $\Iw^+$, i.e.~the inverse images of $U$ and $U^+$ under $\mathrm{ev}_0$.

Lor $\lambda \in \bX^\vee$, we denote by $L_\lambda$ the coset of $z^{\lambda}$ in $\Gr_G$.
The group scheme $L^+G$ acts on $\Gr_G$ (see~\S\ref{ss:loop-flag}), and the orbits of this action are parametrized by the subsemigroup $\bX^\vee_+ \subset \bX^\vee$ of dominant cocharacters. More precisely, 
%each cocharacter $\lambda \in \bX^\vee$ determines a point $z^\lambda \in \GK$, whose image in $\Gr_G$ is denoted $L_\lambda$. Then 
we have
\begin{equation}
\label{eqn:orbits-Gr-GO}
(\Gr_G)_{\mathrm{red}} = \bigsqcup_{\lambda \in \bX^\vee_+} \Gr_G^\lambda \quad \text{with  $\Gr_G^\lambda := L^+G \cdot L_\lambda$,}
\end{equation}
where the left-hand side denotes the reduced ind-scheme associated with $\Gr_G$.
Moreover, for any $\lambda \in \bX^\vee_+$ the closure $\overline{\Gr_G^\lambda}$ is a projective $\F$-scheme, on which the action of $L^+G$ factors through an action of a smooth quotient group scheme of finite type.

The orbits of $\Iw$ and $\Iw^+$ can be described similarly: we have
\begin{equation}
\label{eqn:GrG-Iworbits}
(\Gr_G)_{\mathrm{red}} = \bigsqcup_{\lambda \in \bX^\vee} \Gr_{G,\lambda} \quad \text{with $\Gr_{G,\lambda} := \Iw \cdot L_\lambda$}
\end{equation}
and
\begin{equation}
\label{eqn:GrG-Iw+orbits}
(\Gr_G)_{\mathrm{red}} = \bigsqcup_{\lambda \in \bX^\vee} \Gr^+_{G,\lambda} \quad \text{with $\Gr^+_{G,\lambda} := \Iw^+ \cdot L_\lambda$.}
\end{equation}
Moreover, each $\Iw$-orbit (resp.~$\Iw^+$-orbit) is also an $\Iwu$-orbit (resp.~$\Iwu^+$-orbit), and for any $\mu \in \bX^\vee_+$ we have
\[
\Gr_G^\mu = \bigsqcup_{\lambda \in \Wf \mu} \Gr_{G,\lambda} = \bigsqcup_{\lambda \in \Wf \mu} \Gr^+_{G,\lambda}.
\]
For $\lambda \in \bX^\vee$, the embedding of $\Gr_{G,\lambda}^+$ in $\Gr_G$ will be denoted $j_\lambda^+$.

If $n \in \Z_{>0}$, we can also consider the
Iwahori subgroups $\Iw_n, \Iw_n^+ \subset L_n G$ defined as above with $z$ replaced by $z^n$, and their pro-unipotent radicals $\Iwun, \Iwun^+$.

\begin{lem}
\label{lem:fixed-points-orbits}
We have
%\begin{gather*}
%(\GO)^{\roots_n} = G_{\mathscr{O},n}, \quad 
\[
\Iw^{\roots_n}=\Iw_n, \quad (\Iw^+)^{\roots_n}=\Iw_n^+, \quad
(\Iwu)^{\roots_n}=\Iwun, \quad (\Iwu^+)^{\roots_n}=\Iwun^+.
\]
%\end{gather*}
For any $\lambda \in \bX^\vee$ we have
\[
(\Gr_{G,\lambda})^{\roots_n} = \Iw_n \cdot L_\lambda, \quad (\Gr^+_{G,\lambda})^{\roots_n} = \Iw^+_n \cdot L_\lambda.
\]
\end{lem}

\begin{proof}
%For any $\F$-algebra $R$, the $R$-points $\GO(R)$ consist of the $\F$-schemes morphisms $\Spec(R[ \hspace{-1pt} [ z] \hspace{-1pt}]) \to G$. Therefore, the $R$-points of $(\GO)^{\roots_n}$ consist of the $\roots_n$-invariant morphisms $\Spec(R[ \hspace{-1pt} [ z] \hspace{-1pt}]) \to G$, i.e.~the morphisms which factor through the quotient $\Spec(R[ \hspace{-1pt} [ z] \hspace{-1pt}]) \to \Spec(R[ \hspace{-1pt} [ z] \hspace{-1pt}]) /\roots_n = \Spec(R[ \hspace{-1pt} [ z^n] \hspace{-1pt}])$. This proves the first identification. The identifications of $\Iw^{\roots_n}$, $(\Iw^+)^{\roots_n}$, $(\Iwu)^{\roots_n}$ and $(\Iwu^+)^{\roots_n}$ follow.
The identifications in the first sentence are immediate consequences of Lem\-ma~\ref{lem:fixed-points-loop-group}.

For the description of $(\Gr_{G,\lambda})^{\roots_n}$, for any $\alpha \in \mathfrak{R}$ we set $\delta_\alpha=1$ if $\alpha \in \mathfrak{R}^+$, and $\delta_\alpha=0$ otherwise. 
%We also fix an isomorphism $u_\alpha$ from $\Ga$ to the root subgroup of $G$ associated with $\alpha$, and denote by $U_{\alpha,m} \subset \GK$ the image of the morphism $x \mapsto u_\alpha(x z^m)$. If 
Using the notation introduced in~\S\ref{ss:loop-flag}, we set
\[
\Iwu^\lambda := \prod_{\alpha \in \mathfrak{R}} \left( \prod_{\delta_\alpha \leq m < \langle \lambda,\alpha \rangle} U_{\alpha+m\hbar} \right),
\]
where the products are taken in any chosen order.
Then it is well known that the map $u \mapsto u \cdot L_\lambda$ induces an isomorphism
$\Iwu^\lambda \simto \Gr_{G,\lambda}$. Since $L_\lambda$ is stable under the $\roots_n$-action
we deduce an isomorphism $(\Iwu^\lambda)^{\roots_n} \simto (\Gr_{G,\lambda})^{\roots_n}$, and here by Lemma~\ref{lem:properties-root-subgroups}\eqref{it:fixed-points-root-subgroups} we have
\[
(\Iwu^\lambda)^{\roots_n} = \prod_{\alpha \in \mathfrak{R}} \left( \prod_{\substack{\delta_\alpha \leq m < \langle \lambda,\alpha \rangle \\ n \mid m}} U_{\alpha + m\hbar} \right).
\]
It follows that $(\Gr_{G,\lambda})^{\roots_n} = \Iw_n \cdot L_\lambda$, as desired.
The proof that $(\Gr_{G,\lambda}^+)^{\roots_n} = \Iw_n^+ \cdot L_\lambda$ is similar.
\end{proof}

%----------------------------------------------------------------
\subsection{Big cells in partial affine flag varieties}
\label{ss:big-cells}
%----------------------------------------------------------------

Our arguments below will make use of the ``big cell'' in $\Fl^n_{\mathbf{f}}$, whose construction we now recall following de Cataldo--Haines--Li~\cite{dchl}. We first consider the ind-affine group ind-scheme $L^{(-1)}_n G$ which represents the functor sending $R$ to the kernel of the morphism
\[
 G(R[z^{-n}]) \to G(R)
\]
of evaluation at $z^{-n}=0$. Then $L^{(-1)}_n G$ is a subgroup ind-scheme of $L_n G$, and we set
\[
 L^{--}_n P_{\mathbf{a}_n} = L^{(-1)}_n G \cdot U^+.
\]
With this definition, it is well known (see e.g.~\cite[\S 2]{faltings}) that the morphism
\[
 L^{--}_n P_{\mathbf{a}_n} \to \Fl^{n,\circ}_{\mathbf{a}_n}
\]
induced by the action on the base point is representable by an open immersion,
and that from this one can obtain an ``open cover'' of $\Fl^{n,\circ}_{\mathbf{a}_n}$ parametrized by $\Waff$, where the open subset corresponding to $w$ is the image of
\[
 \iota_n(w) \cdot L^{--}_n P_{\mathbf{a}_n} \cdot \iota_n(w)^{-1}
\]
under the map $g \mapsto g \cdot [\iota_n(w)]$. (Here, $[\iota_n(w)]$ is the image of $\iota_n(w)$ in $\Fl^{n,\circ}_{\mathbf{a}_n}$.)

For a general facet $\mathbf{f} \subset \overline{\mathbf{a}}_n$, we denote by $\Waff^{\mathbf{f}}$ the pointwise stabilizer of $\mathbf{f}$ in $\Waff$ (a finite parabolic subgroup) and set
\[
 L^{--}_n P_{\mathbf{f}} = \bigcap_{w \in \Waff^{\mathbf{f}}} \iota_n(w) \cdot L_n^{--} P_{\mathbf{a}_n} \cdot \iota_n(w)^{-1}.
\]
For instance, in the special case $\mathbf{f}=\mathbf{o}_n$, we have
\begin{equation}
\label{eqn:opposite-parahoric-o}
 L^{--}_n P_{\mathbf{o}_n} = L^{(-1)}_n G.
\end{equation}

The following claim is easily checked (see~\cite[Proposition~3.6.4]{dchl}).

\begin{lem}
\label{lem:root-subgroups-L-}
For $\alpha \in \mathfrak{R}$ and $m \in \Z$, we have $U_{\alpha+nm\hbar} \subset L^{--}_n P_{\mathbf{f}}$ iff $f^n_{\alpha+m\hbar}$ takes negative values on $\mathbf{f}$.
\end{lem}

With this definition, as explained in~\cite[\S 3.8.1]{dchl}, the morphism
\[
 L^{--}_n P_{\mathbf{f}} \to \Fl^{n,\circ}_{\mathbf{f}}
\]
defined by the action on the base point is representable by an open immersion.
One can obtain from this an ``open cover'' of $\Fl^{n,\circ}_{\mathbf{f}}$ parametrized by the quotient $\Waff / \Waff^{\mathbf{f}}$, where the open subset attached to a coset $w \Waff^{\mathbf{f}}$ is the image of the subgroup 
\[
 \iota_n(w) \cdot L^{--}_n P_{\mathbf{f}} \cdot \iota_n(w)^{-1}
\]
under the morphism of action on the image of $\iota_n(w)$. (These data do not depend on the choice of $w$ in its coset, and this claim can be deduced from the corresponding fact for $\mathbf{a}_n$ by using the morphism $\Fl^{n,\circ}_{\mathbf{a}_n} \to \Fl^{n,\circ}_{\mathbf{f}}$ induced by~\eqref{eqn:inclusion-parahoric}.)

For $m \in \Z_{\geq 1}$, we will denote by $L^{(-1)} G(m)$, resp.~$L_n^{(-1)} G(m)$, the subgroup of $L^{(-1)}G$, resp.~$L_n^{(-1)} G$, which represents the functor sending $R$ to the preimage of $T(R[z^{-1}]/z^{-m})$ under the composition
 \[
  L^{(-1)} G(R) \hookrightarrow G(R[z^{-1}]) \to G(R[z^{-1}]/z^{-m}),
 \]
resp.~the preimage of $T(R[z^{-n}]/z^{-mn})$ under the composition
 \[
  L_n^{(-1)} G(R) \hookrightarrow G(R[z^{-n}]) \to G(R[z^{-n}]/z^{-nm}).
 \]
Below we will require the following properties of these subgroups:
\begin{enumerate}
\item
\label{it:property-1}
for fixed $\lambda \in \bX^\vee$ and $m \in \Z_{>0}$, for $m' \gg 0$ we have $z^{\lambda} L^{(-1)} G(m') z^{-\lambda} \subset L^{(-1)} G(m)$;
\item
\label{it:property-2}
for any facet $\mathbf{f} \subset \overline{\mathbf{a}}_n$, for $m \gg 0$ we have $L_n^{(-1)}G(m) \subset L_n^{--} P_{\mathbf{f}}$.
\end{enumerate}
%Here,~\eqref{it:property-2} follows from the fact that for any given $w \in \Waff$, for $m \gg 0$ we have $L_n^{(-1)}G(m) \subset \iota_n(w) \cdot L_n^{--} P_{\mathbf{a}_n} \cdot \iota_n(w)^{-1}$, see~\cite{dchl}\footnote{\color{red}GW: more precise ref?}. 
(For~\eqref{it:property-1}, we can use~\cite[Remark~3.1.1]{dchl} to reduce the claim to the case $G=\mathrm{GL}_n(\F)$, which is clear from a matrix calculation. This property implies that for any fixed $w \in \Waff$, for $m \gg 0$ we have $L_n^{(-1)}G(m) \subset \iota_n(w) \cdot L_n^{(-1)}G \cdot \iota_n(w)^{-1} \subset \iota_n(w) \cdot L_n^{--} P_{\mathbf{a}_n} \cdot \iota_n(w)^{-1}$, which implies~\eqref{it:property-2}.)

For $\lambda \in (-\overline{\mathbf{a}}_n) \cap \bX^\vee$, we will denote by 
\begin{equation}
\label{eqn:def-flambda}
\mathbf{f}_\lambda \subset \overline{\mathbf{a}}_n
\end{equation}
the facet containing $-\lambda$.

\begin{lem}
\label{lem:fixed-points-stab}
 For any $\lambda \in (-\overline{\mathbf{a}}_n) \cap \bX^\vee$, we have
 \[
  (z^\lambda \cdot L^{--}_1 P_{\mathbf{o}_1} \cdot z^{-\lambda})^{\roots_n} = L^{--}_n P_{\mathbf{f}_\lambda}.
 \]
%where $\mathbf{f}_\lambda \subset \overline{\mathbf{a}}_n$ is the facet containing $-\lambda$.
\end{lem}

\begin{proof} 
 For any $\alpha \in \fR$, let us denote by $i_\alpha$ the largest integer such that the function $f^{n}_{\alpha+i_\alpha \hbar}$ takes negative values on $\mathbf{f}_\lambda$. In fact, since $\mathbf{f}_\lambda \subset \overline{\mathbf{a}}_n$ we can describe this integer very explicitly:
 \begin{itemize}
 \item
  if $\alpha \in \fR^+$ then $\langle \lambda, \alpha \rangle \in \{0, \cdots, n\}$, and
 \begin{equation}
 \label{eqn:ialpha+}
 i_\alpha = \begin{cases}
 0 & \text{if $\langle \lambda,\alpha \rangle >0$;}\\
 -1 & \text{if $\langle \lambda,\alpha \rangle =0$;}
 \end{cases}
 \end{equation}
  \item
  if $\alpha \in -\fR^+$ then $\langle \lambda, \alpha \rangle \in \{-n, \cdots, 0\}$, and
  \begin{equation}
 \label{eqn:ialpha-}
 i_\alpha = \begin{cases}
 -1 & \text{if $\langle \lambda,\alpha \rangle >-n$;}\\
 -2 & \text{if $\langle \lambda,\alpha \rangle =-n$.}
 \end{cases}
 \end{equation}
 \end{itemize}
 
 Recall from~\eqref{it:property-1}--\eqref{it:property-2} above that we can choose $m$ large enough such that
 \[
 L_n^{(-1)} G(m) \subset L^{--}_n P_{\mathbf{f}_\lambda} \quad \text{and} \quad z^{-\lambda} \cdot L_n^{(-1)} G(m) \cdot z^{\lambda} \subset L^{(-1)} G.
 \]
 Then as in~\cite[Proposition~3.6.4]{dchl} we have a direct product decomposition
 \[
  L^{--}_n P_{\mathbf{f}_\lambda} = L_n^{(-1)} G(m) \cdot \prod_{\alpha \in \mathfrak{R}^+} \prod_{j=-m}^{i_\alpha} U_{\alpha+jn\hbar} \cdot \prod_{\alpha \in -\mathfrak{R}^+} \prod_{j=-m}^{i_\alpha} U_{\alpha+jn\hbar}.
 \]
 (Here we use arbitrary orders on $\mathfrak{R}^+$ and on $-\mathfrak{R}^+$.) Our choice of $m$ guarantees that
$z^{-\lambda} \cdot L_n^{(-1)} G(m) \cdot z^{\lambda} \subset L^{(-1)} G$, and 
%using~\eqref{eqn:conjugation-root-subgroups} and the fact that $i_\alpha=-1$ for $\alpha \in -\mathfrak{R}^+$ we see that for any affine root $\alpha+jn\hbar$ appearing in the decomposition above we have
in view of Lemma~\ref{lem:properties-root-subgroups}\eqref{it:conjugation-root-subgroups}, for any affine root $\alpha+jn\hbar$ appearing in the decomposition above, the fact that $f^n_{\alpha+j\hbar}(-\lambda)<0$ implies that
\[
z^{-\lambda} \cdot U_{\alpha+jn\hbar} \cdot z^{\lambda} \subset L^{(-1)}G.
\]
These considerations show that
\[
 L^{--}_n P_{\mathbf{f}_\lambda} \subset z^\lambda \cdot L^{--}_1 P_{\mathbf{o}_1} \cdot z^{-\lambda}
\]
(see~\eqref{eqn:opposite-parahoric-o}),
so that
\[
 L^{--}_n P_{\mathbf{f}_\lambda} \subset (z^\lambda \cdot L^{--}_1 P_{\mathbf{o}_1} \cdot z^{-\lambda})^{\roots_n}.
\]

To prove the reverse inclusion, we continue with some $m$ as above, and choose $m' \gg 0$ such that
\[
 z^{\lambda} L^{(-1)} G (m') z^{-\lambda} \subset L^{(-1)} G (nm)
\]
(see~\eqref{it:property-1} above). We then have
\[
 (z^{\lambda} L^{(-1)} G (m') z^{-\lambda})^{\roots_n} \subset (L^{(-1)} G (nm))^{\roots_n}=L_n^{(-1)}G(m).
\]
As above we have a direct product decomposition
\[
 L^{(-1)} G = L^{(-1)} G (m') \cdot \prod_{\alpha \in \mathfrak{R}^+} \prod_{j=-m'}^{-1} U_{\alpha+j\hbar} \cdot \prod_{\alpha \in -\mathfrak{R}^+} \prod_{j=-m'}^{-1} U_{\alpha+j\hbar},
\]
where now (for notational convenience) we choose the order on $\mathfrak{R}^+$ such that all the roots such that $\langle \lambda,\alpha \rangle=0$ are bigger than the other ones, and the order on $-\mathfrak{R}^+$ such that all the roots such that $\langle \lambda,\alpha \rangle = -n$ are bigger than the other ones. From this decomposition we see that $ (z^{\lambda} \cdot L^{(-1)} G \cdot z^{-\lambda})^{\roots_n}$ is the product of $(z^{\lambda} L^{(-1)} G (m') z^{-\lambda})^{\roots_n}$, which is included in $L^{--}_n P_{\mathbf{f}_\lambda}$ by the choices of $m$ and $m'$, and of
\[
\left( \prod_{\alpha \in \mathfrak{R}^+} \prod_{j=-m'+\langle \lambda,\alpha\rangle}^{-1+\langle \lambda,\alpha\rangle} U_{\alpha+j\hbar} \cdot \prod_{\alpha \in -\mathfrak{R}^+} \prod_{j=-m'+\langle \lambda,\alpha\rangle}^{-1+\langle \lambda,\alpha\rangle} U_{\alpha+j\hbar} \right)^{\roots_n}, 
\]
which by Lemma~\ref{lem:properties-root-subgroups}\eqref{it:fixed-points-root-subgroups} is included in
\begin{multline*}
\prod_{\substack{\alpha \in \mathfrak{R}^+ \\ \langle \lambda,\alpha \rangle > 0}} \prod_{j=-N}^0 U_{\alpha+jn\hbar} \cdot \prod_{\substack{\alpha \in \mathfrak{R}^+ \\ \langle \lambda,\alpha \rangle = 0}} \prod_{j=-N}^{-1} U_{\alpha+jn\hbar} \cdot \\
\prod_{\substack{\alpha \in -\mathfrak{R}^+ \\ \langle \lambda,\alpha \rangle > -n}} \prod_{j=-N}^{-1} U_{\alpha+jn\hbar} \cdot \prod_{\substack{\alpha \in -\mathfrak{R}^+ \\ \langle \lambda,\alpha \rangle = -n}} \prod_{j=-N}^{-2} U_{\alpha+jn\hbar}
\end{multline*}
for $N \gg 0$. Here all the affine root subgroups are included in $L^{--}_n P_{\mathbf{f}_\lambda}$ by~\eqref{eqn:ialpha+}--\eqref{eqn:ialpha-} and Lemma~\ref{lem:root-subgroups-L-}, which finally proves that
\[
 (z^{\lambda} \cdot L^{(-1)} G \cdot z^{-\lambda})^{\roots_n} \subset L^{--}_n P_{\mathbf{f}_\lambda}
\]
and concludes the proof.
\end{proof}

% For the proof of Proposition~\ref{prop:fixed-points-Grass} we will need more explicit information on the subgroups $L^{--}_n P_{\mathbf{f}}$. First we consider the case $n=1$ and $\mathbf{f}=\mathbf{o}_1$; in this case we have
% \[
%  L^{--}_1 P_{\mathbf{o}_1} = L_1^{(-1)} G.
% \]
% If $\alpha \in \mathfrak{R}$, we will denote by $U_\alpha$ the root subgroup of $G$ attached to $\alpha$. Then, for an affine root $\alpha + m \hbar$, we will denote by $U_{\alpha + m \hbar}$ the subgroup of $L_1 G$ which, for any isomorphism $u_\alpha : \Ga \simto U_\alpha$, identifies with the image of the morphism $x \mapsto u_\alpha(xz^m)$. If $m \in n\Z$, then $U_{\alpha + m\hbar}$ is a subgroup of $L_n G$.

%------------------------------------------------
\subsection{Fixed points on the affine Grassmannian}
\label{ss:Gr-fixed-points}
%------------------------------------------------

% Among the affine flag varieties introduced in~\S\ref{ss:loop-flag}, the one we are mostly interested in is the \emph{affine Grassmannian}
% \[
%  \Gr_G := \Fl^1_{\mathbf{o}_1},
% \]
% which represents the functor $R \mapsto G(R ( \hspace{-1pt} ( z ) \hspace{-1pt} ))/G(R[ \hspace{-1pt} [ z ] \hspace{-1pt} ])$. For any $\lambda \in \bX^\vee$, we will denote by $z^\lambda$ the image of $z$ under the morphism $\Gm(\scK) \to T(\scK)$ induced by $\lambda$, and by $L_\lambda$ the associated point of $\Gr_G$.
% 
% The more general partial affine flag varieties naturally appear in this setting when considering the action of the $\F$-group scheme $\roots_n$ of $n$-th roots of unity obtained by restricting the action of the multiplicative group $\Gm$ on $LG$, and therefore also $\Gr_G$, by loop rotation.
% More precisely, 

Fix $\lambda \in (-\overline{\mathbf{a}}_n) \cap \bX^\vee$, and consider the morphism from $L_n G$ to $\Gr_G$ defined by $g \mapsto g \cdot L_\lambda$. Arguments similar to those for Lemma~\ref{lem:fixed-points-stab} show that $L_\lambda$ is fixed under the action of $L_n^+ P_{\mathbf{a}_n}$. Since this point is also stable under the action of lifts of elements in $\Waff^{\mathbf{f}_\lambda}$ (where $\mathbf{f}_\lambda$ is as in~\eqref{eqn:def-flambda}), it is stabilized by $L_n^+ P_{\mathbf{f}_\lambda}$. Since $\roots_n$ acts trivially on $L_nG$, our morphism therefore factors through a morphism
$\Fl^{n}_{\mathbf{f}_\lambda} \to (\Gr_G)^{\roots_n}$,
which we can then restrict to a morphism of ind-schemes
\begin{equation}
\label{eqn:morph-fixed-pts-conn-comp}
 \Fl^{n,\circ}_{\mathbf{f}_\lambda} \to (\Gr_G)^{\roots_n}.
\end{equation}

Recall that the category of ind-schemes admits arbitrary coproducts, see~\cite[Lemma~1.10(3)]{richarz}.
From the morphisms~\eqref{eqn:morph-fixed-pts-conn-comp}, using the universal property of coproducts we deduce a morphism of ind-schemes
\begin{equation}
\label{eqn:morph-fixed-pts}
 \bigsqcup_{\lambda \in (-\overline{\mathbf{a}}_n) \cap \bX^\vee} \Fl^{n,\circ}_{\mathbf{f}_\lambda} \to (\Gr_G)^{\roots_n}.
\end{equation}

\begin{prop}
\label{prop:fixed-points-Grass}
% For any $\lambda \in (-\overline{\mathbf{a}}_n) \cap \bX^\vee$, the morphism from $L_n G$ to $\Gr_G$ defined by $g \mapsto g \cdot L_\lambda$ factors through a morphism
% \[
%  \Fl^{n,\circ}_{\mathbf{f}_\lambda} \to (\Gr_G)^{\roots_n},
% \]
% %where $J_\lambda^n := \{s \in \Saff \mid s \cdot_n \lambda = \lambda\}$.
% where $\mathbf{f}_\lambda$ is as in~\eqref{eqn:def-flambda}.
% %Lemma~\ref{lem:fixed-points-stab}.
% %the facet containing $\lambda$.
% Moreover, the induced morphism
% \[
%  \bigsqcup_{\lambda \in (-\overline{\mathbf{a}}_n) \cap \bX^\vee} \Fl^{n,\circ}_{\mathbf{f}_\lambda} \to (\Gr_G)^{\roots_n}
% \]
% (provided by the universal property of coproducts) is an isomorphism of ind-schemes.
The morphism~\eqref{eqn:morph-fixed-pts} is an isomorphism of ind-schemes.
\end{prop}

\begin{proof}
Let us first check that our map induces a bijection on $\F$-points. For that we use the decomposition~\eqref{eqn:GrG-Iworbits}. Since each orbit in this decomposition is stable under the action of $\roots_n$, we deduce a decomposition of fixed points indexed by $\bX^\vee$, whose parts are described in Lemma~\ref{lem:fixed-points-orbits}:
\[
 (\Gr_G)^{\roots_n}(\F) = \bigsqcup_{\mu \in \bX^\vee} \Iw_n(\F) \cdot L_\mu.
\]
On the other hand, for any facet $\mathbf{f} \subset \overline{\mathbf{a}}_n$, the $\Iw_n$-orbits on $\Fl^{n,\circ}_{\mathbf{f}}$ are parametrized in a natural way by the quotient $\Waff/\Waff^{\mathbf{f}}$; we deduce a decomposition 
\[
 \bigsqcup_{\lambda \in (-\overline{\mathbf{a}}_n) \cap \bX^\vee} \Fl^{n,\circ}_{\mathbf{f}_\lambda}(\F) = \bigsqcup_{\lambda \in (-\overline{\mathbf{a}}_n) \cap \bX^\vee} \bigsqcup_{w\Waff^{\mathbf{f}_\lambda} \in \Waff/\Waff^{\mathbf{f}_\lambda}} \Iw_n(\F) \cdot [\iota_n(w)]_{\mathbf{f}_\lambda}
\]
where $[\iota_n(w)]_{\mathbf{f}_\lambda}$ is the image of $\iota_n(w)$ in $\Fl^{n,\circ}_{\mathbf{f}_\lambda}$.
Since $(-\overline{\mathbf{a}}_n) \cap \bX^\vee$ is a fundamental domain for the action of $\Waff$ on $\bX^\vee$ (via $\square_n$), we have a bijection
\begin{equation}
\label{eqn:bijection-Xvee-Waff}
 \bigsqcup_{\lambda \in (-\overline{\mathbf{a}}_n) \cap \bX^\vee} \Waff/\Waff^{\mathbf{f}_\lambda} \simto \bX^\vee
\end{equation}
sending $w \Waff^{\mathbf{f}_\lambda} \in \Waff / \Waff^{\mathbf{f}_\lambda}$ to $-(w \cdot_n (-\lambda))=w \square_n \lambda$. To conclude the proof of our claim, it therefore suffices to check that for any $w$ and $\lambda$ our morphism induces a bijection
\[
 \Iw_n(\F) \cdot [\iota_n(w)]_{\mathbf{f}_\lambda} \simto \Iw_n(\F) \cdot L_{w \square_n \lambda},
\]
which follows from the description of both sides in terms of (affine) root subgroups; see Lemma~\ref{lem:properties-root-subgroups}\eqref{it:conjugation-root-subgroups} (for the left-hand side) and the proof of Lemma~\ref{lem:fixed-points-orbits} (for the right-hand side).

% We will now prove that the induced morphism
% \begin{equation}
% \label{eqn:map-fixed-points-Gr}
%  \bigsqcup_{\lambda \in (-\overline{\mathbf{a}}_n) \cap \bX^\vee} \Fl^{n,\circ}_{\mathbf{f}_\lambda} \to (\Gr_G)^{\roots_n}
% \end{equation}
% is an isomorphism, which will conclude the proof. For this w
The rest of the proof will use the ``big cell'' theory recalled in~\S\ref{ss:big-cells}. Namely, 
% we have the subgroup ind-scheme $L^{--}_1 P_{\mathbf{o}_1}$, whose action on the base point $L_0$ induces an open embedding
% \[
% L^{--}_1 P_{\mathbf{o}_1} \to \Gr_G.
% \]
% Translating this open subset, we obtain an open cover of $\Gr_G$ parametrized by $\bX^\vee$: more specifically,
recall that for $\nu \in \bX^\vee$ the morphism $g \mapsto g \cdot L_\nu$ defines an ``open embedding'' 
\[
z^{\nu} \cdot L^{--}_1 P_{\mathbf{o}_1} \cdot z^{-\nu} \to \Gr_G,
\]
and that the images of these maps constitute an ``open cover'' of $\Gr_G$ (see~\S\ref{ss:big-cells}). These open subsets are stable under the $\Gm$-action by loop rotation, hence also under the $\roots_n$-action we consider here; it follows that $(\Gr_G)^{\roots_n}$ has an ``open cover'' parametrized by $\bX^\vee$, with the subset corresponding to $\nu$ naturally isomorphic to $(z^{\nu} L^{--}_1 P_{\mathbf{o}_1} z^{-\nu})^{\roots_n}$.
%The construction recalled in~\S\ref{ss:loop-flag} provides, 
Similarly,
for any $\lambda \in (-\overline{\mathbf{a}}_n) \cap \bX^\vee$ and any coset $w \Waff^{\mathbf{f}_\lambda}$ in $\Waff/\Waff^{\mathbf{f}_\lambda}$, we have considered in~\S\ref{ss:big-cells} an open subset of $\Fl^{n,\circ}_{\mathbf{f}_\lambda}$ naturally isomorphic to $\iota_n(w) \cdot L^{--}_n P_{\mathbf{f}_\lambda} \cdot \iota_n(w)^{-1}$. Taken together, this subsets provide an ``open cover'' of the left-hand side of~\eqref{eqn:morph-fixed-pts-conn-comp}.
%
% Now since $(-\overline{\mathbf{a}}_n) \cap \bX^\vee$ is a fundamental domain for the action of $\Waff$ on $\bX^\vee$ (via $\square_n$), we have a bijection
% \[
%  \bigsqcup_{\lambda \in (-\overline{\mathbf{a}}_n) \cap \bX^\vee} \Waff/\Waff^{\mathbf{f}_\lambda} \simto \bX^\vee
% \]
% sending $w \Waff^{\mathbf{f}_\lambda} \in \Waff / \Waff^{\mathbf{f}_\lambda}$ to $-(w \cdot_n (-\lambda))=w \square_n \lambda$.
%
The sets which parametrize these open covers are in a canonical bijection via~\eqref{eqn:bijection-Xvee-Waff}.
Now we will show that the map~\eqref{eqn:morph-fixed-pts} identifies the open subset of $\Fl^{n,\circ}_{\mathbf{f}_\lambda}$ associated with the coset $w \Waff^{\mathbf{f}_\lambda}$ with the open subset of $(\Gr_G)^{\roots_n}$ corresponding to $w \square_n \lambda$. 

To prove this claim it suffices to prove the equality
\begin{equation*}
%\label{eqn:fixed-points-stab-general}
 (z^{w \square_n \lambda} \cdot L^{--}_1 P_{\mathbf{o}_1} \cdot z^{-w \square_n \lambda})^{\roots_n} = \iota_n(w) \cdot L^{--}_n P_{\mathbf{f}_\lambda} \cdot \iota_n(w)^{-1}.
\end{equation*}
In case $w=1$, this equality was checked in Lemma~\ref{lem:fixed-points-stab}. To deduce the general case, write $w=\st_\mu v$ with $\mu \in \bX^\vee$ and $v \in \Wf$. Then we have
\begin{multline*}
 (z^{w \square_n \lambda} \cdot L^{--}_1 P_{\mathbf{o}_1} \cdot z^{-w \square_n \lambda})^{\roots_n} = (z^{v(\lambda)+n\mu} \cdot L^{--}_1 P_{\mathbf{o}_1} \cdot z^{-v(\lambda)-n\mu})^{\roots_n} \\
 = z^{n\mu} \dot{v} \cdot (z^{\lambda} \cdot L^{--}_1 P_{\mathbf{o}_1} \cdot z^{-\lambda})^{\roots_n} \cdot \dot{v}^{-1} z^{-n\mu} = z^{n\mu} \dot{v} \cdot L^{--}_n P_{\mathbf{f}_\lambda} \cdot \dot{v}^{-1} z^{-n\mu},
\end{multline*}
which concludes the proof.

These identifications allow us to define a right inverse to~\eqref{eqn:morph-fixed-pts} on each open subset of our ``open cover'' of $(\Gr_G)^{\roots_n}$. These morphisms coincide on intersections of open subsets because they coincide on $\F$-points (by the bijectivity claim at the beginning of the proof), which are dense since our ind-schemes are of ind-finite type. We can therefore glue these locally defined morphisms to obtain an inverse to~\eqref{eqn:morph-fixed-pts}, which finishes the proof.
\end{proof}

For $\lambda \in (-\overline{\mathbf{a}}_n) \cap \bX^\vee$,
we will denote by $\Gr_{G,(\lambda)}$ the image of $\Fl^{n,\circ}_{\mathbf{f}_\lambda}$ in $(\Gr_G)^{\roots_n}$ under the map of Proposition~\ref{prop:fixed-points-Grass}. We then have
\begin{equation}
\label{eqn:fixed-points-conn-comp}
 (\Gr_G)^{\roots_n} = \bigsqcup_{\lambda \in (-\overline{\mathbf{a}}_n) \cap \bX^\vee} \Gr_{G,(\lambda)},
\end{equation}
which describes $(\Gr_G)^{\roots_n}$ as the union of its connected components.

% Note that 
% %since since $G^{\sc}$ is (by definition) simply-connected, the ind-scheme $G^{\sc}_{\mathscr{K}}$ is connected. As a consequence, 
% each $\Gr_{G,(\mu)}$ is connected; therefore~\eqref{eqn:fixed-points-Gr} is the decomposition of $(\Gr_G)^\varpi$ as the union of its connected components.

\begin{rmk}
%\label{rmk:fixed-points-Grell}
\phantomsection
\label{rmk:fixed-points-Gr}
\begin{enumerate}
 \item 
The action of $L_n G$ on $L_0$ induces an embedding
 \[
  L_n G / L_n^+ G \hookrightarrow (\Gr_G)^{\roots_n}.
 \]
Here $L_n G / L_n^+ G$ is of course isomorphic to $\Gr_G$. In terms of the decomposition~\eqref{eqn:fixed-points-conn-comp},
%in Proposition~\ref{prop:fixed-points-Grass}, 
this embedding identifies $L_n G / L_n^+ G$ with the union of the components $\Gr_{G,(\lambda)}$ where $\lambda$ runs over $(-\overline{\mathbf{a}}_n) \cap n \bX^\vee$. These components are the only ones that are considered in~\cite{leslie-lonergan}.
%\end{rmk}
\item
\label{it:parameters-orbits-fixed-points}
As in the proof of Proposition~\ref{prop:fixed-points-Grass} in the case of $\Iw_n$-orbits,
for any facet $\mathbf{f} \subset \overline{\mathbf{a}}_n$ the $\Iw_n^+$-orbits on $\Fl^{n,\circ}_{\mathbf{f}}$ are parametrized in a natural way by the quotient $\Waff/\Waff^{\mathbf{f}}$. On the other hand, the $\Iw^+$-orbits on $\Gr_G$ are naturally parametrized by $\bX^\vee$, so that by Lemma~\ref{lem:fixed-points-orbits} the $\Iw_n^+$-orbits on $(\Gr_G)^{\roots_n}$ are also parametrized by $\bX^\vee$. Under the identification of Proposition~\ref{prop:fixed-points-Grass}, for any $\lambda \in (-\overline{\mathbf{a}}_n) \cap \bX^\vee$ the orbit in $\Fl^{n,\circ}_{\mathbf{f}_\lambda}$ corresponding to the coset $w\Waff^{\mathbf{f}}$ is mapped to the orbit in $(\Gr_G)^{\roots_n}$ parametrized by $w \square_n \lambda$. This correspondence of orbits will be important for our applications in Section~\ref{sec:applications-red-gps}.
\end{enumerate}
\end{rmk}

\section{Iwahori--Whittaker sheaves on the affine Grassmannian}
\label{sec:perv-Gr}
%%%%%%%%%%%%%%%%%%%%%%%%%

%As in Section~\ref{sec:Smith-etale} we denote by $\F$ an algebraically closed field of characteristic $p>0$.

We continue with the setting of Section~\ref{sec:fixed-points-Gr}.

%--------------------------------------------------------------
\subsection{Iwahori--Whittaker sheaves}
\label{ss:IW}
%--------------------------------------------------------------

The category of sheaves on $\Gr_G$ we will study is the \emph{Iwahori--Whittaker} derived category, whose definition we briefly recall. (For more details, see e.g.~\cite[Appendix~A]{modrap1}.)

From now on we 
%choose a nontrivial $p$-th root of unity if $\bk$. (We assume that such an element exists, which is of course always satisfied, possibly after a finite extension of $\bk$.) This 
let $\bk$ be a finite field of positive characteristic $\ell \neq p$ containing a nontrivial $p$-th root of unity. After choosing such a root of unity $\zeta$, we obtain
an Artin--Schreier local system $\LAS$ on $\Ga$, defined as the direct summand of the local system $\mathrm{AS}_* \underline{\bk}_{\Ga}$ on which $\F_p$ acts via $n \mapsto \zeta^n$. (Here, $\mathrm{AS} : \Ga \to \Ga$ is the map $x \mapsto x^p-x$, a Galois cover of group $\F_p$.) We choose once and for all a morphism of $\F$-algebraic groups $\chi_0 : U^+ \to \Ga$ which is nontrivial on any root subgroup of $U^+$ associated with a simple root, and denote by
\[
\chi : \Iwu^+ \to \Ga
\]
the composition of $\chi_0$ with the morphism $\Iwu^+ \to U^+$ induced by $\mathrm{ev}_0$.

For $X \subset \Gr_G$ a locally closed finite union of $\Iw^+$-orbits, we can choose a smooth quotient $J$ of $\Iwu^+$ of finite type such that the $\Iwu^+$-action on $X$ factors through an action of $J$, and such that $\chi$ factors through a morphism $\chi_J : J \to \Ga$. Let $\Db_c(X,\bk)$ be the constructible derived category of $\bk$-sheaves on $X$, i.e.~the full subcategory of $\Db \Sh(X,\bk)$ whose objects are the complexes with constructible cohomology sheaves. Then the constructible $(J,\chi_J^* \LAS)$-equivariant derived category of $\bk$-sheaves on $X$ is by definition the full subcategory\footnote{We insist that here the equivariant derived category can be simply defined as a full subcategory of the ordinary derived category, because we are working with an acyclic group. This is completely different from the theory of~\cite{bl}; see again~\cite[Appendix~A]{modrap1} for details.} of $\Db_c(X,\bk)$ whose objects are the complexes $\scF$ whose pullback under the action map $J \times X \to X$ is isomorphic to $\chi_J^* \LAS \boxtimes \scF$. It is well known that this subcategory is triangulated, and that it does not depend on the choice of $J$; it will be denoted
\[
\Db_{\IW}(X,\bk).
\]
%There exists a canonical fully faithful forgetful functor $\Db_{\IW}(X,\bk) \to \Db \Sh(X,\bk)$, and 
It is known also that the perverse t-structure on $\Db\Sh(X,\bk)$ restricts to a t-structure on $\Db_{\IW}(X,\bk)$, which will also be called the perverse t-structure.
(Here, ``$\IW$'' stands for ``Iwahori--Whittaker.'' We will use this expression as a replacement for ``constructible and $(J,\chi_J^* \LAS)$-equivariant'' where $J$ is as above, in all circumstances where this notion does not depend on the choice of $J$.) 

One can also define the category
\[
\Db_{\IW}(\Gr_G,\bk)
\]
of Iwahori--Whittaker sheaves on $\Gr_G$
as the direct limit of the categories $\Db_{\IW}(X,\bk)$ where $X$
runs over the closed finite unions of $\Iw^+$-orbits, ordered by
inclusion. (Here, the transition functors are the fully-faithful pushforward functors.) Since, for $X \subset Y$, the pushforward functor $\Db_{\IW}(X,\bk) \to \Db_{\IW}(Y,\bk)$ is t-exact, from the perverse t-structures on the categories $\Db_{\IW}(X,\bk)$ we obtain a perverse t-structure on $\Db_{\IW}(\Gr_G,\bk)$, whose heart will be denoted $\Perv_{\IW}(\Gr_G,\bk)$.

The considerations in the proof of Lemma~\ref{lem:fixed-points-orbits} can be used to see that
for $\lambda \in \bX^\vee$, 
%and choose a smooth quotient $J$ of finite type such that the $\Iwu^+$-action on $\Gr_{G,\lambda}^+$ factors through an action of $J$, and such that $\chi$ factors through a morphism $\chi_J : J \to \Ga$. Then 
the orbit $\Gr_{G,\lambda}^+$ supports a nonzero Iwahori--Whittaker local system iff $\lambda$ belongs to the subset
\[
\bX^\vee_{++}:=\{\mu \in \bX^\vee \mid \forall \alpha \in \mathfrak{R}^+, \, \langle \mu,\alpha \rangle >0\}.
\]
Moreover, in this case there exists (up to isomorphism) exactly one such local system of rank~$1$; it will be denoted $\LAS^\lambda$. This remark implies that for any $\mu \in \bX^\vee \smallsetminus \bX^\vee_{++}$ the category $\Db_{\IW}(\Gr_{G,\mu}^+,\bk)$ is $0$; in particular, the restriction and co-restriction of any object in $\Db_{\IW}(X,\bk)$ to $\Gr_{G,\mu}^+$ (where $X$ is any locally closed finite union of $\Iw^+$-orbits containing $\Gr_{G,\mu}^+$) vanishes.

\begin{rmk}
 If we assume that there exists an element $\varsigma \in \bX^\vee$ such that $\langle \varsigma, \alpha \rangle =1$ for all $\alpha \in \mathfrak{R}^{\mathrm{s}}$, then we have $\bX^\vee_{++} = \varsigma + \bX^\vee_{+}$.
\end{rmk}

For $\lambda \in \bX^\vee_{++}$, we set
\[
\Delta^{\IW}_\lambda := R(j_\lambda^+)_! \LAS^\lambda [\dim(\Gr_{G,\lambda}^+)], \quad \nabla^{\IW}_\lambda := R(j_\lambda^+)_* \LAS^\lambda [\dim(\Gr_{G,\lambda}^+)].
\]
Since $j_\lambda^+$ is an affine embedding, these objects are perverse sheaves by~\cite[Corollaire~4.1.3]{bbd}.
Standard arguments (going back to~\cite{bgs}) show that the category $\Perv_{\IW}(\Gr_G,\bk)$ admits a structure of a highest weight category (in the sense of~\cite[\S 7]{riche-hab}) with weight set $\bX^\vee_{++}$, standard objects the objects $(\Delta^\IW_\lambda : \lambda \in \bX^\vee_{++})$, and costandard objects the objects $(\nabla^\IW_\lambda : \lambda \in \bX^\vee_{++})$. In particular, one can consider the tilting objects in this category, i.e.~the objects which admit both a filtration with standard subquotients, and a filtration with costandard subquotients. Recall that, as remarked in~\cite{bbm}, this notion can also be characterized topologically: a perverse sheaf $\scF$ is tilting iff the complexes $(j_\lambda^+)^* \scF$ and $(j_\lambda^+)^! \scF$ are perverse (i.e.~are direct sums of copies of $\LAS^\lambda [\dim(\Gr_{G,\lambda}^+)]$) for any $\lambda \in \bX^{\vee}_{++}$.

The full subcategory of $\Perv_{\IW}(\Gr_G,\bk)$ whose objects are the tilting objects will be denoted $\Tilt_\IW(\Gr_G,\bk)$. The general theory of highest weight categories (reviewed e.g.~in~\cite{riche-hab}) guarantees that the indecomposable objects in this category are parametrized in a natural way by $\bX^\vee_{++}$. More precisely, for any $\lambda \in \bX^\vee_{++}$ there exists a unique (up to isomorphism) indecomposable object $\scT^\IW_\lambda$ in $\Tilt_\IW(\Gr_G,\bk)$ which is supported on $\overline{\Gr_{G,\lambda}^+}$, and whose restriction to $\Gr_{G,\lambda}^+$ is $\LAS^\lambda [\dim(\Gr_{G,\lambda}^+)]$; then the assignment $\lambda \mapsto \scT^\IW_\lambda$ induces a bijection between $\bX^\vee_{++}$ and the set of isomorphism classes of indecomposable objects in $\Tilt_\IW(\Gr_G,\bk)$.

%the indecomposable object associated with $\lambda$ will be denoted $\scT^\IW_\lambda$.

%--------------------------------------------------------------
\subsection{Loop rotation equivariant Iwahori--Whittaker sheaves}
\label{ss:IW-Gm}
%--------------------------------------------------------------

We will need to ``add'' the (loop rotation) $\Gm$-equivariance in the constructions of~\S\ref{ss:IW}. We therefore consider a locally closed finite union of $\Iw^+$-orbits $X \subset \Gr_G$ as above. The $\Gm$-action by loop rotation on $\Gr_G$ stabilizes each $\Iw^+$-orbit, hence also $X$, so that we can consider the constructible $\Gm$-equivariant derived category of \'etale $\bk$-sheaves $\Db_{\Gm,c}(X,\bk)$, see~\S\ref{ss:Gm-equ-Db}. The quotient $J$ of $\Iwu^+$ as in~\S\ref{ss:IW} can be chosen in such a way that the $\Gm$-action on $\Iwu^+$ induces an action on $J$. Since the morphism $\chi : \Iwu^+ \to \Ga$ is $\Gm$-equivariant (for the trivial $\Gm$-action on $\Ga$), so is $\chi_J$, and the local system $\chi_J^* \LAS$ is therefore $\Gm$-equivariant. We define the category
\[
\Db_{\IW,\Gm}(X,\bk)
\]
as the full subcategory of $\Db_{\Gm,c}(X,\bk)$ whose objects are the complexes $\scF$ such that
\[
a_J^* \scF \cong \chi_J^* \LAS \boxtimes \scF \quad \text{in } \Db_{\Gm}(J \times X,\bk),
\]
where $a_J : J \times X \to X$ is the action morphism. (Here, $\Gm$ acts diagonally on $J \times X$.) Arguments similar to those for the case when the $\Gm$-action is dropped 
%(see e.g.~\cite[Appendix~A]{modrap1}) 
show that $\Db_{\IW,\Gm}(X,\bk)$ is a triangulated subcategory of $\Db_{\Gm,c}(X,\bk)$; in fact this category is the essential image of the fully faithful functor
\[
\Db_{\Gm,c}(X,\bk) \to \Db_{\IW,\Gm}(X,\bk)
\]
sending a complex $\scF$ to $R(a_J)_! \bigl( \chi_J^* \LAS \boxtimes \scF \bigr)$. 
It is also easily checked that this category does not depend on the choice of $J$, and that the perverse t-structure on $\Db_{\Gm,c}(X,\bk)$ restricts to a t-structure on $\Db_{\IW,\Gm}(X,\bk)$.
%, whose heart will be denoted $\Perv_{\IW,\Gm}(X,\bk)$.

Taking the direct limit of the categories $\Db_{\IW,\Gm}(X,\bk)$ where $X$ runs over the closed finite unions of $\Iw^+$-orbits, we obtain a triangulated category
\[
\Db_{\IW,\Gm}(\Gr_G,\bk)
\]
with a natural perverse t-structure, whose heart will be denoted
%and an abelian category 
$\Perv_{\IW,\Gm}(\Gr_G,\bk)$.
%of perverse sheaves. 
We have a natural 
t-exact 
forgetful functor
\begin{equation}
\label{eqn:For-IW-Gr}
\Db_{\IW,\Gm}(\Gr_G,\bk) \to \Db_{\IW}(\Gr_G,\bk).
\end{equation}

%-----------------------------------------------------
\subsection{Parity complexes}
\label{ss:def-parity}
%-----------------------------------------------------

Let $X \subset \Gr_G$ be a locally closed finite union of $\Iw^+$-orbits. Recall (see~\cite{jmw, rw}) that an object $\scF$ in $\Db_\IW(X,\bk)$ is called $*$-even, resp.~$!$-even, if for any $\lambda \in \bX_{++}^\vee$ such that $\Gr_{G,\lambda}^+ \subset X$ the complex $(j_\lambda^+)^* \scF$, resp.~$(j_\lambda^+)^! \scF$, is concentrated in even degrees, i.e.~is a direct sum of objects of the form $\LAS^\lambda[n]$ with $n \in 2\Z$. (Here, by abuse we still denote by $j_\lambda^+$ the embedding of $\Gr_{G,\lambda}^+$ in $X$. Note also that if $\lambda \in \bX^\vee \smallsetminus \bX_{++}^\vee$ is such that $\Gr_{G,\lambda}^+ \subset X$, then as explained in~\S\ref{ss:IW} we have $(j_\lambda^+)^* \scF=(j_\lambda^+)^! \scF=0$ for any $\scF$ in $\Db_\IW(X,\bk)$, so that no condition is required for these strata.) We define similarly the $*$-odd and $!$-odd objects (requiring that $n$ is odd in this case), and we say that $\scF$ is even, resp.~odd, if it is both $*$-even and $!$-even, resp.~$*$-odd and $!$-odd.

These notions can also be considered in $\Db_{\IW,\Gm}(X,\bk)$; more precisely an object $\scF$ in $\Db_{\IW,\Gm}(X,\bk)$ is said to be $*$-even, resp.~$!$-even, etc., if its image in $\Db_{\IW}(X,\bk)$ (under the forgetful functor) is $*$-even, resp.~$!$-even, etc. If $\scF \in \Db_{\IW,\Gm}(X,\bk)$ is $*$-even, for any $\lambda \in \bX_{++}^\vee$ such that $\Gr_{G,\lambda}^+ \subset X$ the complex $(j_\lambda^+)^* \scF$ is a direct sum of objects of the form $\LAS^\lambda[n]$ with $n \in 2\Z$ in $\Db_{\IW,\Gm}(\Gr_{G,\lambda}^+,\bk)$. A similar comment applies to $!$-even objects (with respect to $!$-restriction), and to $*$-odd and $!$-odd objects.

By definition, the category $\Db_\IW(\Gr_G,\bk)$, resp.~$\Db_{\IW,\Gm}(\Gr_G,\bk)$, is the direct limit of the categories $\Db_{\IW}(X,\bk)$, resp.~$\Db_{\IW,\Gm}(X,\bk)$, where $X$ runs over the closed finite unions of $\Iw^+$-orbits in $\Gr_G$. Hence it makes sense to consider even and odd complexes in these categories. The general theory of~\cite{jmw} (see also~\cite{rw, acr} for some comments on the Iwahori--Whittaker case) guarantees that for any $\lambda \in \bX^\vee_{++}$ there exists a unique (up to isomorphism) indecomposable object in $\Db_{\IW}(\Gr_G,\bk)$, resp.~$\Db_{\IW,\Gm}(\Gr_G,\bk)$, which has the same parity as $\dim(\Gr^+_{G,\lambda})$, which is supported on $\overline{\Gr_{G,\lambda}^+}$, and whose restriction to $\Gr^+_{G,\lambda}$ is $\LAS^\lambda[\dim(\Gr_{G,\lambda}^+)]$. This object will be denoted
\[
\scE^\IW_\lambda, \quad \text{resp.} \quad \scE^\IW_{\lambda,\Gm}.
\]
It is known also that the image of $\scE^\IW_{\lambda,\Gm}$ under the forgetful functor~\eqref{eqn:For-IW-Gr} is $\scE^\IW_{\lambda}$, see e.g.~\cite[Lemma~2.4]{mr}. Moreover, any parity object in $\Db_{\IW}(\Gr_G,\bk)$, resp.~$\Db_{\IW,\Gm}(\Gr_G,\bk)$, is a direct sum of cohomological shifts of objects of the form $\scE^\IW_\lambda$, resp.~$\scE^\IW_{\lambda,\Gm}$.

As remarked already in~\cite{bgmrr}, these objects have an alternative description, as follows. 
It is known that the parity of $\dim(\Gr_{G,\lambda}^+)$ (with $\lambda \in \bX^\vee_{++}$) is constant on each connected component of $\Gr_G$. (This follows from the same property for $L^+G$-orbits, since each orbit $\Gr_{G,\lambda}^+$ with $\lambda \in \bX^\vee_{++}$ is dense in an $L^+G$-orbit.) As a consequence, a tilting object supported on a component where these dimensions are even, resp.~odd, is even, resp.~odd. In particular, by unicity, for any $\lambda \in \bX^\vee_{++}$ we must have
\begin{equation}
\label{eqn:TIW-EIW}
\scT^\IW_\lambda \cong \scE^\IW_\lambda.
\end{equation}
As a consequence, we obtain that the parity objects in $\Db_{\IW}(\Gr_G,\bk)$ are exactly the direct sums of cohomological shifts of tilting perverse sheaves.

Using these considerations we prove the following lemma, to be used later.

\begin{lem}
\label{lem:IW-Gm-Perv}
The forgetful functor induces an equivalence of categories
\[
\Perv_{\IW,\Gm}(\Gr_G,\bk) \simto \Perv_{\IW}(\Gr_G,\bk).
\]
\end{lem}

\begin{proof}
It follows from the general theory of equivariant perverse sheaves
(recalled e.g.~ in~\cite[\S 1.16]{br}) that the forgetful functor
$\Perv_{\Gm}(\Gr_G,\bk) \to \Perv(\Gr_G,\bk)$ is fully faithful;
therefore, so is its restriction $\Perv_{\IW,\Gm}(\Gr_G,\bk) \to
\Perv_{\IW}(\Gr_G,\bk)$. This general theory also implies that the
essential image of this functor is stable under subquotients (see
e.g.~\cite[\S 12.19]{jantzen-nilp}).
%\footnote{\color{red}GW: are you sure this is the ref? I couldn't get a copy to check.}. 
  Now from~\eqref{eqn:TIW-EIW} and the fact that each object $\scE^\IW_\lambda$ belongs to the essential image of the forgetful functor $\Db_{\IW,\Gm}(\Gr_G,\bk) \to \Db_{\IW}(\Gr_G,\bk)$ (see the remarks above), we see that the essential image of our functor contains all the tilting objects. By the general theory of highest weight categories (see~\cite[Proposition~7.17]{riche-hab}), the canonical functor provides an equivalence of triangulated categories
\begin{equation}
\label{eqn:equiv-Tilt-IW}
 \Kb \Tilt_{\IW}(\Gr_G,\bk) \simto \Db \Perv_{\IW}(\Gr_G,\bk).
\end{equation}
In particular, it follows that
any object of $\Perv_{\IW}(\Gr_G,\bk)$ is a subquotient of a tilting object, hence that it belongs to this essential image, which finishes the proof.
\end{proof}

%%%%%%%%%%%%%%%%%%%%%%%%%
\section{Smith theory for Iwahori--Whittaker sheaves on \texorpdfstring{$\Gr_G$}{GrG}}
\label{sec:smith-theory}
%%%%%%%%%%%%%%%%%%%%%%%%%

We continue with the setting of Sections~\ref{sec:fixed-points-Gr}--\ref{sec:perv-Gr}.
Our goal in this section is to build a ``Smith theory'' for the category $\Db_{\IW}(\Gr_G,\bk)$, following Treumann~\cite{treumann} and Leslie--Lonergan~\cite{leslie-lonergan}. 
%Since the setting of \'etale sheaves is not covered in these references, and since the definitions we wish to use are slightly different from those considered in these references, we reprove all the results we will need.

%--------------------------------------------------------------
\subsection{The Iwahori--Whittaker Smith category}
\label{ss:IW-Smith}
%--------------------------------------------------------------

%We now choose an identification of $\varpi=\Z/\ell \Z$ with the group $\roots_\ell$ of $\ell$-roots of unity in $\bk$, and
As in Section~\ref{sec:Smith-etale} we consider the subgroup scheme $\varpi=\roots_\ell \subset \Gm$, and the fixed points $(\Gr_G)^\varpi \subset \Gr_G$ with respect to the loop rotation action. This subscheme is described in~\S\ref{ss:Gr-fixed-points}; in particular since $(\Iw^+)^\varpi=\Iw_\ell^+$ (see Lemma~\ref{lem:fixed-points-orbits}), this group acts on $(\Gr_G)^\varpi$, and each $\Iw^+_\ell$-orbit is also an $\Iwuell^+$-orbit.

The $\Gm$-action on $\Gr_G$ stabilizes $(\Gr_G)^\varpi$, hence induces an action on this sub-ind-scheme. On the other hand, as explained above we also have an action of $\Iwuell^+$ on $(\Gr_G)^\varpi$. The analysis in~\S\ref{ss:orbits-Gr} shows that the orbits of the latter action are naturally parametrized by $\bX^\vee$, and that each orbit is stable under the action of $\Gm$. Repeating the construction in~\S\ref{ss:IW-Gm}, now with the morphism $\Iwuell^+ \to \Ga$ obtained by restricting $\chi$ one can define for any locally closed finite union of $\Iw^+_\ell$-orbits $Y \subset (\Gr_G)^\varpi$ the Iwahori--Whittaker loop rotation equivariant derived category
\[
\Db_{\IW_\ell,\Gm}(Y,\bk).
\]
As in~\S\ref{ss:IW}, in case $Y=(\Gr_{G,\lambda}^+)^\varpi$ for some $\lambda \in \bX^\vee$, this category vanishes unless $\lambda \in \bX^\vee_{++}$.

We define $\Db_{\IW_\ell,\Gm}(Y,\bk)_{\varpi\perf}$ as the full subcategory of $\Db_{\IW_\ell,\Gm}(Y,\bk)$ whose objects are the complexes $\scF$ such that $\Res^{\Gm}_{\varpi}(\scF)$ has perfect geometric stalks in the sense of~\S\ref{ss:crucial-lemma}.
We then define
the \emph{Iwahori--Whittaker Smith category} of $Y$ as the Verdier quotient
\[
\Smith_\IW(Y,\bk) := \Db_{\IW_\ell,\Gm}(Y,\bk)/\Db_{\IW_\ell,\Gm}(Y,\bk)_{\varpi\perf}.
\]
% is defined
% %If $X \subset \Gr_G$ is a locally closed finite union of $\Iw^+$-orbits, we define the Iwahori--Whittaker Smith category $\Smith_\IW(X,\bk)$ 
% as the Verdier quotient of the category $\Db_{\IW_\ell,\Gm}(Y,\bk)$ by the subcategory of objects whose image in $\Db_\varpi(Y,\bk)$ is perfect. 
This category has a natural structure of triangulated category; the (cohomological) shift functor will be denoted $[1]$ as usual.

We now check that this construction is functorial in the following sense.

\begin{lem}
 Let $Y,Z \subset (\Gr_G)^\varpi$ be two locally closed finite unions of $\Iw_\ell^+$-orbits such that $Z \subset Y$. Denoting by $f$ this inclusion, for $? \in \{*,!\}$ there exist canonical functors
 \[
  f_?^{\Smith} : \Smith_\IW(Z,\bk) \to \Smith_\IW(Y,\bk), \quad f_{\Smith}^? : \Smith_\IW(Y,\bk) \to \Smith_\IW(Z,\bk)
 \]
such that the diagrams
\[
\vcenter{
 \xymatrix{
 \Db_{\IW_\ell,\Gm}(Z,\bk) \ar[r]^-{Rf_?} \ar[d] & \Db_{\IW_\ell,\Gm}(Y,\bk) \ar[d] \\
 \Smith_\IW(Z,\bk) \ar[r]^-{f_?^\Smith} & \Smith_{\IW}(Y,\bk)
 }} \ \text{and} \
 \vcenter{
 \xymatrix{
 \Db_{\IW_\ell,\Gm}(Y,\bk) \ar[r]^-{f^?} \ar[d] & \Db_{\IW_\ell,\Gm}(Z,\bk) \ar[d] \\
 \Smith_\IW(Y,\bk) \ar[r]^-{f^?_\Smith} & \Smith_{\IW}(Z,\bk)
 }}
\]
are commutative, where the vertical arrows are the quotient functors.
\end{lem}

\begin{proof}
By the universal property of Verdier quotients,
we need to show that the functors $Rf_*$, $Rf_!$, resp.~$f^*$, $f^!$, send $\Db_{\IW_\ell,\Gm}(Z,\bk)_{\varpi\perf}$ into $\Db_{\IW_\ell,\Gm}(Y,\bk)_{\varpi\perf}$, resp.~$\Db_{\IW_\ell,\Gm}(Y,\bk)_{\varpi\perf}$ into $\Db_{\IW_\ell,\Gm}(Z,\bk)_{\varpi\perf}$. For the functor $f^*$ this claim is obvious from definition, and for $Rf_*$ it follows from Lemma~\ref{lem:pushforward-perfect-stalks}. For the functor $Rf_!$, one can argue as follows. If $\overline{y} : \Spec(K) \to Y$ is a geometric point of $Y$, then by~\cite[Expos\'e~XVII, Proposition~5.2.8]{sga4} we have
 \[
  (Rf_! \scF)_{\overline{y}} \cong R\Gamma_c(Z \times_Y \Spec(K), \scF'),
 \]
where $\scF'$ is the pullback of $\scF$. Now $Z \times_Y \Spec(K)$ is a locally closed subscheme of $\Spec(K)$, hence is either $\varnothing$ or $\Spec(K)$. Hence $(Rf_! \scF)_{\overline{y}}$ is either equal to $\scF_{\overline{y}}$ or to $0$, which shows that $Rf_! \scF$ must belong to $\Db_{\IW_\ell,\Gm}(Y,\bk)_{\varpi\perf}$.

Finally we treat the case of $f^!$. For this we can assume that $f$ is either a closed embedding or an open embedding. In the latter case we have $f^!=f^*$, hence the claim is known. In the former case, we denote by $g$ the complementary open embedding. Then, given $\scF$ in $\Db_{\IW_\ell,\Gm}(Y,\bk)_{\varpi\perf}$ we consider the distinguished triangle
\[
 f_* f^! \scF \to \scF \to Rg_* g^* \scF \xrightarrow{[1]}.
\]
Here $\scF$ and $Rg_* g^* \scF$ belong to $\Db_{\IW_\ell,\Gm}(Y,\bk)_{\varpi\perf}$, hence so does $f_* f^! \scF$. This implies that $f^! \scF$ belongs to $\Db_{\IW_\ell,\Gm}(Z,\bk)_{\varpi\perf}$, which completes the proof.
\end{proof}

% If $Y,Z \subset (\Gr_G)^\varpi$ are two locally closed finite unions of $\Iw_\ell^+$-orbits such that $Z \subset Y$, and if we denote by $f : Z \to Y$ the embedding, then by Lemma~\ref{lem:pb-pf-perfect} the functors $f_*$, $f_!$, $f^*$, $f^!$ induce triangulated functors between the appropriate Smith categories, which we will denote by
% \[
% f^\Smith_*, f^\Smith_! : \Smith_\IW(Z,\bk) \to \Smith_\IW(Y,\bk), \quad f_\Smith^*, f_\Smith^! : \Smith_\IW(Y,\bk) \to \Smith_\IW(Z,\bk).
% \]
It is easily seen that $(f^*_\Smith, f^\Smith_*)$ and $(f^\Smith_!, f_\Smith^!)$ are adjoint pairs of functors. In particular, if $f$ is a closed embedding then the functor $f^\Smith_*=f^\Smith_!$ is fully faithful, so that the category $\Smith_\IW(Z,\bk)$ can (and will) be identified with a full triangulated subcategory in $\Smith_\IW(Y,\bk)$. It is also easily checked that, given a decomposition of $Y$ as a disjoint union of a closed (in $Y$) finite union of $\Iw_\ell^+$-orbits and its open complement, we have canonical distinguished triangles as in the ``recollement'' setting of~\cite[\S 1.4]{bbd}.

The full faithfulness of pushforward under closed embeddings allows to define the category $\Smith_\IW((\Gr_G)^\varpi,\bk)$ as the direct limit of the categories $\Smith_\IW(Y,\bk)$ where $Y$ runs over the closed finite unions of $\Iw^+_\ell$-orbits in $(\Gr_G)^\varpi$.

%-----------------------------------
\subsection{The Smith localization functor}
%-----------------------------------

We will be particularly interested in the construction of~\S\ref{ss:IW-Smith} in the case $Y=X^\varpi$ for some
locally closed finite union of $\Iw^+$-orbits $X \subset \Gr_G$. In this case, we denote by $i_X : X^\varpi \to X$ the embedding. For any $\scF$ in $\Db_{\IW,\Gm}(X,\bk)$, we have objects $i_X^! \scF$ and $i_X^* \scF$ in $\Db_{\IW_\ell,\Gm}(X^\varpi,\bk)$, and a canonical morphism
\begin{equation*}
\label{eqn:morph-i!*}
i_X^! \scF \to i_X^* \scF,
\end{equation*}
see~\eqref{eqn:morph-Smith}.
%obtained by applying the functor $i_X^!$ to the adjunction morphism $\scF \to (i_X)_* i_X^* \scF$. 
It follows from Lemma~\ref{lem:crucial-Smith} that the cone of this morphism is killed by the quotient functor
\[
 \Db_{\IW_{\ell}, \Gm}(X^\varpi, \bk) \to \Smith_{\IW}(X^\varpi,\bk).
\]
We can therefore define the functor
\[
i_X^{!*} : \Db_{\IW, \Gm}(X, \bk) \to \Smith_\IW(X^\varpi,\bk)
\]
as the composition of either $i_X^*$ or $i_X^!$ with this quotient functor.

This functor is compatible with the push/pull functors associated with locally closed embeddings, in the following sense.

\begin{prop}
\label{prop:i!*-pushpull}
If $X,Y \subset \Gr_G$ are two locally closed finite unions of $\Iw^+$-orbits such that $X \subset Y$, and if we denote by $f : X \to Y$ the embedding and by $f^\varpi : X^\varpi \to Y^\varpi$ its restriction to $X^\varpi$, then we have canonical isomorphisms of functors
\begin{gather*}
i_Y^{!*} \circ Rf_* \cong (f^\varpi)^\Smith_* \circ i_X^{!*}, \qquad i_Y^{!*} \circ Rf_! \cong (f^\varpi)^\Smith_! \circ i_X^{!*}, \\
i_X^{!*} \circ f^* \cong (f^\varpi)_\Smith^* \circ i_Y^{!*}, \qquad i_X^{!*} \circ f^! \cong (f^\varpi)_\Smith^! \circ i_Y^{!*}.
\end{gather*}
\end{prop}

\begin{proof}
The first, resp.~second, isomorphism on the first line follows from the base change theorem (see~\cite[Expos\'e~XVIII, Corollaire~3.1.12.3]{sga4} and~\cite[Expos\'e~XVII, Th\'eor\`eme~5.2.6]{sga4} respectively) if we see $i_Y^{!*}$ and $i_X^{!*}$ as the compositions of $i_Y^!$ and $i_X^!$, resp.~of $i_Y^*$ and $i_X^*$, with the appropriate quotient functors. The isomorphisms on the second line follow similarly from the compatibility of pullback functors with composition.
\end{proof}

Taking the direct limit of the functors $i_X^{!*}$ for $X$ a closed finite union of $\Iw^+$-orbits in $\Gr_G$, we also obtain a functor
\[
i_{\Gr_G}^{!*} : \Db_{\IW, \Gm}(\Gr_G, \bk) \to \Smith_\IW((\Gr_G)^\varpi,\bk).
\]

%---------------------------------------------------
\subsection{Some first properties of the Iwahori--Whittaker Smith category}
%---------------------------------------------------

Let us fix some $\lambda \in \bX^\vee_{++}$, and consider the $\Iwuell^+$-orbit $(\Gr_{G,\lambda}^+)^\varpi \subset (\Gr_G)^\varpi$. (Once again, the Iwahori--Whittaker category associated with an orbit labelled by a weight in $\bX^\vee \smallsetminus \bX^{\vee}_{++}$ vanishes; these coweights can therefore be ignored.) We set
\[
\scL_\Smith^\lambda := i^{!*}_{\Gr_{G,\lambda}^+}(\LAS^\lambda).
\]

\begin{lem}
\label{lem:parity-vanishing}
For any $n \in \Z$, we have
\[
\Hom_{\Smith_\IW((\Gr_{G,\lambda}^+)^\varpi,\bk)}(\scL_\Smith^\lambda,\scL_\Smith^\lambda[n])=\begin{cases}
\bk & \text{if $n$ is even;} \\
0 & \text{if $n$ is odd.}
\end{cases}
\]
\end{lem}

\begin{proof}
Since $\Iwuell^+$ acts transitively on $(\Gr_{G,\lambda}^+)^\varpi$ (see Lemma~\ref{lem:fixed-points-orbits}), we have an equivalence of triangulated categories
\[
 \Db_{\IW_\ell,\Gm}((\Gr_{G,\lambda}^+)^\varpi,\bk) \cong \Db_{\Gm,c}(\pt,\bk)
\]
which matches $\LAS^\lambda$ with $\underline{\bk}_{\pt}$. This equivalence induces an equivalence
\[
 \Smith_{\IW}((\Gr_{G,\lambda}^+)^\varpi,\bk) \cong \Smith(\pt,\bk),
\]
where the right-hand side is defined in~\S\ref{ss:Smith-pt}. The claim then follows from Lemma~\ref{lem:Hom-Smith-pt}.
\end{proof}

We consider once again a general locally closed finite union of $\Iw_\ell^+$-orbits $Y \subset (\Gr_G)^\varpi$.

\begin{lem}
\label{lem:shift-Smith}
There exists a canonical isomorphism of endofunctors of $\Smith_\IW(Y,\bk)$
\[
\id \simto [2].
\]
\end{lem}

\begin{proof}
 As explained in Lemma~\ref{lem:Hom-Smith-pt}, there exists a canonical map $\ubk_{\pt} \to \ubk_{\pt}[2]$ in $\Db_{\Gm,c}(\pt,\bk)$ whose cone has perfect geometric stalks. Pulling back to $Y$ we deduce a canonical morphism $\ubk_Y \to \ubk_Y[2]$ whose cone has perfect geometric stalks. Since the tensor product with $\ubk_Y$, resp.~$\ubk_Y[2]$, defines an endofunctor of $\Db_{\IW_\ell,\Gm}(Y,\bk)$ which is isomorphic to $\id$, resp.~to $[2]$, and since the tensor product (over $\bk$) of a perfect complex of $\bk[\varpi]$-modules with any bounded complex is perfect, the desired claim follows.
\end{proof}

\begin{prop}
\label{prop:Hom-finite}
For any $\scF,\scG$ in $\Smith_\IW(Y,\bk)$, the $\bk$-vector space
\[
\Hom_{\Smith_\IW(Y,\bk)}(\scF,\scG)
\]
is finite-dimensional.
\end{prop}

\begin{proof}
 The proof proceeds by induction on the number of $\Iw_\ell^+$-orbits in $Y$. In fact the distinguished triangles from the ``recollement'' setting (see~\S\ref{ss:IW-Smith}) reduce the proof to the case $Y$ consists of one orbit, which follows from Lemma~\ref{lem:parity-vanishing}.
\end{proof}

%%%%%%%%%%%%%%%%%%%%%%%%
\section{Parity objects in Smith categories}
\label{sec:parity-Smith}
%%%%%%%%%%%%%%%%%%%%%%%%

We continue with the setting of Sections~\ref{sec:fixed-points-Gr}--\ref{sec:smith-theory}.

%-----------------------------------------------------
\subsection{Definition}

As remarked already in~\cite{leslie-lonergan} (using slightly different definitions), the theory of parity complexes from~\cite{jmw} adapts easily to the Smith category $\Smith_\IW(Y,\bk)$, where $Y \subset (\Gr_G)^\varpi$ is any locally closed union of $\Iw_\ell^+$-orbits. Namely, we will say
that an object $\scF$ in $\Smith_\IW(Y,\bk)$ is $*$-\emph{even}, resp. $!$-\emph{even}, if for any $\lambda \in \bX_{++}^\vee$ such that $(\Gr^+_{G,\lambda})^\varpi \subset Y$, denoting by $j_\lambda^{+,\varpi} : (\Gr^+_{G,\lambda})^\varpi \to Y$ the embedding, the object $(j^{+,\varpi}_\lambda)^*_\Smith \scF$, resp.~$(j^{+,\varpi}_\lambda)^!_\Smith \scF$, is isomorphic to a direct sum of copies of $\scL_\Smith^\lambda$. (In this case we do not need to consider even shifts because of Lemma~\ref{lem:shift-Smith}.) We will then say that $\scF$ is \emph{even} if it is both $*$-even and $!$-even, and define similarly the notions of $*$-odd, $!$-odd, and odd objects (replacing $\scL_\Smith^\lambda$ by its shift by $1$). We will denote by $\Smith^0_\IW(Y,\bk)$, resp.~$\Smith^1_\IW(Y,\bk)$, resp.~$\Smith^{\mathrm{par}}_\IW(Y,\bk)$, the full subcategory of $\Smith_\IW(Y,\bk)$ whose objects are the even objects, resp.~the odd objects, resp.~the objects which are isomorphic to a direct sum of an even and an odd object.

Recall that an additive category is called \emph{Krull--Schmidt} if any object can be written as a direct sum of indecomposable objects whose endomorphism rings are local. (If this property holds, then such a decomposition is unique up to isomorphism and permutation of factors.)

\begin{lem}
\label{lem:Krull-Schmidt}
The categories $\Smith^0_\IW(Y,\bk)$, $\Smith^1_\IW(Y,\bk)$ and $\Smith^{\mathrm{par}}_\IW(Y,\bk)$ are Krull--Schmidt.
\end{lem}

\begin{proof}
By Proposition~\ref{prop:Hom-finite} and~\cite[Corollary~A.2]{cyz}, to prove the lemma it suffices to prove that any idempotent in the category $\Smith^{\mathrm{par}}_\IW(Y,\bk)$ splits. We do this by induction on the number of $\Iw^+_\ell$-orbits in $Y$. If $Y=(\Gr^+_{G,\lambda})^\varpi$ for some $\lambda \in \bX^\vee$, and if $\scF$ belongs to $\Smith^{\mathrm{par}}_\IW(Y,\bk)$ then either $\scF=0$ (in which case there is nothing to prove) or $\lambda \in \bX^\vee_{++}$ and $\scF=(\scL_\Smith^\lambda)^{\oplus n} \oplus (\scL_\Smith^\lambda)^{\oplus m}[1]$ for some $n,m \in \Z_{\geq 0}$. In this case, by Lemma~\ref{lem:parity-vanishing} we have
\[
\End_{\Smith_\IW(Y,\bk)}(\scF) \cong \mathrm{M}_n(\bk) \times \mathrm{M}_m(\bk),
\]
so that any idempotent in $\End_{\Smith_\IW(Y,\bk)}(\scF)$ indeed splits.

To treat the induction step, we choose a closed $\Iw^+_\ell$-orbit $Z \subset Y$, and denote by
\[
i : Z \hookrightarrow Y, \qquad j : Y \smallsetminus Z \hookrightarrow Y
\]
the embeddings. For any $\scF$ in $\Smith^{\mathrm{par}}_\IW(Y,\bk)$ we then have a distinguished triangle
\[
i^\Smith_! i^!_\Smith \scF \to \scF \to j_*^\Smith j^*_{\Smith} \scF \xrightarrow{[1]},
\]
and the objects $i^!_\Smith \scF$ and $j^*_{\Smith} \scF$ belong to $\Smith^{\mathrm{par}}_\IW(Z,\bk)$ and to $\Smith^{\mathrm{par}}_\IW(Y \smallsetminus Z,\bk)$ respectively. If $e \in \End_{\Smith_\IW(Y,\bk)}(\scF)$ is an idempotent, then $i^!_\Smith(e)$ and $j^*_{\Smith}(e)$ are idempotents too, hence they split by the induction hypothesis. By~\cite[Proposition~2.3]{le-chen}, this implies that $e$ splits.
\end{proof}

We will also define the categories
\[
\Smith^0_\IW((\Gr_G)^\varpi,\bk), \quad \Smith^1_\IW((\Gr_G)^\varpi,\bk) \quad \text{and} \quad \Smith^{\mathrm{par}}_\IW((\Gr_G)^\varpi,\bk)
\]
as the direct limits of their counterparts for $Y$, where $Y$ runs over closed finite unions of $\Iw^+_\ell$-orbits in $(\Gr_G)^\varpi$. (Equivalently, these categories can be defined in terms of restrictions and corestrictions to $\Iw^+_\ell$-orbits, as for their counterparts above.) Of course, Lemma~\ref{lem:Krull-Schmidt} implies that these categories
are Krull--Schmidt.

%-----------------------------------------------------
\subsection{Basic properties}
%-----------------------------------------------------

The study of parity objects in $\Smith_\IW(Y,\bk)$ is very similar to its counterpart in ordinary derived categories of sheaves performed in~\cite{jmw}; its essential ingredients are the parity vanishing property for one stratum proved in Lemma~\ref{lem:parity-vanishing}, and standard distinguished triangles associated with a decomposition of a space into a closed part and its open complement. For this reason we will not give any proof in this subsection; these can be obtained by repeating the proofs of~\cite{jmw} essentially word-for-word.

The following is the analogue of~\cite[Corollary~2.8 and Proposition~2.11]{jmw}.

\begin{lem}
\label{lem:Hom-vanishing-even-odd}
If $\scF, \scG \in \Smith_\IW(Y,\bk)$ are such that $\scF$ is $*$-even and $\scG$ is $!$-odd, then we have
\[
\Hom_{\Smith_\IW(Y,\bk)}(\scF,\scG)=0.
\]
As a consequence, if $Z \subset Y$ is an open union of $\Iw^+_\ell$-orbits, the restriction of an indecomposable even (resp.~odd) object of $\Smith_\IW(Y,\bk)$ to $Z$ is either indecomposable or zero.
\end{lem}

Next, we define the \emph{support} of an object $\scF \in \Smith_\IW(Y,\bk)$ as the closure of the union of the strata $(\Gr_{G,\lambda}^+)^\varpi$ where $\lambda \in \bX^\vee_{++}$ is such that $(\Gr_{G,\lambda}^+)^\varpi \subset Y$ and $(j_\lambda^{+,\varpi})^*_\Smith \scF$ or $(j_\lambda^{+,\varpi})^!_\Smith \scF$ is nonzero. The following claim is the analogue of~\cite[Theorem~2.12]{jmw}.

\begin{prop}
\label{prop:classification-parity}
If $\scF \in \Smith_\IW(Y,\bk)$ is even (resp.~odd), nonzero, and indecomposable, then there exists exactly one $\lambda \in \bX^\vee_{++}$ such that $(\Gr_{G,\lambda}^+)^\varpi$ is open in the support of $\scF$.

Moreover, for any $\lambda \in \bX^\vee_{++}$ such that $(\Gr_{G,\lambda}^+)^\varpi \subset Y$, there exists at most one indecomposable even, resp.~odd, object $\scF$ in $\Smith_\IW(Y,\bk)$ such that $(\Gr_{G,\lambda}^+)^\varpi$ is open in the support of $\scF$ and $(j^{+,\varpi}_\lambda)^* \scF \cong \scL_\Smith^\lambda$, resp.~$(j^{+,\varpi}_\lambda)^* \scF \cong \scL_\Smith^\lambda[1]$.
\end{prop}

%--------------------------------------------------------------
\subsection{Comparison of parity objects in \texorpdfstring{$\Db_\IW(\Gr_G,\bk)$}{DbGr} and \texorpdfstring{$\Smith_\IW((\Gr_G)^\varpi,\bk)$}{SmithGr}}
\label{ss:comparison-IW-Sm}
%--------------------------------------------------------------

%Similarly, the category $\Smith_{\IW}((\Gr_G)^\varpi,\bk)$ identifies with the direct limit of the categories $\Smith_{\IW}(Y,\bk)$, where $Y$ runs over the closed finite unions of $\Iw_\ell^+$-orbits in $(\Gr_G)^\varpi$. Hence it makes sense to consider even and odd objects in this category. By 
Proposition~\ref{prop:classification-parity} implies that for any $\lambda \in \bX^\vee_{++}$ there exists at most one indecomposable even, resp.~odd, object in $\Smith_\IW((\Gr_G)^\varpi,\bk)$ in the support of which $(\Gr_{G,\lambda}^+)^\varpi$ is open, and whose restriction to $(\Gr_{G,\lambda}^+)^\varpi$ is $\scL_\Smith^\lambda$, resp.~$\scL_\Smith^\lambda[1]$. If it exists (which, as we shall see very soon, is always the case), this object will be denoted $\scE^{\Smith,0}_\lambda$, resp.~$\scE^{\Smith,1}_\lambda$. (Of course, as soon as one of these objects exists the other one exists also, and we have $\scE^{\Smith,1}_\lambda \cong \scE^{\Smith,0}_\lambda[1]$.) With this notation, any indecomposable object in $\Smith^{\mathrm{par}}_\IW((\Gr_G)^\varpi,\bk)$ is isomorphic to an object $\scE^{\Smith,0}_\lambda$ or $\scE^{\Smith,1}_\lambda$, and
Lemma~\ref{lem:Krull-Schmidt} implies that any object of $\Smith^{\mathrm{par}}_\IW((\Gr_G)^\varpi,\bk)$ is a direct sum of such objects (in an essentially unique way).

Recall that the connected components of $(\Gr_G)^\varpi$ are the subvarieties $\Gr_{G,(\lambda)}$ with $\lambda \in (-\overline{\mathbf{a}}_\ell) \cap \bX^\vee$, see~\eqref{eqn:fixed-points-conn-comp}.
%Proposition~\ref{prop:fixed-points-Grass}. 
Of course, each such connected component is contained in a connected component of $\Gr_G$. Recall also (see~\S\ref{ss:def-parity}) that the dimensions of the orbits $\Gr_{G,\mu}^+$ with $\mu \in \bX^\vee_{++}$ contained in a given connected component of $\Gr_G$ are of constant parity. We set $\mathsf{p}(\lambda)=0$, resp.~$\mathsf{p}(\lambda)=1$, if all these orbits 
%$\Gr^+_{G,\mu}$ ($\mu \in \varsigma + \bX^\vee_+$) 
contained in the connected component containing $\Gr_{G,(\lambda)}$ are even-dimensional, resp.~odd-dimensional. We then denote by $\Smith^{\natural}_\IW((\Gr_G)^\varpi,\bk)$ the full subcategory of $\Smith_\IW((\Gr_G)^\varpi,\bk)$ whose objects are those whose restriction to $\Gr_{G,(\lambda)}$ is even if $\mathsf{p}(\lambda)=0$, and odd if $\mathsf{p}(\lambda)=1$.

The following statement is 
%easy to prove, but is 
the crux of this paper.
%Here, we denote by $i_{\Gr_G}^{!*}$ the functor defined as the direct limit of the functors $i_X^{!*}$ of~\S\ref{ss:Smith-cat} for all closed finite unions of $\Iw^+$-orbits $X \subset \Gr_G$.
(This statement is equivalent to Theorem~\ref{thm:main1} in the introduction since perverse parity objects in $\Db_{\IW,\Gm}(\Gr_G,\bk)$ are the same as images of tilting perverse sheaves under the equivalence of Lemma~\ref{lem:IW-Gm-Perv}; see~\S\ref{ss:def-parity}.)

\begin{thm}
\label{thm:i!*-parities}
The composition
\[
\Perv_\IW(\Gr_G,\bk) \xrightarrow[\sim]{\text{Lemma~\ref{lem:IW-Gm-Perv}}} \Perv_{\IW,\Gm}(\Gr_G,\bk) \xrightarrow{i^{!*}_{\Gr_G}} \Smith_\IW((\Gr_G)^\varpi,\bk)
\]
restricts to an equivalence of categories
\begin{equation*}
\Tilt_\IW(\Gr_G,\bk) \to \Smith^{\natural}_\IW((\Gr_G)^\varpi,\bk).
\end{equation*}
Moreover, the objects $\scE^{\Smith,0}_\lambda$ and $\scE^{\Smith,1}_\lambda$ exist for any $\lambda \in \bX^\vee_{++}$.
\end{thm}

\begin{proof}
It easily follows from Proposition~\ref{prop:i!*-pushpull} and the considerations above that the functor $i^{!*}_{\Gr_G}$ sends even, resp.~odd, objects to even, resp.~odd, objects. Therefore, since any indecomposable object in $\Tilt_\IW(\Gr_G,\bk)$ is either even or odd (see~\eqref{eqn:TIW-EIW}), our functor restricts to a functor
\[
\Tilt_\IW(\Gr_G,\bk) \to \Smith^{\natural}_\IW((\Gr_G)^\varpi,\bk).
\]
Next, standard arguments (based in particular on Lemma~\ref{lem:Hom-vanishing-even-odd}) allow to prove by induction on the length of the filtrations that, for any $\scF,\scG$ in $\Perv_\IW(\Gr_G,\bk)$ such that $\scF$ admits a standard filtration and $\scG$ admits a costandard filtration, this functor induces an isomorphism
\[
\Hom_{\Perv_\IW(\Gr_G,\bk)}(\scF,\scG) \simto \Hom_{\Smith_\IW((\Gr_G)^\varpi,\bk)}(i^{!*}_{\Gr_G}(\scF),i^{!*}_{\Gr_G}(\scG)).
\]
(Here, the crucial case when $\scF=\Delta^\IW_\lambda$ and $\scG=\nabla^\IW_\lambda$ for some $\lambda \in \bX^\vee_{++}$ is given by Lemma~\ref{lem:parity-vanishing}.) Full faithfulness of our functor follows.

For any $\lambda \in \bX^\vee_{++}$, the object $i_{\Gr_G}^{!*}(\scE^\IW_{\lambda,\Gm})$ is indecomposable (by full faithfulness) and either even or odd. Moreover, since $\Gr_{G,\lambda}^+$ is open in the support of $\scE^\IW_{\lambda,\Gm}$ we see that $(\Gr_{G,\lambda}^+)^\varpi$ is open in the support of $i_{\Gr_G}^{!*}(\scE^\IW_{\lambda,\Gm})$. Therefore the objects $\scE^{\Smith,0}_\lambda$ and $\scE^{\Smith,1}_\lambda$ exist, and we have
\begin{equation}
\label{eqn:i!*-parities}
i_{\Gr_G}^{!*}(\scE^\IW_{\lambda,\Gm}) \cong 
\begin{cases}
\scE^{\Smith,0}_\lambda & \text{if $\dim(\Gr_{G,\lambda}^+)$ is even;} \\
\scE^{\Smith,1}_\lambda & \text{if $\dim(\Gr_{G,\lambda}^+)$ is odd.}
\end{cases}
\end{equation}
These considerations show that our functor is essentially surjective, hence an equivalence of categories.
\end{proof}

%-----------------------------------------------------------------
\subsection{Comparison of parity objects on \texorpdfstring{$(\Gr_G)^\varpi$}{Grpi} and in the Smith category}
\label{ss:comparison}
%-----------------------------------------------------------------

%We now fix $\lambda \in A_\ell$, and consider the associated component $\Gr_{G,(\lambda)} \subset (\Gr_G)^\varpi$. The considerations on parity objects from~\S\ref{ss:def-parity} make sense also in the categories $\Db_{\IW_\ell}(\Gr_{G,(\lambda)},\bk)$ and $\Db_{\IW_\ell,\Gm}(\Gr_{G,(\lambda)},\bk)$. 

Now we consider the Iwahori--Whittaker categories
\[
 \Db_{\IW_\ell}((\Gr_G)^\varpi,\bk) \quad \text{and} \quad \Db_{\IW_\ell,\Gm}((\Gr_G)^\varpi,\bk),
\]
and the quotient functor
\[
 \mathsf{Q} : \Db_{\IW_\ell,\Gm}((\Gr_{G})^{\varpi},\bk) \to \Smith_{\IW}((\Gr_G)^{\varpi},\bk).
\]
The theory of parity complexes (as in~\S\ref{ss:def-parity}) of course also applies in the categories $\Db_{\IW_\ell}((\Gr_G)^\varpi,\bk)$ and $\Db_{\IW_\ell,\Gm}((\Gr_G)^\varpi,\bk)$;
once again the indecomposable parity objects in these categories are classified (up to cohomological shift) by the $\Iw_\ell^+$-orbits in $(\Gr_G)^{\varpi}$ which support a nonzero Iwahori--Whittaker local system (i.e.~by $\bX^\vee_{++}$), and the forgetful functor
\begin{equation*}
%\label{eqn:For-IWell}
 \For_{\Gm} : \Db_{\IW_\ell,\Gm}((\Gr_{G})^{\varpi},\bk) \to \Db_{\IW_\ell}((\Gr_{G})^{\varpi},\bk)
\end{equation*}
sends indecomposable parity complexes to indecomposable parity complexes. In particular, this functor induces a bijection between the sets of isomorphism classes of indecomposable parity objects in $\Db_{\IW_\ell,\Gm}((\Gr_{G})^{\varpi},\bk)$ and in $\Db_{\IW_\ell}((\Gr_{G})^{\varpi},\bk)$; up to replacing the category of parity objects in $\Db_{\IW_\ell}((\Gr_{G})^{\varpi},\bk)$ by an equivalent category (which we will omit from notation), one can therefore consider whenever convenient that the objects in these categories are the same.

The situation in this setting is even more favorable than in that of~\S\ref{ss:def-parity}, due to the following property. (Here, if $\mathsf{D}$ is a triangulated category, we write $\Hom^\bullet_{\mathsf{D}}(-,-)$ for $\bigoplus_{n \in \Z} \Hom_{\mathsf{D}}(-,-[n])$.)

\begin{lem}
\label{lem:Hom-parity-For-fixed-points}
 For any parity complexes $\scE,\scE'$ in $\Db_{\IW_\ell,\Gm}((\Gr_{G})^{\varpi},\bk)$, there exists a canonical isomorphism of graded $\bk$-vector spaces
 \begin{multline*}
  \Hom^\bullet_{\Db_{\IW_\ell,\Gm}((\Gr_{G})^{\varpi},\bk)}(\scE,\scE') \\
  \cong \coH^\bullet_{\Gm}(\mathrm{pt};\bk) \otimes_\bk \Hom^{\bullet}_{\Db_{\IW_\ell}((\Gr_{G})^{\varpi},\bk)}(\For_{\Gm}(\scE),\For_{\Gm}(\scE')).
 \end{multline*}
%where $\For$ is the functor of~\eqref{eqn:For-IWell}. 
Moreover, these isomorphisms are compatible with composition in the obvious way.
\end{lem}

\begin{rmk} 
%{\color{red} GW: added}
  The moral of this lemma is that the $\Gm$-action on
  $(\Gr_{G})^{\varpi}$ ``looks like the trivial action'' as it factors through the map $t \mapsto t^\ell$ and our
  coefficients are of characteristic $\ell$.
\end{rmk}

\begin{proof}
 By definition, the $\Gm$-action on $(\Gr_G)^\varpi$ factors through the quotient
 \[
 \Gm \to \Gm/\varpi = \Gm, \quad t \mapsto t^\ell.
 \]
 In other words, if we denote by $\Gm'$ another copy of $\Gm$, then there exists an action of $\Gm'$ on $(\Gr_G)^\varpi$ from which the $\Gm$-action we want to consider is deduced via the morphism $\Gm \to \Gm'$ defined by $t \mapsto t^\ell$. The $\Gm$-action on $\Iwuell^+$ is similarly obtained from an action of $\Gm'$, so that one can consider the category $\Db_{\IW_\ell,\Gm'}((\Gr_G)^\varpi,\bk)$ defined in the obvious way. With this notation introduced, the same considerations as in~\cite[Lemma~2.2]{mr} show that for any parity complexes $\scF,\scF'$ in $\Db_{\IW_\ell,\Gm'}((\Gr_G)^\varpi,\bk)$, the restriction functor
 \begin{equation}
 \label{eqn:Res-Gm'-Gm}
  \Res^{\Gm'}_{\Gm} : \Db_{\IW_\ell,\Gm'}((\Gr_G)^\varpi,\bk) \to \Db_{\IW_\ell,\Gm}((\Gr_G)^\varpi,\bk)
 \end{equation}
induces an isomorphism of graded $\bk$-vector spaces
\begin{multline*}
 \coH^\bullet_{\Gm}(\mathrm{pt};\bk) \otimes_{\coH^\bullet_{\Gm'}(\mathrm{pt};\bk)} \Hom^{\bullet}_{\Db_{\IW_\ell, \Gm'}((\Gr_G)^\varpi,\bk)}(\scF,\scF') \\
 \simto \Hom^\bullet_{\Db_{\IW_\ell,\Gm}((\Gr_G)^\varpi,\bk)}(\Res^{\Gm'}_{\Gm}(\scF),\Res^{\Gm'}_{\Gm}(\scF')).
\end{multline*}
Now since $\bk$ has characteristic $\ell$, the morphism $\coH^\bullet_{\Gm'}(\mathrm{pt};\bk) \to \coH^\bullet_{\Gm}(\mathrm{pt};\bk)$ induced by our morphism $\Gm \to \Gm'$ vanishes in positive degrees, so that we have
\begin{multline*}
 \coH^\bullet_{\Gm}(\mathrm{pt};\bk) \otimes_{\coH^\bullet_{\Gm'}(\mathrm{pt};\bk)} \Hom^{\bullet}_{\Db_{\IW_\ell, \Gm'}((\Gr_G)^\varpi,\bk)}(\scF,\scF') \\
 \cong \coH^\bullet_{\Gm}(\mathrm{pt};\bk) \otimes_\bk \left( \bk \otimes_{\coH^\bullet_{\Gm'}(\mathrm{pt};\bk)} \Hom^{\bullet}_{\Db_{\IW_\ell, \Gm'}((\Gr_G)^\varpi,\bk)}(\scF,\scF') \right).
\end{multline*}
As in~\cite[Lemma~2.2]{mr} the forgetful functor $\For_{\Gm'}$ induces an isomorphism
\begin{multline*}
  \bk \otimes_{\coH^\bullet_{\Gm'}(\mathrm{pt};\bk)} \Hom^{\bullet}_{\Db_{\IW_\ell, \Gm'}((\Gr_G)^\varpi,\bk)}(\scF,\scF') \\
  \simto  \Hom^{\bullet}_{\Db_{\IW_\ell}((\Gr_G)^\varpi,\bk)}(\For_{\Gm'}(\scF),\For_{\Gm'}(\scF')),
\end{multline*}
so that we finally obtain a canonical isomorphism
\begin{multline*}
 \coH^\bullet_{\Gm}(\mathrm{pt};\bk) \otimes_\bk  \Hom^{\bullet}_{\Db_{\IW_\ell}((\Gr_G)^\varpi,\bk)}(\For_{\Gm'}(\scF),\For_{\Gm'}(\scF')) \\
 \simto \Hom^\bullet_{\Db_{\IW_\ell,\Gm}((\Gr_G)^\varpi,\bk)}(\Res^{\Gm'}_{\Gm}(\scF),\Res^{\Gm'}_{\Gm}(\scF')).
\end{multline*}
To conclude it suffices to remark that the functor $\Res^{\Gm'}_{\Gm}$ from~\eqref{eqn:Res-Gm'-Gm} induces a (canonical) bijection between the isomorphism classes of parity complexes in the categories $\Db_{\IW_\ell,\Gm'}((\Gr_G)^\varpi,\bk)$ and $\Db_{\IW_\ell,\Gm}((\Gr_G)^\varpi,\bk)$; one can therefore replace $\Res^{\Gm'}_{\Gm}(\scF)$ and $\Res^{\Gm'}_{\Gm}(\scF')$ in these isomorphisms by general parity complexes in $\Db_{\IW_\ell,\Gm}((\Gr_G)^\varpi,\bk)$.
\end{proof}

It is clear from definitions that the functor $\mathsf{Q}$ sends parity complexes to parity complexes. 
%In fact this functor (when restricted to parity complexes) is close to being an equivalence, as explained in the following statement.

\begin{prop}
\label{prop:Hom-parity-Q}
 For any complexes $\scE,\scE'$ in $\Db_{\IW_\ell,\Gm}((\Gr_G)^\varpi,\bk)$ which are either both even or both odd, there exis\-ts a canonical isomorphism of $\bk$-vector spaces
\[
% \begin{multline*}
%  \bigoplus_{n \in \Z} 
  \Hom_{\Db_{\IW_\ell}((\Gr_G)^\varpi,\bk)}^\bullet(\For_{\Gm}(\scE),\For_{\Gm}(\scE'))
  \cong \Hom_{\Smith_{\IW}((\Gr_G)^{\varpi},\bk)}(\mathsf{Q}(\scE), \mathsf{Q}(\scE')).
\]
% \end{multline*}
%where $\mathsf{Q} : \Db_{\IW_\ell,\Gm}((\Gr_{G})^{\varpi},\bk) \to \Smith_{\IW}((\Gr_G)^{\varpi},\bk)$ is the quotient functor. 
Moreover, these isomorphisms are compatible with composition in the obvious way.
\end{prop}

\begin{proof}
 %By definition, any parity complex in $\Db_{\IW_\ell,\Gm}((\Gr_G)^\varpi,\bk)$ is the direct sum of an even and an odd part. In view of Lemma~\ref{lem:Hom-vanishing-even-odd} and its counterpart in the category $\Db_{\IW_\ell}((\Gr_G)^\varpi,\bk)$ (see~\cite[Corollary~2.8]{jmw}), we can assume that $\scE$ and $\scE'$ are either both even, or both odd. In this case, 
 Recall from~\cite[Corollary~2.8]{jmw} that since $\scE,\scE'$ are either both even or both odd,
 the graded $\bk$-vector space $\Hom_{\Db_{\IW_\ell,\Gm}((\Gr_G)^\varpi,\bk)}^\bullet(\scE,\scE')$
is concentrated in even degrees. Using Lemma~\ref{lem:shift-Smith} (and its proof), we see that the functor $\mathsf{Q}$ induces a canonical morphism
\[
 \Hom^\bullet_{\Db_{\IW_\ell,\Gm}((\Gr_G)^\varpi,\bk)}(\scE,\scE') \to \Hom_{\Smith_{\IW}((\Gr_G)^{\varpi},\bk)}(\mathsf{Q}(\scE), \mathsf{Q}(\scE'))
\]
which factors through a morphism
%\begin{multline*}
\[
\bk' \otimes_{\coH^{\bullet}_{\Gm}(\pt,\bk)} \Hom_{\Db_{\IW_\ell,\Gm}((\Gr_G)^\varpi,\bk)}^\bullet(\scE,\scE') \\
 \to \Hom_{\Smith_{\IW}((\Gr_G)^{\varpi},\bk)}(\mathsf{Q}(\scE), \mathsf{Q}(\scE')),
\]
 %\end{multline*}
where $\bk'$ means $\bk$ seen as an $\coH^{\bullet}_{\Gm}(\pt,\bk)$-module where $x$ acts by multiplication by $1$ under the identification~\eqref{eqn:eq-cohom-pt}. Standard arguments based on Lemma~\ref{lem:parity-vanishing}, Lemma~\ref{lem:Hom-vanishing-even-odd} and the distinguished triangles in the ``recollement'' setting show that the latter morphism is an isomorphism. The desired isomorphism follows, in view of Lemma~\ref{lem:Hom-parity-For-fixed-points}.
\end{proof}

\begin{cor}
\label{cor:Q-indec}
 The functor $\mathsf{Q}$ sends indecomposable parity complexes to indecomposable parity complexes.
\end{cor}

\begin{proof}
If $\scE$ is an indecomposable parity complex in $\Db_{\IW_\ell,\Gm}((\Gr_G)^\varpi,\bk)$, then as explained above the complex $\For_{\Gm}(\scE)$ is indecomposable, so that the ring
\[
 \Hom_{\Db_{\IW_\ell}((\Gr_G)^\varpi,\bk)}(\For_{\Gm}(\scE),\For_{\Gm}(\scE)) 
\]
is local. Since a finite-dimensional graded ring whose degree-$0$ part is local is itself local (see~\cite[Theorem~3.1]{gg}), it follows that the ring
\[
 \bigoplus_{n \in \Z} \Hom_{\Db_{\IW_\ell}((\Gr_G)^\varpi,\bk)}(\For_{\Gm}(\scE),\For_{\Gm}(\scE)[n]) 
\]
is also local. In view of Proposition~\ref{prop:Hom-parity-Q}, this implies that $\mathsf{Q}(\scE)$ is indecomposable. 
 \end{proof}

\section{Applications in representation theory of reductive algebraic groups}
\label{sec:applications-red-gps}
%%%%%%%%%%%%%%%%%%%%%%%%%%%%%%

In this section, we finally use the constructions of Sections~\ref{sec:fixed-points-Gr}--\ref{sec:parity-Smith} to derive consequences on categories of representations of split connected reductive algebraic groups over $\bk$.

%--------------------------------------------------------
\subsection{The geometric Satake equivalence and its Iwahori--Whittaker variant}
%--------------------------------------------------------

%Recall the prime number $\ell$ chosen in~\S\ref{ss:fixed-points}. We now denote by $\bk$ a finite field of characteristic $\ell$. 
The \emph{Satake category} is the category
\[
\Perv_{\sph}(\Gr_G,\bk)
\]
of \'etale $L^+G$-equivariant perverse sheaves on $\Gr_G$. (By
definition, this category is the inductive limit of the categories
$\Perv_{\sph}(X,\bk)$ where $X$ runs over the closed finite unions of
$L^+G$-orbits in $\Gr_G$. Given such  an $X$, the category $\Perv_{\sph}(X,\bk)$ is defined as $\Perv_{H}(X,\bk)$, where $H$ is a smooth quotient of $L^+G$ of finite type such that the $L^+G$-action on $X$ factors through $H$, and such that the kernel of the surjection $L^+G \to H$ is contained in $\ker(\mathrm{ev}_0)$; the resulting category does not depend on the choice of $H$ up to canonical equivalence.)
The natural convolution product $\star$ on the equivariant derived category $\Db_{L^+G}(\Gr_G,\bk)$ restricts to an exact monoidal product on the category $\Perv_{\sph}(\Gr_G,\bk)$, see~\cite{mv}.

The classification of the simple objects in $\Perv_{\sph}(\Gr_G,\bk)$ is given by the general theory of perverse sheaves from~\cite{bbd}. Namely, in view of the description of the $L^+G$-orbits on $\Gr_G$ (see~\eqref{eqn:orbits-Gr-GO}) and since each of these orbits is simply connected, for any $\lambda \in \bX^\vee_+$ there exists a unique simple perverse sheaf $\IC^\lambda$ in $\Perv_{\sph}(\Gr_G,\bk)$ which is supported on $\overline{\Gr_G^\lambda}$ and whose restriction to $\Gr_G^\lambda$ is $\ubk_{\Gr_G^\lambda}[\dim(\Gr_G^\lambda)]$. Moreover, the assignment $\lambda \mapsto \IC^\lambda$ induces a bijection between $\bX^\vee_+$ and the set of isomorphism classes of simple objects in $\Perv_{\sph}(\Gr_G,\bk)$.

On the other hand, we denote by $G^\vee_\Z$ the unique split reductive group scheme over $\Z$ whose base change to $\C$ has root datum $(\bX^\vee,\bX, \fR^\vee,\fR)$, and then set
\[
G^\vee_\bk := \Spec(\bk) \times_{\Spec(\Z)} G^\vee_\Z.
\]
We will denote by $\Rep(G^\vee_\bk)$ the category of finite-dimensional algebraic representations of the group scheme $G^\vee_\bk$.

The following theorem is due (in this generality) to Mirkovi{\'c}--Vilonen~\cite{mv,mv-erratum}.

\begin{thm}
\label{thm:Satake}
There exists an equivalence of monoidal categories
\[
\Satake: (\Perv_{\sph}(\Gr_G,\bk), \star) \cong (\Rep(G^\vee_\bk), \otimes_\bk).
\]
\end{thm}

\begin{rmk}\phantomsection
\label{rmk:Satake}
\begin{enumerate}
 \item In fact, the proof of~\cite{mv} gives slightly more than what is stated in Theorem~\ref{thm:Satake}: the authors construct a \emph{canonical} $\bk$-group scheme out of the category $\Perv_{\sph}(\Gr_G,\bk)$, and then check that this group scheme is isomorphic to $G^\vee_\bk$.
 \item
 \label{it:Satake-equiv}
 In addition to the category $\Perv_{\sph}(\Gr_G,\bk)$, one can also consider the category $\Perv_{(L^+G)}(\Gr_G,\bk)$ of $\bk$-perverse sheaves on $\Gr_G$ whose restriction to each $\Gr_G^\lambda$ ($\lambda \in \bX^\vee_+$) has constant cohomology sheaves. We then have a canonical forgetful functor $\Perv_{\sph}(\Gr_G,\bk) \to \Perv_{(L^+G)}(\Gr_G,\bk)$, which by~\cite[Proposition~2.1]{mv} is an equivalence of categories.
\end{enumerate}
\end{rmk}

Once an equivalence as in Theorem~\ref{thm:Satake} is fixed, the constructions in~\cite{mv} provide a canonical embedding $T^\vee_\bk \hookrightarrow G^\vee_\bk$, where $T^\vee_\bk$ is the split $\bk$-torus which is Langlands dual to $T$ (i.e.~whose character lattice is $\bX^\vee$). We will denote by $B^\vee_\bk$ the Borel subgroup of $G^\vee_\bk$ containing (the image of) $T^\vee_\bk$ and whose roots are the negative coroots of $G$. For any $\lambda \in \bX^\vee_+$ we can then consider the ``induced representation''
\[
 \mathsf{N}(\lambda):=\mathsf{Ind}_{B^\vee_\bk}^{G^\vee_\bk}(\lambda).
\]
It is well known that $\mathsf{N}(\lambda)$ contains a unique simple submodule, denoted $\mathsf{L}(\lambda)$, and that the assignment $\lambda \mapsto \mathsf{L}(\lambda)$ induces a bijection between $\bX^\vee_+$ and the set of isomorphism classes of simple $G^\vee_\bk$-modules. It is well known also that for any $\lambda \in \bX^\vee_+$ we have
\begin{equation}
\label{eqn:Satake-simples}
 \Satake(\IC^\lambda) \cong \mathsf{L}(\lambda).
\end{equation}

Below we will use an alternative geometric realization of $\Rep(G^\vee_\bk)$, in terms of the Iwahori--Whittaker derived category of~\S\ref{ss:IW}, which was found in~\cite{bgmrr}.
The same construction as for the convolution product on $\Db_{L^+G}(\Gr_G,\bk)$ defines a right action of the latter monoidal category on $\Db_{\IW}(\Gr_G,\bk)$. The corresponding bifunctor will also be denoted $\star$.

We will assume that there exists (and fix) an element $\varsigma \in \bX^\vee$ such that $\langle \varsigma, \alpha \rangle = 1$ for any $\alpha \in \mathfrak{R}^{\mathrm{s}}$.
Then there exists no orbit in $\overline{\Gr_{G,\varsigma}^+} \smallsetminus \Gr_{G,\varsigma}^+$ which supports a nonzero Iwahori--Whittaker local system. Therefore, the canonical map $\Delta^\IW_\varsigma \to \nabla^\IW_\varsigma$ is an isomorphism, and this object is a simple perverse sheaf.

The following theorem is the main result of~\cite{bgmrr}.

\begin{thm}
\label{thm:bgmrr}
The functor sending $\mathscr{F}$ to $\Delta^\IW_\varsigma \star \scF$ induces an equivalence of abelian categories
\[
\Perv_{\sph}(\Gr_G,\bk) \simto \Perv_{\IW}(\Gr_G,\bk).
\]
\end{thm}

\begin{rmk}
Theorem~\ref{thm:bgmrr} can also be used to give an alternative proof of Lem\-ma~\ref{lem:IW-Gm-Perv}.
Namely, we see as in the proof of this lemma that our functor is fully faithful. If $\scF$ belongs to $\Perv_{\IW}(\Gr_G,\bk)$, by Theorem~\ref{thm:bgmrr} there exists an object $\scF'$ in $\Perv_{\sph}(\Gr_G,\bk)$ and an isomorphism
\[
\scF \cong \Delta^\IW_\varsigma \star \scF'.
\]
By~\cite[Proposition~2.2]{mv}, the perverse sheaf $\scF'$ is equivariant for the group $\Gm \ltimes L^+G$. Therefore the perverse sheaf $\Delta^\IW_\varsigma \star \scF'$ is the image of an object in $\Perv_{\IW,\Gm}(\Gr_G,\bk)$, and we deduce the same property for $\scF$.
\end{rmk}

%-----------------------------------------------------------------
\subsection{The linkage principle}
\label{ss:linkage}
%-----------------------------------------------------------------

We now come back to the setting where $G$ is an arbitrary connected reductive algebraic group over $\F$. Recall the actions $\cdot_\ell$ and $\square_\ell$  of $\Waff$ on $\bX^\vee$ defined in~\S\ref{ss:Waff}. The action which is relevant in Representation Theory is the ``dot action'' defined by
\[
(\st_\lambda v) \bullet_\ell \mu = v(\mu+\rho^\vee)-\rho^\vee + \ell \lambda
\]
for $\lambda \in \Z\mathfrak{R}^\vee$, $v \in \Wf$ and $\mu \in \bX^\vee$, where $\rho^\vee$ is the halfsum of the positive coroots. 
%In terms of the action considered in~\S\ref{ss:Waff}, we have
It is clear that for any $w \in \Waff$ and $\mu \in \bX^\vee$ we have
\begin{equation}
\label{eqn:comparison-actions}
w \bullet_\ell \mu = w \square_\ell (\mu+\rho^\vee)-\rho^\vee = w^* \cdot_\ell (\mu+\rho^\vee)-\rho^\vee,
\end{equation}
where $(\st_\lambda v)^*:= \st_{-\lambda} v$ for $\lambda \in \Z\fR^\vee$ and $v \in \Wf$.

The following statement is the first main result of this section.
%a geometric counterpart of the ``linkage principle'' in the representation theory of $G^\vee_\bk$, see e.g.~\cite[Chap.~6]{jantzen}.

\begin{thm}
\label{thm:linkage}
 For $\lambda,\mu \in \bX^\vee_+$, if $\Ext^n_{\Perv_{\sph}(\Gr_G,\bk)}(\IC^\lambda,\IC^\mu) \neq 0$ for some $n$, then $\Waff \bullet_\ell \lambda = \Waff \bullet_\ell \mu$.
\end{thm}

\begin{proof}
Note that if $\Ext^n_{\Perv_{\sph}(\Gr_G,\bk)}(\IC^\lambda,\IC^\mu) \neq 0$ for some $n$, then the orbits $\Gr_G^\lambda$ and $\Gr_G^\mu$ are contained in the same connected component of $\Gr_G$.
If $Z$ denotes the center of $G$, then the natural morphism $(\Gr_G)_{\mathrm{red}} \to (\Gr_{G/Z})_{\mathrm{red}}$ restricts, on each connected component $X$ of $(\Gr_G)_{\mathrm{red}}$, to a universal homeomorphism onto a connected component of $(\Gr_{G/Z})_{\mathrm{red}}$; see e.g.~\cite[Proof of Theorem~8.1]{brr} for references. By~\cite[\href{https://stacks.math.columbia.edu/tag/03SI}{Tag 03SI}]{stacks-project} and Remark~\ref{rmk:Satake}\eqref{it:Satake-equiv} we then have a fully faithful functor $\Perv_{\sph}(X,\bk) \to \Perv_{(G/Z)_{\mathscr{O}}}(\Gr_{G/Z},\bk)$, which reduces the proof to the case $G$ is semisimple of adjoint type, which we assume from now on.

In particular, under this assumption we can take $\varsigma=\rho^\vee$, and apply Theorem~\ref{thm:bgmrr}. This result implies that the simple objects in $\Perv_{\IW}(\Gr_G,\bk)$ are the perverse sheaves
 \[
  \IC^\IW_\lambda := \Delta^\IW_\varsigma \star \IC^{\lambda - \varsigma}
 \]
for $\lambda \in \bX^\vee_{++} = \varsigma + \bX^\vee_+$. In view of~\eqref{eqn:comparison-actions}, this shows that to prove the lemma it suffices to prove that for $\lambda,\mu \in \bX^\vee_{++}$, if $\Ext^n_{\Perv_{\IW}(\Gr_G,\bk)}(\IC^\IW_\lambda,\IC^\IW_\mu) \neq 0$ for some $n$, then $\Waff \cdot_\ell \lambda = \Waff \cdot_\ell \mu$.

In view of Theorem~\ref{thm:i!*-parities} (see also~\eqref{eqn:TIW-EIW} and~\eqref{eqn:i!*-parities}) and the decomposition of $(\Gr_G)^\varpi$ into its connected components (see Proposition~\ref{prop:fixed-points-Grass}), if $\lambda,\mu \in \bX^\vee_{++}$ satisfy $\Waff \cdot_\ell \lambda \neq \Waff \cdot_\ell \mu$ then we have
\[
 \Hom_{\Perv_{\IW}(\Gr_G,\bk)}(\scT^\IW_\lambda, \scT^\IW_\mu)=0.
\]
This implies that the category $\Tilt_{\IW}(\Gr_G,\bk)$ breaks canonically as a direct sum of full subcategories parametrized by orbits of $\Waff$ on $\bX^\vee$ (for the action $\cdot_\ell$), where the subcategory corresponding to a given orbit has as objects the direct sums of indecomposable tilting perverse sheaves labelled by elements in this orbit. Using~\eqref{eqn:equiv-Tilt-IW} we deduce a similar decomposition of $\Db \Perv_{\IW}(\Gr_G,\bk)$ as a direct sum of full subcategories parametrized by $\bX^\vee/(\Waff,\cdot_\ell)$. For any $\lambda \in \bX^\vee$ the object $\IC_\lambda^\IW$ is indecomposable, hence belongs to one of these subcategories; the existence of the nonzero maps
\[
 \IC^\IW_\lambda \twoheadleftarrow \Delta^\IW_\lambda \hookrightarrow \scT^\IW_\lambda
\]
shows that this subcategory is the one attached to $\Waff \cdot_\ell \lambda$, which finishes the proof.
\end{proof}

In view of Theorem~\ref{thm:Satake} and~\eqref{eqn:Satake-simples}, Theorem~\ref{thm:linkage} is equivalent to the statement that if $\Ext^n_{\Rep(G^\vee_\bk)}(\mathsf{L}(\lambda),\mathsf{L}(\mu)) \neq 0$ for some $n \in \Z$, then $\Waff \bullet_\ell \lambda = \Waff \bullet_\ell \mu$. This property is of course well known, and called the \emph{Linkage Principle}, see~\cite[\S II.6]{jantzen}. (This statement was first conjectured by Verma, and proved by Andersen in full generality, after partial results of Humphreys, Kac--Weisfeiler and Carter--Lusztig; see~\cite{andersen} for more details.)

%The linkage principle implies that the category $\Rep(G^\vee_\bk)$ decomposes into subcategories parametrized by $\bX^\vee / (\Waff, \bullet_\ell)$, or in other words by $A_\ell^\bullet$ (see~\S\ref{ss:Waff}). More precisely, for $\lambda \in A_\ell^\bullet$ we denote by $\Rep(G^\vee_\bk)_\lambda$ the Serre subcategory of $\Rep(G^\vee_\bk)$ generated by the simple objects $\mathsf{L}(\mu)$ with $\mu \in (\Waff \bullet_\ell \lambda) \cap \bX^\vee_+$; then we have
%\[
% \Rep(G^\vee_\bk)=\bigoplus_{\lambda \in A_\ell^\bullet} \Rep(G^\vee_\bk)_\lambda.
%\]

%In particular, it follows that the isomorphism classes of simple objects in $\Rep(G)_\lambda$ are in natural bijection with $\Waff^\lambda$.

%-----------------------------------------------------------------
\subsection{Antispherical $\ell$-canonical basis and parity complexes}
%-----------------------------------------------------------------

Let $\Haff$ be the Hecke algebra of $(\Waff,\Saff)$, and let $\Masph$ be its antispherical module, with ``standard'' basis $(N_w : w \in \fW)$ parametrized by the subset $\fW \subset \Waff$ of elements $w$ which are minimal in $\Wf w$. (Here we follow the conventions of~\cite{soergel}.) Let us consider
\[
\Fl_G^\circ := \Fl^{1,\circ}_{\mathbf{a}_1},
\]
%denote by $\Fl_G=\GK/\Iw$ the affine flag variety of $G$, and by $\Fl_G^\circ \subset \Fl_G$ 
the connected component of the base point in the affine flag variety associated with $LG$. 
We can then define, as for $\Gr_G$, the Iwahori--Whittaker derived category $\Db_\IW(\Fl_G^\circ,\bk)$, and its full subcategory $\Parity_{\IW}(\Fl_G^\circ,\bk)$ of parity complexes. The $\Iwu^+$-orbits on $\Fl_G^\circ$ are naturally parametrized by $\Waff$, and those which support a nonzero Iwahori--Whittaker local system are the ones corresponding to elements in $\fW$; we will denote by $\nabla^\IW_w$ and $\scE^\IW_w$ the costandard perverse sheaf and indecomposable parity complex attached to $w \in \fW$, respectively. (The latter object is characterized by properties similar to those considered in Proposition~\ref{prop:classification-parity}, taking into account the dimension of the corresponding orbit.)
We
then have a canonical isomorphism
\[
 \mathrm{ch}: [\Parity_{\IW}(\Fl_G^\circ,\bk)] \simto \Masph
\]
determined by
\[
 \mathrm{ch}([\scF]) = \sum_{\substack{w \in \fW \\ n \in \Z}} \dim_\bk \Hom_{\Db_\IW(\Fl_G^\circ,\bk)}(\scF,\nabla_w^\IW[n]) \cdot v^n N_w,
\]
where $[\Parity_{\IW}(\Fl_G^\circ,\bk)]$ is the split Grothendieck group of the additive category $\Parity_{\IW}(\Fl_G^\circ,\bk)$.

In terms of this isomorphism,
the \emph{$\ell$-canonical basis} $(\puN_w : w \in \fW)$ of $\Masph$ (see~\cite{rw,ar-survey}) can be characterized by
\begin{equation}
\label{eqn:pcan-basis}
 \puN_w := \mathrm{ch}(\scE^\IW_w).
\end{equation}
The associated \emph{$\ell$-Kazhdan--Lusztig polynomials} $(\ppn_{y,w} : y,w \in \fW)$ are characterized by the equality
\[
 \puN_w = \sum_{y \in \fW} \ppn_{y,w} \cdot N_y.
\]

\begin{rmk}
 It is easily seen that the computation of the $\ell$-canonical basis and $\ell$-Kazhdan--Lusztig polynomials can be reduced to the case $G$ is quasi-simple. In this case, the results of~\cite[Part~III]{rw} show that this basis coincides with the basis with the same name studied in~\cite{jw}, for the Coxeter system $(\Waff,\Saff)$ and the realization considered in~\cite[Remark~10.7.2(2)]{rw}. In particular, these data can be computed algorithmically using the procedure described in~\cite{jw, gjw}.
\end{rmk}

These considerations have been stated for the ind-variety $\Fl^{1,\circ}_{\mathbf{a}_1}$, but in practice we will rather use them for the isomorphic variety $\Fl^{\ell,\circ}_{\mathbf{a}_\ell}$, with respect to the action of $\Iwuell^+$.
More generally we can consider a facet $\mathbf{f} \subset \overline{\mathbf{a}}_\ell$, and the basic connected component in the associated partial affine flag variety $\Fl^{\ell,\circ}_{\mathbf{f}}$; see~\S\ref{ss:loop-flag}. Here again the Iwahori--Whittaker derived category (with respect to the action of $\Iwuell^+$) makes sense, and so does the notion of parity complexes. This derived category will be denoted $\Db_{\IW_\ell}(\Fl^{\ell,\circ}_{\mathbf{f}},\bk)$.

The indecomposable parity objects in $\Db_{\IW_\ell}(\Fl^{\ell,\circ}_{\mathbf{f}},\bk)$ can be described in terms of those on $\Fl^{\ell,\circ}_{\mathbf{a}_\ell}$ as follows.
As usual, the general theory of parity complexes ensures that there exists at most one indecomposable parity complex on $\Fl^{\ell,\circ}_{\mathbf{f}}$ associated with each $\Iwuell^+$-orbit which supports a nonzero Iwahori--Whittaker local system, and that each indecomposable parity complex is isomorphic (up to cohomological shift) to such an object. As explained in Remark~\ref{rmk:fixed-points-Gr}\eqref{it:parameters-orbits-fixed-points}, the $\Iwuell^+$-orbits on $\Fl^{\ell,\circ}_{\mathbf{f}}$ are parametrized in the natural way by the quotient $\Waff / \Waff^{\mathbf{f}}$, or in other words by the elements $w \in \Waff$ which are maximal in $w \Waff^{\mathbf{f}}$. If $w \in \Waff$ is maximal in $w \Waff^{\mathbf{f}}$, if the corresponding orbit supported a nonzero Iwahori--Whittaker local system, and if the associated indecomposable parity complex exists, we will denote it by $\scE^{\IW_\ell}_{\mathbf{f},w}$.

In the following statement, the morphism
\[
\Fl^{\ell,\circ}_{\mathbf{a}_\ell} \to \Fl^{\ell,\circ}_{\mathbf{f}}
\]
induced by~\eqref{eqn:inclusion-parahoric} will be denoted $\pi_{\mathbf{f}}$. We will also denote by $N_{\mathbf{f}}$ the length of the longest element in $\Waff^{\mathbf{f}}$.

\begin{lem}
\label{lem:pi*-IWparity}
If $w \in \Waff$ is maximal in $w \Waff^{\mathbf{f}}$, then the $\Iwuell^+$-orbit on $\Fl^{\ell,\circ}_{\mathbf{f}}$ associated with $w$ supports a nonzero Iwahori--Whittaker local system iff $w$ is minimal in $\Wf w$. Moreover, in this case the indecomposable Iwahori--Whittaker parity complex $\scE^{\IW_\ell}_{\mathbf{f},w}$ on $\Fl^{\ell,\circ}_{\mathbf{f}}$ exists, and we have
\[
 \pi_{\mathbf{f}}^*(\scE^{\IW_\ell}_{\mathbf{f},w})[N_{\mathbf{f}}] \cong \scE^{\IW_\ell}_{\mathbf{a}_\ell,w}.
\]
%its image under $\pi_{\mathbf{f}}^*[N_{\mathbf{f}}]$ coincides with the indecomposable Iwahori--Whittaker parity complex on $\Fl^{\ell,\circ}_{\mathbf{a}_\ell}$ associated with $w$.
\end{lem}

\begin{proof}
The proof is similar to that of its counterpart in the setting of Kac--Moody flag varieties considered in~\cite[Appendix~A]{acr}.
\end{proof}

\begin{rmk}
\label{rmk:identification-parity}
 Under the identification of $\Db_\IW(\Fl_G^\circ,\bk)$ with $\Db_{\IW_\ell}(\Fl^{\ell,\circ}_{\mathbf{a}_\ell},\bk)$, the object $\scE^\IW_w$ corresponds to $\scE^{\IW_\ell}_{\mathbf{a}_\ell,w}$.
\end{rmk}

%-----------------------------------------------------------------
\subsection{The tilting character formula}
%-----------------------------------------------------------------

Now we return to Representation Theory.
Recall that a $G^\vee_\bk$-module $M$ in $\Rep(G^\vee_\bk)$ is called \emph{tilting} if both $M$ and $M^*$ admit filtrations with subquotients of the form $\mathsf{N}(\lambda)$ with $\lambda \in \bX^\vee_+$. It is well known (see~\cite[\S II.E]{jantzen}) that the indecomposable tilting $G^\vee_\bk$-modules are classified by their highest weight (a dominant weight), and that any tilting module is a direct sum of indecomposable tilting modules. The indecomposable tilting module of highest weight $\lambda \in \bX^\vee_+$ will be denoted $\mathsf{T}(\lambda)$.

The same considerations as in~\S\ref{ss:Waff} show that a fundamental domain for the action of $\Waff$ on $\bX^\vee$ via $\bullet_\ell$ is given by the subset
\[
\overline{C}_\ell := \{\lambda \in \bX^\vee \mid \forall \alpha \in \mathfrak{R}^+, \, 0 \leq \langle \lambda+\rho^\vee,\alpha \rangle \leq \ell\}.
\]
Below we will need to describe the subset $(\Waff \bullet_\ell \lambda) \cap \bX^\vee_+$ more explicitly for $\lambda \in \overline{C}_\ell$. For this we set $I_\lambda := \{s \in \Saff \mid s \bullet_\ell \lambda = \lambda\}$, so that the stabilizer in $\Waff$ of $\lambda$ (for $\bullet_\ell$) is the parabolic subgroup $W_\lambda$ of $\Waff$ generated by $I_\lambda$. We set
\[
 W_\aff^{(\lambda)} := \{w \in \Waff \mid \text{$w$ is maximal in $wW_\lambda$ and minimal in $\Wf w$}\}.
\]
Then it is known that the assignment $w \mapsto w \bullet_\ell \lambda$ induces a bijection
\begin{equation*}
%\label{eqn:bijection-weights-block}
 W_\aff^{(\lambda)} \simto (\Waff \bullet_\ell \lambda) \cap \bX^\vee_+;
\end{equation*}
see~\cite[\S 10.1]{ar} for similar considerations.
In view of these comments,
%at the end of~\S\ref{ss:linkage} (see in particular the bijection~\eqref{eqn:bijection-weights-block}),
the following result gives an answer to the question of describing characters of \emph{all} indecomposable tilting $G^\vee_\bk$-modules, without any restriction on $\ell$.

\begin{thm}
\label{thm:characters-tilting}
 Let $\lambda \in \overline{C}_\ell$. Then for any $w \in \Waff^{(\lambda)}$ we have
 \[
  [\mathsf{T}(w \bullet_\ell \lambda)] = \sum_{y \in \Waff^{(\lambda)}} \ppn_{y,w}(1) \cdot [\mathsf{N}(y \bullet_\ell \lambda)]
 \]
 in the Grothendieck group of $\Rep(G^\vee_\bk)$.
\end{thm}

\begin{rmk}
\begin{enumerate}
\item
In the special case when $\ell$ is bigger than the Coxeter number $h$ of $G$, we have $W_0=\{1\}$. In this case, the formula in Theorem~\ref{thm:characters-tilting} was conjectured, and proved in the case of the group $G=\mathrm{GL}(n)$, in~\cite{rw}. A proof of this formula (again for $\ell>h$ and $\lambda=0$, but for a general reductive group) was later given in~\cite{amrw}. (Under the assumption that $\ell \geq h$, the formula of Theorem~\ref{thm:characters-tilting} in the special case $\lambda=0$ is sufficient to deduce the formula for all values of $\lambda$; see~\cite[Remark~1.4.4]{rw}. This property does \emph{not} hold for smaller values of $\ell$.) It was noticed in~\cite{rw} that a similar formula could be stated for any block of $\Rep(G^\vee_\bk)$, see~\cite[Conjecture~1.4.3]{rw}. Theorem~\ref{thm:characters-tilting} confirms this formula in full generality.
\item
Theorem~\ref{thm:characters-tilting} allows to make Theorem~\ref{thm:linkage} a bit more precise.\footnote{This remark was suggested to us by E.~Zabeth.} Namely, fix $\lambda \in \overline{C}_\ell$. Our formula implies that if $\mathsf{N}(y \bullet_\ell \lambda)$ occurs in a costandard filtration of $\mathsf{T}(w \bullet_\ell \lambda)$ then $y \leq w$ in the Bruhat order. In view of~\cite[\S 4.3.2]{coulembier}, it follows that if $\mathsf{L}(y \bullet_\ell \lambda)$ occurs as a composition factor of $\mathsf{N}(w \bullet_\ell \lambda)$ then $y \leq w$. (In particular, the block of $\Rep(G^\vee_\bk)$ associated with the orbit of $\lambda$ is a highest weight category with respect to the Bruhat order on $W_\aff^{(\lambda)}$.) Note that the condition that $y \leq w$ implies that $y \bullet_\ell A_0 \uparrow w \bullet_\ell A_0$ in the notation of~\cite[\S II.6.5]{jantzen} (where $A_0$ is the fundamental alcove) by the results discussed in~\cite{humphreys-note}, hence that $y \bullet_\ell \lambda \uparrow w \bullet_\ell \lambda$ by~\cite[Equation~II.6.5(2)]{jantzen}. Hence this claim implies the case $i=0$ (and $w=1$) of the St`s$H$rong Linkage Principle as stated in~\cite[Proposition~II.6.13]{jantzen}.
\end{enumerate}
\end{rmk}

Theorem~\ref{thm:characters-tilting} implies Theorem~\ref{thm:tilting} in the introduction thanks to Weyl's character formula (see~\cite[Corollary~II.5.11]{jantzen}).
The main step of the proof of this theorem is the following claim.

\begin{prop}
\label{prop:morph-tiltings}
Assume that there exists $\varsigma \in \bX^\vee$ such that $\langle \varsigma, \alpha \rangle = 1$ for any $\alpha \in \mathfrak{R}^{\mathrm{s}}$.
 For any $\lambda \in \overline{C}_\ell$ and $w,y \in \Waff^{(\lambda)}$, setting $\mu=\lambda+\varsigma$ we have an isomorphism
 \[
  \Hom_{\Rep(G^\vee_\bk)} \bigl( \mathsf{T}(w \bullet_\ell \lambda), \mathsf{T}(w' \bullet_\ell \lambda) \bigr) \cong \Hom^\bullet_{\Db_{\IW_\ell}(\Fl^{\ell,\circ}_{\mathbf{f}_\mu},\bk)}(\scE^{\IW_\ell}_{\mathbf{f}_\mu,w}, \scE^{\IW_\ell}_{\mathbf{f}_\mu,w'}).
 \]
\end{prop}

\begin{proof}
 By Theorem~\ref{thm:Satake}, Theorem~\ref{thm:bgmrr} and~\eqref{eqn:TIW-EIW} we have
 \begin{multline*}
  \Hom_{\Rep(G^\vee_\bk)} \bigl( \mathsf{T}(w \bullet_\ell \lambda), \mathsf{T}(w' \bullet_\ell \lambda) \bigr) \cong \Hom_{\Perv_\IW(\Gr_G,\bk)}(\scT^{\IW}_{w \square_\ell \mu}, \scT^{\IW}_{w' \square_\ell \mu}) \\
  \cong \Hom_{\Perv_\IW(\Gr_G,\bk)}(\scE^{\IW}_{w \square_\ell \mu}, \scE^{\IW}_{w' \square_\ell \mu}).
 \end{multline*}
% Denote now by $\delta_w$, resp.~$\delta_{w'}$, the parity of $\dim(\Gr^+_{G,w \square_\ell \mu},\bk)$, resp.~$\dim(\Gr^+_{G,w \square_\ell \mu},\bk)$, seen as an element in $\{0,1\}$. 
Using Theorem~\ref{thm:i!*-parities} and~\eqref{eqn:i!*-parities}, we deduce an isomorphism
 \[
  \Hom_{\Rep(G^\vee_\bk)} \bigl( \mathsf{T}(w \bullet_\ell \lambda), \mathsf{T}(w' \bullet_\ell \lambda) \bigr) \cong \Hom_{\Smith_{\IW}((\Gr_G)^\varpi,\bk)}(\scE^{\Smith, \mathsf{p}(\mu)}_{w \square_\ell \mu}, \scE^{\Smith, \mathsf{p}(\mu)}_{w' \square_\ell \mu}),
 \]
 where we use the notation of~\S\ref{ss:comparison-IW-Sm}.
 The two parity objects in the right-hand side are supported on $\Gr_{G,(\mu)}$, and the corresponding orbits in $\Fl^{\ell,\circ}_{\mathbf{f}_\mu}$ under the identification of Proposition~\ref{prop:fixed-points-Grass} are those corresponding to $w$ and $w'$; see Remark~\ref{rmk:fixed-points-Gr}\eqref{it:parameters-orbits-fixed-points}. In view of Proposition~\ref{prop:Hom-parity-Q} and Corollary~\ref{cor:Q-indec} we deduce an isomorphism
  \begin{multline*}
  \Hom_{\Rep(G^\vee_\bk)} \bigl( \mathsf{T}(w \bullet_\ell \lambda), \mathsf{T}(w' \bullet_\ell \lambda) \bigr) \cong \\
  \Hom^\bullet_{\Db_{\IW_\ell}(\Fl^{\ell,\circ}_{\mathbf{f}_\mu},\bk)}(\scE^{\IW_\ell}_{\mathbf{f}_\mu,w}[-\ell(w)+N_{\mathbf{f}_\mu}-\mathsf{p}(\mu)], \scE^{\IW_\ell}_{\mathbf{f}_\mu,w'}[-\ell(w')+N_{\mathbf{f}_\mu}-\mathsf{p}(\mu)]).
 \end{multline*}
 The desired isomorphism follows.
\end{proof}

\begin{proof}[Proof of Theorem~\ref{thm:characters-tilting}]
Recall that if we denote (as in the proof of Theorem~\ref{thm:linkage}) by $Z$ the center of $G$, then the group $(G/Z)^\vee_\bk$ identifies with the simply-connected cover of the derived subgroup of $G^\vee_\bk$. In view of the results recalled in~\cite[\S II.E.7]{jantzen}, this reduces the proof to the case $G$ is semisimple of adjoint type, which we will assume from now on. In this case we can take $\varsigma=\rho^\vee$ and apply Proposition~\ref{prop:morph-tiltings}.

 Standard arguments show that the formula will follow provided we prove that for any $w,w' \in \Waff^{(\lambda)}$ we have
 \begin{equation}
 \label{eqn:dim-Hom-Til}
  \dim_\bk \Hom_{\Rep(G^\vee_\bk)} \bigl( \mathsf{T}(w \bullet_\ell \lambda), \mathsf{T}(w' \bullet_\ell \lambda) \bigr) = \sum_{y \in \Waff^{(\lambda)}} \ppn_{y,w}(1) \cdot \ppn_{y,w'}(1).
 \end{equation}
 (In fact this formula will allow to prove that the multiplicity of $\mathsf{N}(y \bullet_\ell \lambda)$ in $\mathsf{T}(w \bullet_\ell \lambda)$ is $\ppn_{y,w}(1)$ by induction on $\ell(w)+\ell(y)$ using the arguments in~\cite[\S 6.2]{modrap1}.)
If $\mu=\lambda + \rho^\vee$, in view of Proposition~\ref{prop:morph-tiltings} this amounts to proving that
\[
 \dim_\bk \Hom^\bullet_{\Db_{\IW_\ell}(\Fl^{\ell,\circ}_{\mathbf{f}_\mu},\bk)}(\scE^{\IW_\ell}_{\mathbf{f}_\mu,w}, \scE^{\IW_\ell}_{\mathbf{f}_\mu,w'}) = \sum_{y \in \Waff^{(\lambda)}} \ppn_{y,w}(1) \cdot \ppn_{y,w'}(1).
\]
Now the dimension in the left-hand side
 can be expressed in terms of dimensions of (co)stalks of the involved parity complexes using~\cite[Proposition~2.6]{jmw}. By~\eqref{eqn:pcan-basis}, Lemma~\ref{lem:pi*-IWparity} and Remark~\ref{rmk:identification-parity} these dimensions are computed by $\ell$-Kazhdan--Lusztig polynomials. We deduce the desired formula using the fact that $W_\lambda = \Waff^{\mathbf{f}_\mu}$ (see~\eqref{eqn:comparison-actions}).
 %
% The fixed points $(\Gr_G)^\varpi$ are a union of partial affine flag varieties by Proposition~\ref{prop:fixed-points-Grass}, so that these dimensions can be computed using
%Lemma~\ref{lem:pi*-IWparity} and~\eqref{eqn:pcan-basis}. This provides the desired formula, in view of Remark~\ref{rmk:fixed-points-Gr}\eqref{it:parameters-orbits-fixed-points} and the fact that $W_\lambda = \Waff^{\mathbf{f}_\mu}$ (see~\eqref{eqn:comparison-actions}).
 %
% The $\Iw_\ell^+$-orbits $(\Gr_{G,w \square_\ell \mu}^+)^\varpi$ and $(\Gr_{G,w' \square_\ell \mu}^+)^\varpi$ are included in the component $\Gr_{G,(\mu)}$, which by Proposition~\ref{prop:fixed-points-Grass} identifies with $\Fl^{n,\circ}_{\mathbf{f}_\mu}$.
% %the connected component of the base point in the parahoric affine flag variety of $G$ associated with the subset $J_\mu=I_\lambda$. 
% Therefore its $\Iw_\ell^+$-orbits supporting nonzero Iwahori--Whittaker local systems are in a canonical bijection with $W_\aff^\lambda$ (see~\cite[\S A.B]{acr}), with $(\Gr_{G,w \cdot_\ell \mu}^+)^\varpi$ corresponding to $w$. If follows from~\cite[Lemma~A.5]{acr} that the dimensions of the stalks of the indecomposable parity complex associated with $w$ in $\Gr_{G,(\mu)}$ are computed by $\puN_w$. Using Proposition~\ref{prop:Hom-parity-Q} and~\cite[Proposition~2.6]{jmw} we deduce~\eqref{eqn:dim-Hom-TIW}, which finishes the proof.
\end{proof}

%%%%%%%%%%%%%%%%%%%%%%%%%%%%%%%%%%%%%%%%%%%%%%%%
%%%%%%%%%%%%%%%%%%%%%%%%%%%%%%%%%%%%%%%%%%%%%%%%

\end{document}